\documentclass[reqno, 11pt]{amsart}
\usepackage[margin=1in]{geometry}
\usepackage{bm}
\usepackage{amsrefs}
\usepackage[shortlabels]{enumitem}
\usepackage{hyperref}
\usepackage{cleveref}
\usepackage{tikz-cd}
\usepackage[textsize=scriptsize,backgroundcolor=white]{todonotes}
\usepackage{amssymb}

\DeclareMathOperator{\Id}{Id}
\DeclareMathOperator{\Real}{Re}
\DeclareMathOperator{\tr}{tr}
\DeclareMathOperator{\spn}{span}
\DeclareMathOperator{\ran}{ran}
\DeclareMathOperator{\supp}{supp}
\DeclareMathOperator{\divr}{div}
\DeclareMathOperator{\proj}{proj}

\newcommand{\vertiii}[1]{{\left\vert\kern-0.25ex\left\vert\kern-0.25ex\left\vert #1 \right\vert\kern-0.25ex\right\vert\kern-0.25ex\right\vert}}
\newcommand{\fk}{{F_w(\mathcal H_\tau)}}
\newcommand{\fkf}{{F_w^{\otimes}(\mathcal H_\tau)}}
\newcommand{\fkn}{{F_w(\mathcal H_{\tau,N})}}
\newcommand{\hfk}{\hat F_w(\mathcal H_\tau)}
\newcommand{\hfkn}{\hat F_w(\mathcal H_{\tau, N})}

\newtheorem{thm}{Theorem}
\newtheorem{cor}[thm]{Corollary}
\newtheorem{prop}[thm]{Proposition}
\newtheorem{lem}[thm]{Lemma}

\theoremstyle{definition}
\newtheorem{defn}[thm]{Definition}

\theoremstyle{remark}
\newtheorem{rk}[thm]{Remark}

\crefname{prty}{property}{properties}
\crefname{defn}{Definition}{Definitions}
\crefname{thm}{Theorem}{Theorems}
\crefname{cor}{Corollary}{Corollaries}
\crefname{lem}{Lemma}{Lemmas}
\crefname{prop}{Proposition}{Propositions}

\title{Second quantization for classical nonlinear dynamics}

\author[D.\ Giannakis et al.]{Dimitrios Giannakis}
\address{Department of Mathematics, Dartmouth College, Hanover, NH 03755, USA.}
\email{dimitrios.giannakis@dartmouth.edu}
\author[]{Mohammad Javad Latifi Jebelli}
\address{Department of Mathematics, Dartmouth College, Hanover, NH 03755, USA.}
\email{mohammad.javad.latifi.jebelli@dartmouth.edu}
\author[]{Michael Montgomery}
\address{Department of Mathematics, Dartmouth College, Hanover, NH 03755, USA.}
\email{michael.r.montgomery@dartmouth.edu}
\author[]{Philipp Pfeffer}
\address{Institut f\"ur Thermo- und Fluiddynamik, Technische Universit\"at Ilmenau, D-98684 Ilmenau, Germany.}
\email{philipp.pfeffer@tu-ilmenau.de}
\author[]{J\"org Schumacher}
\address{Institut f\"ur Thermo- und Fluiddynamik, Technische Universit\"at Ilmenau, D-98684 Ilmenau, Germany.}
\email{joerg.schumacher@tu-ilmenau.de}
\author[]{Joanna Slawinska}
\address{Department of Mathematics, Dartmouth College, Hanover, NH 03755, USA.}
\email{joanna.m.slawinska@dartmouth.edu}

\begin{document}

\begin{abstract}
    Using techniques from many-body quantum theory, we propose a framework for representing the evolution of observables of measure-preserving ergodic flows through infinite-dimensional rotation systems on tori. This approach is based on a class of weighted Fock spaces $\fk$ generated by a 1-parameter family of reproducing kernel Hilbert spaces $\mathcal H_\tau$, and endowed with commutative Banach algebra structure under the symmetric tensor product using a subconvolutive weight $w$. We describe the construction of the spaces $\fk$ and show that their Banach algebra spectra, $\sigma(\fk)$, decompose into a family of tori of potentially infinite dimension. Spectrally consistent unitary approximations $U^t_\tau$ of the Koopman operator acting on $\mathcal H_\tau$ are then lifted to rotation systems on these tori akin to the topological models of ergodic systems with pure point spectra in the Halmos--von Neumann theorem. Our scheme also employs a procedure for representing observables of the original system by polynomial functions on finite-dimensional tori in $\sigma(\fk)$ of arbitrarily large degree, with coefficients determined from pointwise products of eigenfunctions of $U^t_\tau$. This leads to models for the Koopman evolution of observables on $L^2$ built from tensor products of finite collections of approximate Koopman eigenfunctions. Numerically, the scheme is amenable to consistent data-driven implementation using kernel methods. We illustrate it with applications to Stepanoff flows on the 2-torus and the Lorenz~63 system. Connections with quantum computing are also discussed.
\end{abstract}

\maketitle

%###############################################################
\section{Introduction}
\label{sec:intro}

The pursuit of connections between classical and quantum dynamics dates back to the origins of quantum mechanics and modern ergodic theory in the first few decades of the 20th century. Seminal work of Koopman and von Neumann \cites{Koopman31,KoopmanVonNeumann32} simultaneously established the foundations of operator-theoretic representations of classical dynamics by means of composition operators (now widely known as Koopman operators) and put forward a description of the Liouville evolution of classical probability densities through quantum mechanical wavefunctions (so-called Koopman--von Neumann waves \cite{Mauro02}). The Koopman--von Neumann approach has found applications in areas including Hamiltonian mechanics \cites{WilkieBrumer97a,WilkieBrumer97b}, stochastic dynamics \cite{DellaRiciaWiener66}, and hybrid classical--quantum dynamics \cite{BondarEtAl19}.

A related problem that has received significant attention in recent years is simulation of classical dynamical systems by quantum systems \cites{BenentiEtAl01,Kacewicz06,LeytonOsborne08,BerryEtAl17,ElliottGu18,Joseph20,LloydEtAl20,KalevHen21,GiannakisEtAl22}. Example applications include viscous fluid flow \cites{BharadwajSreenivasan20,Gaitan20}, thermal convection \cites{PfefferEtAl22}, transport \cite{MezzacapoEtAl15}, wave propagation \cite{CostaEtAl19}, climate dynamics \cite{TenniePalmer23}, plasma dynamics \cites{EngelEtAl19,DodinStartsev21}, and energy science \cite{JosephEtAl23}. A primary motivation underpinning these efforts is the advent of quantum computing with its premise to deliver transformational advances in computing capabilities.

Yet another line of research has been on quantum-inspired techniques, i.e., methods based on the mathematical framework of quantum theory that are otherwise implemented classically. These methods leverage properties of spaces of operators associated with quantum systems that aid the design of structure-preserving approximation schemes (e.g., positivity-preserving projections) in applications such as data assimilation \cites{Giannakis19,FreemanEtAl23} and dynamical closure \cite{FreemanEtAl24}.

In broad terms, techniques for quantum simulation of classical dynamics are based on mappings of classical states and observables into states and observables of a quantum system, together with a corresponding mapping of the classical dynamical evolution maps into an evolution of quantum states. Here, a major challenge stems from the fact that classical systems of interest are typically nonlinear, whereas quantum dynamics proceeds by unitary (linear) transformations of the quantum state. A popular strategy to overcome this challenge is to represent polynomial nonlinearities in the equations of motion as multilinear maps that can be treated linearly in a tensor product Hilbert space using appropriate time-stepping schemes. Variants of this strategy have been employed in schemes for solving nonlinear ordinary differential equations (ODEs) \cites{LeytonOsborne08,LloydEtAl20} and certain classes of partial differential equations (PDEs) after space discretization \cites{Gaitan20,LiuEtAl21}.

Rather than attempting to directly simulate the (nonlinear) evolution the classical state, an alternative approach to state-space-centric schemes employs operator techniques from ergodic theory \cites{Baladi00,EisnerEtAl15} and the Koopman--von Neumann approach to simulate the evolution of observables and measures under the dynamics \cites{Joseph20,GiannakisEtAl22}. Indeed, for a measure-preserving, invertible transformation $\Phi \colon X \to X $ of a probability space $(X, \Sigma, \mu)$, the Koopman and transfer (Perron--Frobenius) operators on $L^2(\mu)$, $U f = f \circ \Phi$ and $ P f = f \circ \Phi^{-1}$, respectively, are unitary (and form a dual pair, $U^* = P$), and thus are natural candidates for approximation by quantum algorithms. By the Stone theorem on 1-parameter, strongly continuous unitary groups \cite{Stone32}, the Koopman group $ \{ U^t\colon L^2(\mu) \to L^2(\mu) \}_{t\in \mathbb R}$ induced by a measure-preserving flow $\Phi^t \colon X \to X$, $t \in \mathbb R$, is completely characterized by its generator---a skew-adjoint operator $V\colon D(V) \to L^2(\mu)$ defined on a dense domain $D(V) \subseteq L^2(\mu)$ as the $L^2(\mu)$-norm limit
\begin{equation}
    \label{eq:generator}
    V f = \lim_{t \to 0} (U^t f - f) / t,
\end{equation}
and giving the time-$t$ Koopman operator by exponentiation, $U^t = e^{t V}$. In this continuous-time setting, the self-adjoint operator $V/i$ plays an analogous role to a quantum mechanical Hamiltonian, and is therefore a natural object to study when making connections between classical and quantum systems.

\subsection{Our contributions}
\label{sec:contrib}

Focusing on measure-preserving, ergodic flows in continuous time, the primary challenge we seek to address is to build a quantum mechanical representation of classical dynamics that consistently approximates the Koopman evolution generated by $V$ from~\eqref{eq:generator} for systems with non-trivial continuous spectra. Indeed, a hallmark result in ergodic theory is that a measure-preserving system is weak-mixing if and only if the unitary Koopman group on $L^2(\mu)$ has a single, simple eigenvalue at 1 (with a constant corresponding eigenfunction), and no other elements in the point spectrum \cite{Halmos56}. In continuous time, this means that, aside from a simple eigenvalue at 0, the spectral measure of the generator is continuous. Thinking of weak-mixing as a signature of high dynamical complexity (``chaos'') in a measure-theoretic sense, consistent approximation of the continuous spectrum is thus relevant to quantum simulation of a broad range of nonlinear systems encountered in applications.

Our approach to address this problem consists of the following principal elements:
\begin{itemize}
    \item Spectral regularization techniques for Koopman operators \cites{DasEtAl21,GiannakisValva24,GiannakisValva24b} that approximate the generator $V$ by a 1-parameter family of skew-adjoint, diagonalizable operators $W_\tau: D(W_\tau) \to \mathcal H_\tau$, $\tau>0$, each acting on a reproducing kernel Hilbert space (RKHS) $\mathcal H_\tau$ of continuous functions on state space $X$. Here, $\tau$ is a regularization parameter and $W_\tau$ converges to $V$ as $\tau \to 0^+$ in an appropriate spectral sense.
    \item Constructions of so-called reproducing kernel Hilbert algebras (RKHAs) \cites{GiannakisMontgomery24,DasEtAl23,DasGiannakis23}, which are RKHSs with coalgebra structure with respect to the Hilbert space  tensor product and Banach algebra structure with respect to the pointwise product of functions.
    \item A class of weighted symmetric Fock spaces (described, to our knowledge, for the first time in this paper) with Banach algebra structure with respect to the tensor product. For our purposes, a key property of these spaces is that they are isomorphic to RKHAs of continuous functions on infinite-dimensional tori embedded in their spectra.
    \item Previously developed quantum algorithms for approximating the Koopman evolution of observables of rotation systems on tori \cite{GiannakisEtAl22}.
\end{itemize}
Building on these tools, we propose a provably consistent framework for quantum simulation of observables of continuous-time measure-preserving ergodic flows with arbitrary (i.e., pure-point, continuous, and mixed) spectral characteristics. The mathematical framework of many-body quantum  theory \cite{Lehmann04}, also known as second quantization, plays a central role in our constructions, particularly through the use of Fock space theory.

In essence, our scheme approximates the unitary Koopman evolution on $L^2(\mu)$ by a family of rotation systems on tori, $\mathbb T_{\tau,x}$, of potentially infinite dimension, parameterized by the regularization parameter $\tau$ and the initial condition $x \in X$. The tori $\mathbb T_{\tau,x}$ are realized as weak-$^*$ compact subsets of the spectrum $\sigma(\fk)$ of an abelian Banach algebra $\fk$ built as a weighted Fock space generated by the RKHS $\mathcal H_\tau$ that we employ for spectral approximation of the Koopman operator. Furthermore, the state space dynamics $\Phi^t\colon X \to X$ embeds naturally in $\sigma(\fk)$ by means of a feature map. As a result, we can interpret the torus rotation system obtained by regularization at the Koopman/transfer operator level as an approximate topological model of the unperturbed state space dynamics. This construction should be of independent interest in broader contexts than quantum mechanical representation of classical dynamics.

Each torus $\mathbb T_{\tau,x}$ can be approximated by a sequence of $d$-dimensional tori $\mathbb T_{\sigma,\tau,d,x} \subset \sigma(\fk)$ paramaterized by a smoothing parameter $\sigma>0$, with corresponding rotation systems generated by $d$ basic frequencies. We put forward a procedure for representing continuous classical observables as polynomials of arbitrarily large degree $m$ in the Fourier basis on $\mathbb T_{\sigma,\tau,d,x}$. Increasing $m$ allows to capture spectral information from large-dimensional spaces generated by tensor products of eigenfunctions of $W_\tau$. This enhances prediction skill and should also aid the theoretical efficiency of implementations of our scheme on quantum computing platforms.

In this paper, we describe the mathematical formulation of our second-quantization approach. Moreover, we demonstrate this framework by means of numerical experiments involving measure-preserving ergodic dynamical systems with aperiodic behavior: a Stepanoff flow on the 2-torus \cite{Oxtoby53} and the Lorenz 63 (L63) system \cite{Lorenz63} on $\mathbb R^3$. These experiments are performed on classical hardware using classical numerical methods and demonstrate improved prediction skill over conventional Koopman operator approximation methods utilizing the same number of eigenfunctions. A gate-based implementation of our algorithms on simulated on actual quantum hardware is beyond the scope of this paper, and will be addressed elsewhere.

\subsection{Plan of the paper}

In \cref{sec:prelims}, we introduce the class of dynamical systems under study and establish basic notational conventions. This is followed by a survey of relevant previous work in \cref{sec:previous_work}. In \cref{sec:overview}, we describe our second quantization approach, relegating certain constructions and proofs to \cref{sec:fock_space,sec:fock_emb}. Specifically, \cref{sec:fock_space} describes the construction and properties of the symmetric weighted Fock space $\fk$, and \cref{sec:fock_emb} our schemes for representing observables of the original dynamical system as vectors in $\fk$, including the polynomial functions on the tori $\mathbb T_{\sigma,\tau,d,x}$. \Cref{sec:data_driven_overview} contains an overview of the data-driven formulation of our scheme. In \cref{sec:experiments}, we present the results from our numerical experiments testing these methods. A discussion and our primary conclusions are included in \cref{sec:discussion}. \Cref{app:markov,app:approx_id,app:spectral_approx} collect auxiliary material on Markov semigroups employed in the construction of $\mathcal H_\tau$ and spectral approximation of Koopman operators, respectively.

%###############################################################
\section{Preliminaries and notation}
\label{sec:prelims}

\subsection{Dynamical system}
\label{sec:dynamical_system}

We consider a continuous flow $\Phi^t \colon X \to X$, $t \in \mathbb R$, on a metrizable space $X$ with an ergodic invariant Borel probability measure $\mu$ with compact support $X_\mu \subseteq X$. For $p \in [1, \infty]$, we let $ U^t \colon L^p(\mu) \to L^p(\mu)$ denote the associated time-$t$ Koopman operator that acts isometrically by composition with the flow, $ U^t f = f \circ \Phi^t$. As noted in \cref{sec:intro}, the Koopman group $ \{ U^t \}_{t \in \mathbb R} $ on $H := L^2(\mu)$ is strongly continuous and unitary, and thus has a skew-adjoint generator $V \colon D(V) \to H$ defined by~\eqref{eq:generator}. In addition, $U^t$ acts as a $^*$-isomorphism of the abelian von Neumann algebra $L^\infty(\mu)$ with respect to pointwise multiplication and complex conjugation,
\begin{displaymath}
    U^t(fg) = (U^t f) (U^t g), \quad U^t(f^*) = (U^t f)^*, \quad \forall f, g \in L^\infty(\mu).
\end{displaymath}

For every $p \in (1, \infty]$, $U^t$ has a predual, $P^t \colon L^q(\mu) \to L^q(\mu)$ for $\frac{1}{p} + \frac{1}{q} = 1$, defined by $P^t g = g \circ \Phi^{-t}$ and satisfying
\begin{displaymath}
    \langle U^t f, g \rangle = \langle f, P^t g \rangle, \quad \forall f \in L^p(\mu), \quad \forall g \in L^q(\mu),
\end{displaymath}
where $\langle \cdot, \cdot \rangle$ denotes the natural pairing between $L^p(\mu)$ and $L^q(\mu)$. The operators $P^t$ are known as transfer, or Perron--Frobenius, operators, and govern the evolution of densities with respect to the invariant measure $\mu$ under the dynamics. Specifically, given a finite Borel measure $\nu$ with density $g = \frac{d\nu}{d\mu} \in L^q(\mu)$, then $P^t g$ is the density $\frac{d\nu_t}{d\mu}$ of the pushforward measure $ \nu_t = \Phi^t_* \nu$ under the dynamics. The transfer operator $P^t \colon L^1(\mu) \to L^1(\mu)$ is an integral-preserving Markov operator with invariant measure $\mu$; i.e., (i) $P^t 1_X = 1_X$ for the constant function $1_X$ equal everywhere to 1; (ii) $P^t g \geq 0$ for $g \geq 0$; and (iii) $\int_X P^t g\, d\mu = \int_X g \, d\mu$ for all $g \in L^1(\mu)$. For further details on Koopman and transfer operators we refer the reader to \cites{Halmos56,Baladi00,EisnerEtAl15}.

Throughout the paper, we will assume that the support of $\mu$ is contained in a $C^1$ compact manifold $M \subseteq X$ that is forward-invariant under the flow, $\Phi^t(M) \subseteq M$ for all $t \geq 0$. We also assume that the restriction of $\Phi^t$ on $M$ is generated by a continuous vector field $\vec V \colon M \to TM$. This means that
\begin{equation}
    \label{eq:generator_c1}
    \iota \circ \vec V \cdot \nabla f = V \circ \iota f, \quad \forall f \in C^1(M),
\end{equation}
where $\iota \colon C(M) \to H$ is the map from continuous functions on $M$ to their corresponding equivalence classes in $H$. In this setting, the Koopman operator $f \mapsto f \circ \Phi^t$  acts as a $^*$-homomorphism on $C(M)$, viewed as an abelian $C^*$-algebra with respect to pointwise multiplication, complex conjugation and the supremum norm. We will continue to denote this map as $U^t \colon C(M) \to C(M)$ similarly to the Koopman operator on $L^p(\mu)$.

\subsection{Notation}
\label{sec:notation}

All vector spaces in this work will be over the complex numbers and all Hilbert spaces will be separable. For a normed space $\mathbb V$, $(\mathbb V)_R \subset \mathbb V$ will denote the closed ball of radius $R$ centered at the origin. Moreover, we will denote the space of bounded linear maps between $\mathbb V$ and a Banach space $\mathbb E$ as $B(\mathbb V, \mathbb E)$, and $\lVert A \rVert$ will be the operator norm of $A \in B(\mathbb V, \mathbb E)$. We will use the abbreviation $B(\mathbb E) \equiv B(\mathbb E, \mathbb E)$. Given an operator $A \colon D(A) \to \mathbb E$ defined on a subspace $D(A) \subseteq \mathbb E$ of a Banach space $\mathbb E$, $\rho(A)$, $\sigma(A)$, and $\sigma_p(A)$ will denote the resolvent set, spectrum, and point spectrum of $A$, respectively. Moreover, for $z \in \rho(A)$, $R(z, A) = (z-A)^{-1} \in B(\mathbb E)$ will denote the corresponding resolvent operator. An inner product on a vector space $\mathbb V$ will be denoted as $\langle \cdot, \cdot \rangle_{\mathbb V}$ and will be taken to be conjugate-linear in the first argument. If $\mathbb V$ is a closed subspace of a Hilbert space $\mathbb H$, $\proj_{\mathbb V} \in B(\mathbb H)$ will be the orthogonal projection with range $\mathbb V$.

If $\mathcal H$ is an RKHS of complex-valued functions on a set $\mathbb X$ with reproducing kernel $k \colon \mathbb X \times \mathbb X \to \mathbb C$, we will let $k_x = k(x, \cdot) \in \mathcal H$ be the kernel section at $x \in \mathbb X$ and $\delta_x \in \mathcal H^*$  the corresponding pointwise evaluation functional, $\delta_x = \langle k_x, \cdot\rangle_{\mathcal H}$. Moreover, given a subset $S \subseteq \mathbb X$, $\mathcal H(S)$ will be the closed subspace of $\mathcal H$ defined as $\mathcal H(S) = \overline{\spn\{k_x: x \in S\}}^{\lVert \cdot\rVert_{\mathcal H}}$. Note that $\mathcal H(S)$ is an RKHS of functions on $\mathbb X$, whose kernel section at $x \in \mathbb X$ is equal to the orthogonal projection of $k_x$ onto $\mathcal H(S)$. Moreover, $\mathcal H(S)$ is isomorphic as an RKHS to the restriction $\mathcal H\rvert_S$ of $\mathcal H$ onto $S$, equipped with $k\rvert_{S\times S}$ as the reproducing kernel.

Given a set $\mathbb X$, $1_S \colon X \to \mathbb R$ will denote the characteristic function of a subset $S \subseteq X$. If $\mathbb X$ is a topological space, $\mathcal B(\mathbb X)$ will denote the Borel $\sigma$-algebra of subsets of $\mathbb X$.

%###############################################################
\section{Background and previous work}
\label{sec:previous_work}

\subsection{Quantum simulation of systems with pure point spectra}
\label{sec:pure_point_spec}

In \cite{GiannakisEtAl22} a technique was developed for quantum simulation of a class of measure-preserving, ergodic flows with pure point spectra; i.e., measure-preserving flows for which the union of eigenspaces of the Koopman operator is dense in $L^p(\mu)$, $p \in [1, \infty)$. These systems are examples of highly structured dynamics that in many ways is antithetical to weak-mixing.

By the Halmos--von Neumann theorem \cite{HalmosVonNeumann42}, every pure-point-spectrum system is isomorphic in a measure-theoretic sense to a rotation system on a compact abelian group. In discrete time, this implies that the point spectrum $\sigma_p(U)$ of the Koopman operator is a subgroup of $\mathbb T^1$. In continuous time, the point spectrum $\sigma_p(V)$ of the generator is an additive subgroup of the imaginary line, giving, by the spectral mapping theorem, the point spectra of the Koopman operators as a group homomorphism, $\sigma_p(V) \ni i \omega \mapsto e^{i\omega t} \in \sigma_p(U^t)$.

The paper \cite{GiannakisEtAl22} studied the case of an ergodic torus rotation $R^t\colon \mathbb T^d \to \mathbb T^d$
\begin{displaymath}
    R^t(\theta) = \theta + \alpha \cdot \theta \mod 2 \pi, \quad t \in \mathbb R,
\end{displaymath}
where $\alpha = (\alpha_1, \ldots, \alpha_d) \in \mathbb R^d$ are rationally-independent frequency parameters. This system is a canonical representative in the measure-theoretic isomorphism class of pure-point-spectrum, continuous-time ergodic systems with spectra generated by $d$ basic frequencies. Specifically, we have $\sigma_p(V) = \{ i(j_1 \alpha_1 + \cdots + j_d \alpha_d): j_1, \ldots, j_d \in \mathbb Z \} $, so the point spectrum of the generator is isomorphic to the Pontryagin dual $\widehat{\mathbb T^d} \simeq \mathbb Z^d$ of the state space of the dynamics by rational independence of $\alpha_1, \ldots, \alpha_d$.

Here, the group structure of $\sigma_p(V)$ is a manifestation of the fact that $V$ obeys the Leibniz rule,
\begin{equation}
    \label{eq:leibniz}
    V(fg) = (V f) g + f (V g),
\end{equation}
for any $f, g \in D(V)$ for which the left-hand-side and right-hand side of the above equation are well-defined. In fact, \cite{TerElstLemanczyk17} showed that satisfying~\eqref{eq:leibniz} on the algebra $L^\infty(\mu) \cap D(V)$ is a necessary and sufficient condition for a skew-adjoint operator $V\colon D(V) \to L^2(\mu)$ to be the generator of a unitary Koopman group (i.e., a one-parameter unitary group of composition operators), so one can consider the Leibniz rule as a fundamental structural property of the generators of continuous-time classical dynamical systems.

\subsubsection{Choice of Hilbert space}

Using results from harmonic analysis \cites{FeichtingerEtAl07,DasGiannakis23,DasEtAl23}, the approach of \cite{GiannakisEtAl22} was to build a quantum system on a function space $\mathfrak A \subset C(\mathbb T^d)$ that is simultaneously a reproducing kernel Hilbert space (RKHS) and a Banach $^*$-algebra with respect to pointwise function multiplication and complex conjugation. On this space $\mathfrak A$, a rotation system $R^t \colon \mathbb T^d \to \mathbb T^d$ induces unitary Koopman operators $\hat U^t \colon \mathfrak A \to \mathfrak A$ analogously to $U^t$ on $L^2(\mu)$, and the spectrum of the skew-adjoint generator $\hat W\colon D(\hat W) \to \mathfrak A$ with $D(\hat W) \subseteq \mathfrak A$ is again $\sigma_p(\hat W) \simeq \mathbb Z^d$. Moreover, every classical observable $f \in \mathfrak A$ has an associated quantum observable $M_f \in B(\mathfrak A)$ that acts as a multiplication operator, $M_f g = fg$.

The RKHS $\mathfrak A$ has a translation-invariant reproducing kernel $k\colon \mathbb T^d \times \mathbb T^d \to \mathbb R_{>0}$ obtained from the inverse Fourier image of an integrable, strictly positive, symmetric, subconvolutive function  on the dual group $\widehat{\mathbb T^d} \simeq \mathbb Z^d$; i.e., a function $\lambda\colon \mathbb Z^d \to \mathbb R_{>0}$ satisfying
\begin{displaymath}
    \lambda \in L^1(\mathbb Z^d), \quad \lambda(j) > 0, \quad \lambda(-j) = \lambda(j), \quad \lambda * \lambda(j) \leq C \lambda(j),
\end{displaymath}
where $L^1(\mathbb Z^d)$ is defined with respect to the counting measure (equivalently, a Haar measure on $\widehat{\mathbb T^d}$). The generator $\hat W$ admits the diagonalization
\begin{displaymath}
    \hat W \psi_j = i \omega_j \psi_j, \quad \psi_j = \sqrt{\lambda(j)} \phi_j, \quad \omega_j = j_1 \alpha_1 + \cdots + j_d \alpha_d,
\end{displaymath}
where $\phi_j\colon \mathbb T^d \to \mathbb C$ are characters (Fourier functions) in the dual group $\widehat{\mathbb T^d}$, $\phi_j(\theta) = e^{i j \dot \theta}$, indexed by $j = (j_1, \ldots, j_d) \in \mathbb Z^d$. Moreover, the eigenfunctions $\psi_j$ form an orthonormal basis of $\mathfrak A$, and $\omega_j \in \mathbb R$ are corresponding eigenfrequencies.

The simultaneous RKHA and Banach algebra structure of $\mathfrak A$ provides several useful properties for embedding classical dynamics into quantum dynamics and for building associated quantum algorithms, which we outline below. For detailed treatments of RKHSs we refer the reader to \cites{SteinwartChristmann08,PaulsenRaghupathi16}. Further details on weight functions in harmonic analysis and associated multiplication/convolution algebras can be found in \cites{Feichtinger79,Grochenig07,Kaniuth09}.

\subsubsection{Quantum embedding of classical dynamics}
\label{sec:quantum_embedding_pure_point_spec}

Recall that for the space $B(\mathbb H)$ of bounded operators on a Hilbert space $\mathbb H$, viewed as a von Neumann algebra with respect to operator composition and adjoint, the space of normal states can be identified with the set of density operators on $\mathbb H$, i.e., the set of positive, trace-class operators $\rho \colon \mathbb H \to \mathbb H$ of unit trace \cite{Takesaki01}. We denote this set as $Q(\mathbb H)$ and interpret it as the set of quantum states on $\mathbb H$. Every $\rho \in Q(\mathbb H)$ induces a state $\mathbb E_\rho \colon B(\mathbb H) \to \mathbb C$ such that $\mathbb E_\rho a = \tr(\rho a)$. The number $\mathbb E_\rho a$ corresponds to the expectation of quantum observable $a$ with respect to the quantum state $\rho$.

When $\mathbb H = \mathcal H$ is an RKHS of functions on a set $X$, we have a map $\bm \varphi\colon X \to Q(\mathcal H)$ of classical states into quantum states induced by the feature map,
\begin{equation}
    \label{eq:feature_map}
    \varphi\colon X \to \mathcal H, \quad \varphi(x) = k_x, \quad k_x = k(x, \cdot),
\end{equation}
where $k\colon X \times X \to \mathbb C$ is the reproducing kernel of $\mathcal H$. Specifically, defining $\bm \varphi(x) \equiv \rho_x = \langle \xi_x, \cdot\rangle_{\mathcal H} \xi_x$ with $\xi_x = \varphi(x) / \sqrt{k(x, x)}$, one readily verifies that $\rho_x$ is a rank-1 quantum state that projects along the unit vector $\xi_x \in \mathcal H$. If the feature map $\varphi$ is injective (which is the case for many RKHS examples; e.g., \cite{SriperumbudurEtAl11}), then so is $\bm \varphi$, so we have an embedding on classical states into quantum states. If, in addition, $\mathcal H = \mathfrak A$ is Banach algebra with respect to pointwise function multiplication, there is a faithful representation $\pi\colon \mathfrak A \to B(\mathfrak A)$ of classical observables in $\mathfrak A$ as multiplication operators in $B(\mathfrak A)$,
\begin{equation}
    \label{eq:mult_rep}
    (\pi f) g = fg, \quad \forall g \in \mathfrak A.
\end{equation}

Let $\delta_x \colon \mathfrak A \to \mathbb C$ be the pointwise evaluation functional at $x \in X$, where $\delta_x = \langle k_x, \cdot \rangle_{\mathfrak A}$ since $\mathfrak A$ is an RKHS \cite{PaulsenRaghupathi16}. One readily verifies that the quantum embeddings of states and observables through $\bm \varphi$ and $\pi$, respectively, are consistent with pointwise evaluation,
\begin{displaymath}
    \mathbb E_{\bm \varphi(x)} (\pi f) = \delta_x f = f(x), \quad \forall f \in \mathfrak A, \quad \forall x \in \mathbb X.
\end{displaymath}
In the case of the rotation system on $X = \mathbb T^d$, these embeddings are also compatible with dynamical evolution. Defining the adjoint actions $\bm R^t \colon Q(\mathfrak A) \to Q(\mathfrak A)$ and $\bm U^t \colon B(\mathfrak A) \to B(\mathfrak A)$ of the Koopman group on quantum states and observables as $\bm R^t(\rho) = \hat U^{t*} \rho \hat U^t$ and $\bm U^t A = \hat U^t A \hat U^{t*}$, respectively, we have $ \bm \varphi \circ R^t = \bm R^t \circ \bm \varphi$ and $\pi \circ \hat U^t = \bm U^t \circ \pi$; that is, the quantum feature map $\bm \varphi$ and multiplier representation $\pi$ intertwine the classical and quantum evolutions of states and observables, respectively. Equivalently, the following two diagrams commute:
\begin{displaymath}
    \begin{tikzcd}
        \mathbb T^d \ar[r, "R^t"] \ar[d, "\bm \varphi", swap] & \mathbb T^d \ar[d, "\bm \varphi"] \\
        Q(\mathfrak A) \ar[r, "\bm R^t"] & Q(\mathfrak A)
    \end{tikzcd},
    \quad
    \begin{tikzcd}
        \mathfrak A \ar[r,"\hat U^t"] \ar[d,"\pi",swap] & \mathfrak A \ar[d,"\pi"] \\
        B(\mathfrak A) \ar[r,"\bm U^t"] & B(\mathfrak A)
    \end{tikzcd}
    .
\end{displaymath}

\subsection{Spectral approximation for systems with continuous spectra}
\label{sec:spectral_approx}

Let us recall the following fundamental results on unitary Koopman groups associated with measure-preserving ergodic flows (e.g., \cite{Halmos56}):

\begin{thm}
    \label{thm:ergodic_splitting}
    With the notation and assumptions of \cref{sec:dynamical_system}, the Hilbert space $H$ admits a $U^t$-invariant orthogonal decomposition $H = H_p \oplus H_c$, where the subspaces $H_p$ and $H_c$ have the following properties.
    \begin{enumerate}[(i)]
        \item $H_p$ admits an orthonormal basis $ \{ \xi_j \}$ consisting of eigenvectors of the generator,
            \begin{displaymath}
                V \xi_j = i \omega_j \xi_j, \quad \omega_j \in \mathbb R,
            \end{displaymath}
            where all eigenvalues $i \omega_j$ are simple and can be indexed by integers $j$ such that $\omega_{-j} = -\omega_j$ (in particular, $\omega_0 = 0$). Moreover, the set $ \{ i \omega_j \}$ constitutes the point spectrum $\sigma_p(V)$, and the corresponding eigenvectors can be chosen such that $\xi_{-j} = \xi_j^*$ with $\xi_0 = 1_X$.
        \item Elements of $H_c$ exhibit the following form of decay of correlations (also known as weak-mixing behavior),
            \begin{displaymath}
                \lim_{T\to \infty}\frac{1}{T} \int_0^T \lvert C_{fg}(t)\rvert \, dt = 0, \quad \forall f \in H, \quad \forall g \in H_c,
            \end{displaymath}
            where $C_{fg}(t) = \langle f, U^tg\rangle_H$.
    \end{enumerate}
\end{thm}

By \cref{thm:ergodic_splitting}(i), every element $f \in H_p$ evolves as an observable of a pure point spectrum system,
\begin{equation}
    \label{eq:pure_point_spec_evo}
    U^t f = \sum_j e^{i\omega_jt} \langle \xi_j, f\rangle_H \xi_j,
\end{equation}
and is thus amenable to quantum simulation via the techniques described in \cref{sec:pure_point_spec}. On the other hand, it is not obvious how to apply these methods to observables in $H_c$, for this space does not admit an orthonormal basis consisting of generator eigenfunctions. To overcome this obstacle, we will approximate the generator $V$ by a family of skew-adjoint operators that are diagonalizable on the entire Hilbert space $H$, and whose eigenfunctions behave as approximate Koopman eigenfunctions that we will use for building rotation systems approximating the dynamical flow $\Phi^t$.

\subsubsection{Spectrally accurate approximations}

The development of analytical and computational techniques for spectral approximation of Koopman and transfer operators of measure-preserving systems has been a highly active research area in recent years; e.g., \cites{OttoRowley21,Colbrook24} and references therein. In this work, we employ a variant of the techniques developed in \cites{DasEtAl21,GiannakisValva24,GiannakisValva24b} that yield approximations of the generator that are spectrally accurate in the sense of strong convergence of resolvents.

\begin{defn}
    \label{def:src}A family of skew-adjoint operators $A_\tau \colon D(A_\tau) \to \mathbb H$, $\tau > 0$, on a Hilbert space $\mathbb H$ is said to converge to a skew-adjoint operator $A\colon D(A) \to \mathbb H$ in strong resolvent sense if for some (and thus, every) $z \in \mathbb C \setminus i \mathbb R$ the resolvents $R(z, A_\tau)$ converge strongly to $R(z, A)$; that is, $\lim_{\tau\to 0^+} \lVert (R(z, A_\tau) - R(z, A)) f \rVert_{\mathbb H} = 0$ for every $f\in \mathbb H$.
\end{defn}

It can be shown that strong resolvent convergence $A_\tau \to A$ is equivalent to strong convergence of the corresponding unitaries, $e^{t A_\tau} \to e^{t A}$ for every $t \in \mathbb R$ (also known as strong dynamical convergence); e.g., \cite{Oliveira09}*{Proposition~10.1.8}. For our purposes, this implies that if a family of skew-adjoint operators converges to the Koopman generator $V$ in strong resolvent sense, the unitary evolution groups generated by these operators consistently approximate the Koopman operators $U^t = e^{t V}$ generated by $V$.

Strong resolvent convergence and strong dynamical convergence imply the following form of spectral convergence; \cite{DasEtAl21}*{Proposition~13}.

\begin{thm}
    With the notation of \cref{def:src}, let $\tilde E\colon \mathcal B(i \mathbb R) \to B(\mathbb H)$ and $\tilde E_\tau \colon \mathcal B(i \mathbb R) \to B(\mathbb H)$ be the spectral measures of $A$ and $A_\tau$, respectively, i.e., $A = \int_{i \mathbb R} \lambda\, d\tilde E(\lambda)$ and $A_\tau = \int_{i \mathbb R} \lambda\, d\tilde E_\tau(\lambda)$. Then, the following hold under strong resolvent convergence of $A_\tau$ to $A$.
\begin{enumerate}[(i)]
        \item For every element $\lambda \in \sigma(A)$ of the spectrum of $A$, there exists a sequence $\tau_1, \tau_2, \ldots \searrow 0$ and elements $\lambda_n \in \sigma(A_{\tau_n})$ of the spectra of $A_{\tau_n}$  such that $\lim_{n\to \infty} \lambda_n = \lambda$.
        \item For every bounded continuous function $h\colon i \mathbb R \to \mathbb C$, as $\tau\to 0^+$ the operators $h(A_\tau) = \int_{i \mathbb R} h(\lambda)\,d\tilde E_\tau(\lambda)$ converge strongly to $h(A) = \int_{i \mathbb R} h(\lambda) \,d\tilde E(\lambda)$.
        \item For every bounded Borel-measurable set $\Theta \in \mathcal B(i \mathbb R)$ such that $\tilde E(\partial \Theta) = 0$ (i.e., the boundary of $\Theta$ does not contain eigenvalues of $A_\tau$), as $\tau\to 0^+$ the projections $\tilde E_\tau(\Theta)$ converge strongly to $\tilde E(\Theta)$.
    \end{enumerate}
    \label{thm:spec-conv}
\end{thm}

\subsubsection{Markov smoothing operators}
\label{sec:markov_operators}

In broad terms, the schemes of \cites{DasEtAl21,GiannakisValva24,GiannakisValva24b} regularize the generator or its resolvent by composing it with Markov smoothing operators with associated RKHSs of continuous, complex-valued functions on the state space $X$.

These RKHSs, $\mathcal H_\tau$, are defined for each $\tau>0$ to have a bounded, continuous, strictly positive-definite kernel $k_\tau \colon X \times X \to \mathbb R_{>0}$, built such that $k_\tau\rvert_{M \times M}$ is $C^1$. By standard results on RKHS theory (e.g., \cites{SteinwartChristmann08,PaulsenRaghupathi16}), $K_\tau \colon H \to \mathcal H_\tau$ with
\begin{displaymath}
    K_\tau f = \int_X k_\tau(\cdot, x) f(x) \, d\mu(x)
\end{displaymath}
is a well-defined compact integral operator. Moreover, we have that (i) the range of $K_\tau$ is a dense subspace of $\mathcal H_\tau(X_\mu)$; (ii) $K_\tau^*: \mathcal H_\tau \to H$ implements the inclusion map (that is, $K_\tau^*\rvert_{\mathcal H_\tau(M)} = \iota \rvert_{\mathcal H_\tau(M)}$); and (iii) $\mathcal H_\tau(M)$ is a subspace of $C^1(M)$. In particular, the action of the generator on elements of $\mathcal H_\tau$ can be evaluated via~\eqref{eq:generator_c1}.

Defining $G_\tau = K_\tau^* K_\tau$ for $\tau>0$ and $G_0$ as the identity operator on $H$, we require that $\{ G_\tau \}_{\tau\geq 0}$ is a strongly-continuous semigroup of strictly positive, Markov operators. This means:
\begin{enumerate}[label=(K\arabic*)]
    \item \label[prty]{prty:k1} $\langle f, G_\tau f\rangle_H > 0$ whenever $f \neq 0$ (i.e., $G_\tau$ is strictly positive as an operator in $B(H)$).
    \item \label[prty]{prty:k2} $G_\tau 1_X = 1_X$ and $G_\tau f \geq 0$ whenever $g \geq 0$ (i.e., $G_\tau$ is a Markov operator).
    \item \label[prty]{prty:k3} $G_\tau \circ G_{\tau'} = G_{\tau+\tau'}$ for every $\tau,\tau' \geq 0$ (i.e., $ \{ G_\tau \}_{\tau\geq 0}$ is a semigroup).
    \item \label[prty]{prty:k4} $\lim_{\tau \to 0^+} G_\tau f = f$ for every $f \in H$ (i.e., $ \{ G_\tau \}_{\tau\geq 0}$ is strongly continuous).
\end{enumerate}
It can also be shown that $K_\tau$ admits the polar decomposition
\begin{equation}
    \label{eq:polar_decomp}
    K_\tau = T_\tau G_{\tau/2},
\end{equation}
where $T_\tau: H \to \mathcal H_\tau$ is an isometry with range $\mathcal H_\tau(X_\mu)$. In addition, we have
\begin{equation}
    \label{eq:ran_k_tau}
    \ran K_{\tau/2} = \mathcal H_\tau(X_\mu),
\end{equation}
so we can realize the subspaces $\mathcal H_{\tau}(X_\mu) \subseteq \mathcal H_{\tau} $ by smoothing elements of $H$ by the kernel integral operators $\mathcal H_\tau$.

Possible ways of constructing kernels $k_\tau$ satisfying \crefrange{prty:k1}{prty:k4} include normalization of Gaussian kernels on $X=\mathbb R^n$ and Fourier transform of positive functions on the dual group $\widehat{X}$ when $X$ is a compact abelian group and the invariant measure $\mu$ is a Haar measure (cf.\ $\lambda$ from \cref{sec:pure_point_spec}). In \cref{app:markov} we give an outline of the Markov normalization approach used in the numerical experiments of \cref{sec:experiments}. Examples of Markov kernel constructions on compact abelian groups via weight functions can be found, e.g., in \cites{DasGiannakis23,GiannakisMontgomery24,GiannakisEtAl24}.

\subsubsection{Diagonalizable approximations of the generator}

Using the smoothing operators $G_\tau$, the papers \cites{DasEtAl21,GiannakisValva24} builds 1-parameter families of densely-defined operators $V_\tau \colon D(V_\tau) \to H$ with $z,\tau >0$ and $D(V_\tau) \subseteq H$ with the following properties.

\begin{enumerate}[label=(V\arabic*)]
    \item \label[prty]{prty:v1} $V_\tau$ is skew-adjoint.
    \item \label[prty]{prty:v2} $V_\tau$ is compact \cite{DasEtAl21} or has compact resolvent \cite{GiannakisValva24}.
    \item \label[prty]{prty:v3} $V_\tau$ is real, $(V_\tau f)^* = V_\tau(f^*)$ for all $f\in D(V)$.
    \item \label[prty]{prty:v4} $V_\tau$ annihilates constant functions, $V_\tau \bm 1 = 0$.
    \item \label[prty]{prty:v5} $V_\tau$ converges to $V$ in strong resolvent sense as $\tau \to 0^+$.
\end{enumerate}

In \cite{GiannakisValva24b} a related approach is developed whereby $V$ is approximated by a 2-parameter family of operators $V_{z,\tau}$ satisfying \crefrange{prty:v1}{prty:v4} and converging to $V$ in strong resolvent sense in the iterated limits $z \to 0^+$ after $\tau \to 0^+$. A key feature of this approach is that it is ``physics-informed'' in the sense of using known equations of motion to evaluate the action of the generator by means~\eqref{eq:generator_c1}. We review this method in \cref{app:spectral_approx} as it will be used in our numerical examples.

Here, as a concrete example we mention the (formulaically simpler) approach of \cite{DasEtAl21}, who define the compact approximations
\begin{displaymath}
    V_\tau = G_{\tau/2} V G_{\tau/2}.
\end{displaymath}
These operators are manifestly skew-adjoint. Their compactness follows by compactness of $G_{\tau/2}$, boundedness of $\vec V \cdot \nabla$ as an operator from $C^1(M)$ to $C(M)$, and the relation
\begin{displaymath}
    V G_{\tau/2} = \iota \circ (\vec V \cdot \nabla) \circ K_{\tau/2}.
\end{displaymath}
Unless stated otherwise, in what follows we will use the abbreviated notation $V_\tau$ to denote any of the operators $V_\tau$ or $V_{z,\tau}$.

\subsubsection{Approximate Koopman eigenfunctions}

For every $\tau>0$, $V_\tau$ is unitarily equivalent to a skew-adjoint operator $W_\tau \colon D(W_\tau) \to \mathcal H_\tau$ with domain $D(W_\tau) = T_\tau(D(V_\tau)) \cup \ker T_\tau^*$ and range contained in $\mathcal H_\tau(X_\mu)$, defined as
\begin{displaymath}
    W_\tau = T_\tau V_\tau T_\tau^*.
\end{displaymath}
Both $V_\tau$ and $W_\tau$ are unitarily diagonalizable, and there exist orthonormal bases $ \{ \xi_{j,\tau} \}$ and $ \{ \zeta_{j,\tau} \}$ of $H$ and $\mathcal H_\tau(X_\mu)$, respectively, consisting of their eigenvectors,
\begin{equation}
    \label{eq:approx_koopman_eigs}
    V_\tau \xi_{j,\tau} = i \omega_{j,\tau} \xi_{j,\tau}, \quad W_\tau \zeta_{j,\tau} = i \omega_{j,\tau} \zeta_{j,\tau}, \quad \zeta_{j,\tau} = T_\tau \xi_{j,\tau}
\end{equation}
where $\omega_{j,\tau} \in \mathbb R$ are corresponding eigenfrequencies. By analogy of eigenvectors/eigenfrequencies of pure-point-spectrum systems in \cref{thm:ergodic_splitting}(i), the $\omega_{j,\tau}, \xi_{j,\tau}, \zeta_{j,\tau}$ can be indexed using indices $j \in \mathbb Z$ such that $\omega_{-j,\tau} = - \omega_{j,\tau}$, with the the corresponding eigenvectors chosen such that $\xi_{-j,\tau} = \xi_{j,\tau}^*$, $\zeta_{j,\tau} = \zeta_{j,\tau}^*$, and $\xi_{0,\tau} = \zeta_{j,\tau} = 1_X$. Moreover, the eigenvalue $i\omega_{0,\tau} = 0$ is simple. Following \cite{GiannakisValva24b}, we order the eigenfunctions $\zeta_{0,\tau}, \zeta_{\pm 1, \tau}, \zeta_{\pm 2,\tau} $ in increasing order of a Dirichlet energy functional; see \cref{app:spectral_approx} for more details.

On the basis of \cref{prty:v5} and \cref{thm:spec-conv}, we interpret $\xi_{j,\tau}$/$\zeta_{j,\tau}$ and $\omega_{j,\tau}$ as approximate Koopman eigenfunctions and eigenfrequencies, respectively. The corresponding unitaries satisfy (cf.~\eqref{eq:pure_point_spec_evo})
\begin{displaymath}
    e^{t V_\tau} = \sum_j e^{i\omega_{j,\tau}t} \langle \xi_{j,\tau}, \cdot \rangle_H \xi_{j,\tau}, \quad e^{t W_\tau} = \sum_j e^{i\omega_{j,\tau}t} \langle \zeta_{j,\tau}, \cdot \rangle_{\mathcal H_\tau} \zeta_{j,\tau} + \proj_{\mathcal H_\tau(X_\mu)^\perp},
\end{displaymath}
and $e^{t V_\tau} f$ converges as $\tau \to 0^+$ to $U^tf$ for every $f \in H$ by \cref{prty:v5} and equivalence of strong resolvent convergence and strong dynamical convergence. It can also be shown \cite{GiannakisEtAl24}*{Lemma~6} that for every $f \in \mathcal H_{\tau_0}$, the evolution under $e^{t W_\tau}$, $\tau \leq \tau_0$, converges to the true Koopman evolution in $L^2$ sense,
\begin{equation}
    \label{eq:rkhs_conv}
    \lim_{\tau\to 0^+} \left\lVert (K_\tau^* e^{t W_\tau} - U^t \iota)f \right\rVert_H =0.
\end{equation}
For the remainder of the paper we will use the notation $U^t_\tau := e^{t W_\tau}$.

\subsection{Reproducing kernel Hilbert algebras}
\label{sec:rkha}

RKHAs are RKHSs equipped with coalgebra structure that will be central to the Fock space scheme studied in this paper. In this subsection, we give basic definitions and outline some of properties of RKHAs that are most relevant to our work, referring the reader to \cite{GiannakisMontgomery24} for further details.

\begin{defn}
    \label{def:rkha}
    An RKHS $\mathcal H$ on a set $X$ with reproducing kernel $k \colon X \times X \to \mathbb C$ is a \emph{reproducing kernel Hilbert algebra (RKHA)} if $k_x \mapsto k_x \otimes k_x$, $x \in X$, extends to a bounded linear map (comultiplication) $\Delta\colon \mathcal H \to \mathcal H \otimes \mathcal H$.
\end{defn}

Since
\begin{displaymath}
    \langle k_x, \Delta^*(f \otimes g)\rangle_{\mathcal H} = \langle \Delta k_x, f \otimes g \rangle_{\mathcal H} = \langle k_x \otimes k_x, f \otimes g\rangle_{\mathcal H \otimes \mathcal H} = \langle k_x, f\rangle_{\mathcal H} \langle k_x, g\rangle_{\mathcal H} = f(x)g(x),
\end{displaymath}
it follows that $\Delta^*$ is a bounded linear map that implements pointwise multiplication. As a result, $\mathcal H$ is simultaneously a Hilbert function space and commutative algebra with respect to pointwise function multiplication. It can further be shown that $\mathcal H$ is a Banach algebra,
\begin{displaymath}
    \vertiii{ f g}_{\mathcal H} \leq \vertiii{ f}_{\mathcal H} \vertiii{ g }_{\mathcal H}, \quad \forall f, g \in \mathcal H,
\end{displaymath}
for an operator norm $\vertiii{ \cdot }_{\mathcal H}$ induced by the multiplier representation $\pi \colon \mathcal H \to B(\mathcal H)$ (defined analogously to~\eqref{eq:mult_rep}) that generates a coarser topology than the Hilbert space norm. By commutativity and associativity of pointwise multiplication, it follows that the comultiplication operator $\Delta : \mathcal H \to \mathcal H \otimes \mathcal H$ is cocommutative and coassociative, which implies in turn that $\Delta_n \colon \mathcal H \to \mathcal H^{\otimes n}$ with
\begin{equation}
    \label{eq:delta_ampl}
    \Delta_1 = \Delta, \quad \Delta_n = (\Delta \otimes \Id^{\otimes (n-1)}) \Delta_{n-1} \quad \text{for $n>1$},
\end{equation}
is a well-defined amplification of $\Delta$ to the tensor product spaces $\mathcal H^{\otimes (n+1)}$ for any $n \in \mathbb N$.

If the reproducing kernel $k$ is real-valued, $\mathcal H$ becomes a $^*$-algebra with isometric involution $^* \colon \mathcal H \to \mathcal H$ given by the pointwise conjugation of functions. If $\mathcal H$ contains the constant function $1_X \colon X \to \mathbb R$ that equals 1 everywhere on $X$, then it is a unital algebra with unit $1_X$.

Given an RKHA $\mathcal H$, we will let $\sigma(\mathcal H) \subset \mathcal H^*$ denote its spectrum as a Banach algebra, i.e., the set of nonzero multiplicative linear functionals $\chi: \mathcal H \to \mathbb C$,
\begin{displaymath}
    \chi(fg) = (\chi f)(\chi g), \quad \forall f,g \in \mathcal H,
\end{displaymath}
equipped with the weak-$^*$ topology of $\mathcal H^*$. The dual object to $\sigma(\mathcal H)$ is the cospectrum $\sigma_\text{co}(\mathcal H)$, which is defined as
\begin{displaymath}
    \sigma_\text{co}(\mathcal H) = \left\{ \xi \in \mathcal H : \langle \xi, \cdot \rangle_{\mathcal H} \in \sigma(\mathcal H) \right\}
\end{displaymath}
and is equipped with the weak topology of $\mathcal H$. Equivalently, we have that $\sigma_\text{co}(\mathcal H)$ is the subset of $\mathcal H$ consisting of elements $\xi$ such that $\Delta \xi = \xi \otimes \xi$.

It is clear that the kernel sections $k_x$ and evaluation functionals $\delta_x = \langle k_x, \cdot\rangle_{\mathcal H}$ are elements of $\sigma_\text{co}(\mathcal H)$ and $\sigma(\mathcal H)$, respectively, for every $x \in X$. In particular, the feature map $\varphi$ from~\eqref{eq:feature_map} can be viewed as a map from $X$ into the cospectrum $\sigma_\text{co}(\mathcal H)$. This map is injective whenever $\mathcal H$ has linearly independent kernel sections, and if $X$ is a topological space it is continuous iff $\mathcal H \subseteq C(X)$. If $\mathcal H$ is unital, $\sigma(\mathcal H)$ and $\sigma_\text{co}(\mathcal H)$ are compact Hausdorff spaces.

The paper \cite{GiannakisMontgomery24} developed constructions that can be used to build many examples of RKHAs. These examples include the spaces $\mathfrak A$ on $X=\mathbb T^d$ introduced in \cref{sec:pure_point_spec} which are unital RKHAs of continuous functions with homeomorphic spectra and cospectra to $X$. More generally, for every compact subset $X \subset \mathbb R^n$ there exists a unital RKHA $\mathcal H \subset C(\mathbb R^n)$ with $ X \cong \sigma(\mathcal H) \cong \sigma_\text{co}(\mathcal H)$.

\subsection{Tensor network approximation}
\label{sec:tensor_network}

As mentioned in \cref{sec:pure_point_spec}, a key structural property of skew-adjoint generators of unitary Koopman groups is that they act as derivations on algebras of observables via the Leibniz rule~\eqref{eq:leibniz}. While this property can considerably aid the efficiency of quantum simulation algorithms, aside from special cases (e.g., a priori known rotation systems as in \cite{GiannakisEtAl22}), one is compelled to work with regularizations of the generator that fail to satisfy the Leibniz rule. Examples include the methods of \cites{DasEtAl21,GiannakisValva24,GiannakisValva24b} outlined in \cref{sec:spectral_approx}, as well several other methods for discrete approximation of the typically continuous spectral measures exhibited by unitary Koopman groups under complex measure-preserving dynamics; e.g., \cites{DasEtAl21,ColbrookTownsend24}.

To our knowledge, there is currently no operator approximation methodology for unitary Koopman groups associated with measure-preserving flows that simultaneously preserves skew-adjointness and the Leibniz rule for the generator (but note the recent MultDMD technique \cite{BoulleColbrook24} that yields non-unitary approximations of the Koopman operator that preserve multiplicativity). In the transfer operator literature, a powerful approach for spectral analysis of hyperbolic dynamics is to work in anisotropic Banach spaces adapted to the stable/unstable directions of the dynamics, where the transfer operator is quasicompact (and thus exhibits isolated eigenvalues in the unit disc that are separated from an essential spectrum) \cites{BlankEtAl02,BaladiTsujii07,ButterleyLiverani07}. In general, however, these spaces lack the Hilbert space structure and unitary dynamics required of quantum systems.

\subsubsection{Dilation to Fock space}

As an effort to overcome these challenges, \cite{GiannakisEtAl24} developed a dilation scheme that lifts a given regularized generator to a Fock space generated by an RKHA, in which the Leibniz rule is recovered. In more detail, they consider a family of approximations $V_\tau \colon D(V_\tau) \to H$ of the generator satisfying \crefrange{prty:v1}{prty:v5} and the convergence property \eqref{eq:rkhs_conv} for the associated operators $W_\tau \colon D(W_\tau) \to \mathcal H_\tau$ acting on the RKHSs $\mathcal H_\tau \subseteq C(X)$.

Assuming that $X$ is a compact Hausdorff space, the Hilbert spaces $\mathcal H_\tau$ are built as unital RKHAs $\mathcal H_\tau$ with isomorphic (co)spectrum to $X$. Each RKHA $\mathcal H_\tau$ in this family generates an associated Fock space, $F(\mathcal H_\tau)$, defined using standard constructions from many-body quantum theory as the Hilbert space closure of the tensor algebra $T(\mathcal H_\tau) := \mathbb C \oplus \mathcal H_\tau \oplus \mathcal H_\tau^{\otimes 2} \oplus \ldots$ with respect to the inner product $\langle \cdot, \cdot \rangle_{F(\mathcal H_\tau)}$ satisfying
\begin{equation}
    \label{eq:fock_innerp}
    \begin{gathered}
        \langle a, b \rangle_{F(\mathcal H_\tau)} = \bar a b, \quad a, b \in \mathbb C, \\
        \langle a, f \rangle_{F(\mathcal H_\tau)} = 0, \quad a \in \mathbb C, \quad f \in \mathcal H_\tau, \\
        \langle f_1 \otimes \cdots \otimes f_n, g_1 \otimes \cdots \otimes g_n \rangle_{F(\mathcal H_\tau)} = \prod_{i=1}^n \langle f_i, g_i \rangle_{\mathcal H_\tau}, \quad f_i, g_i \in \mathcal H_\tau.
    \end{gathered}
\end{equation}

The scheme of \cite{GiannakisEtAl24} lifts the regularized generator $W_\tau$ to a skew-adjoint operator $\tilde W_\tau \colon D(W_\tau) \to F(\mathcal H_\tau) $, defined by linear extension of
\begin{align*}
    \tilde W_\tau(f_1 \otimes f_2 \otimes \cdots \otimes f_n)
    &= (W_\tau f_1) \otimes f_2 \otimes \cdots \otimes f_n \\
    & + f_1 \otimes (W_\tau f_2) \otimes \cdots \otimes f_n \\
    & + \ldots \\
    & + f_1 \otimes f_2 \otimes \cdots \otimes (W_\tau f_n)
\end{align*}
for $f_1, \ldots, f_n \in D(W_\tau)$. By construction, $\tilde W_\tau$ satisfies the Leibniz rule
\begin{displaymath}
    \tilde W_\tau(f \otimes g) = (\tilde W_\tau f) \otimes g + f \otimes (\tilde W_\tau g)
\end{displaymath}
with respect to the tensor product for all $f, g \in D(\tilde W_\tau)$ such that the left- and right-hand sides of the above equation are well-defined. As a result, we have:
\begin{enumerate}
    \item The point spectrum $\sigma_p(\tilde W_\tau)$ is an abelian group generated by $\sigma_p(W_\tau)$.
    \item $\tilde W_\tau$ generates a 1-parameter group of unitary operators $\tilde U^t_\tau = e^{t \tilde W_\tau}$, $ t \in \mathbb R$, that act multiplicatively with respect to the tensor product,
        \begin{displaymath}
            \tilde U^t_\tau(f \otimes g) = (\tilde U^t_\tau f) \otimes g + f \otimes (\tilde U^t_\tau g), \quad \forall, f, g \in F(\mathcal H_\tau).
        \end{displaymath}
\end{enumerate}

Putting together the above, it follows that for every vector $q \in L^2(\mu)$ with a representative $\xi \in \mathcal H_{\tau_0}$ for some $\tau_0>0$ and a multiplicative decomposition of the form $\xi = \xi_1 \ldots \xi_n$ for some $\xi_1, \ldots, \xi_n \in \mathcal H_{\tau_0}$, we have
\begin{equation}
    \label{eq:xi_tensor_evo}
    \Delta_n^* \tilde U^t_\tau \xi = \Delta_n^* \left( \bigotimes_{i=1}^n U^t_\tau \xi_i \right) = \prod_{i=1}^n U^t_\tau \xi_i \xrightarrow{\tau \to 0^+} U^t q
\end{equation}
in $L^2(\mu)$ norm, where $\Delta_n$ is the amplified comultiplication operator defined in~\eqref{eq:delta_ampl}. Intuitively, the Fock space $F(\mathcal H_\tau)$ generated by the RKHA $\mathcal H_\tau$ allows one to ``distribute'' the Koopman evolution of observables over tensor products in the Fock space for potentially arbitrarily high grading $n \in \mathbb N$.

\subsubsection{Quantum representation of statistical evolution}

The approach of \cite{GiannakisEtAl24} leverages~\eqref{eq:xi_tensor_evo} to build an approximation of the classical expectation $\mathbb E_p (U^t f) \equiv \int_X (U^t f) p \, d\mu$ for an observable $f \in L^\infty(\mu)$ and a probability density $p \in L^1(\mu)$. To any desired error tolerance, the square root $\sqrt p$ may be approximated by an element $q \in L^2(\mu)$ with a strictly positive representative $\xi \in \mathcal H_{\tau_0}$. The latter may be written in turn as  $\xi = (\xi^{1/n})^n$, $n \in \mathbb N$ (using the holomorphic functional calculus on $\mathcal H_{\tau_0}$ to compute the $n$-th root $\xi^{1/n} \in \mathcal H_\tau$), and dynamically evolved using the unitaries $\tilde U^t_\tau$ on the Fock space as in~\eqref{eq:xi_tensor_evo}.

To cast this in the form of a quantum state evolution, we map $\xi$ to a pure quantum state $\rho_\tau = \langle \eta_\tau, \cdot \rangle_{F(\mathcal H_\tau)} \eta_\tau \in Q(F(\mathcal H_\tau))$, $\tau \leq \tau_0$, with state vector
\begin{displaymath}
    \eta_\tau = s_1 \frac{\xi}{\lVert \xi\rVert_{\mathcal H_\tau}} + s_2 \frac{\xi^{1/2} \otimes \xi^{1/2}}{\lVert \xi^{1/2}\rVert^2_{\mathcal H_\tau}} + \ldots \in F(\mathcal H_\tau).
\end{displaymath}
Here, $s_1, s_2, \ldots$ is an arbitrary sequence of strictly positive numbers such that $\sum_{i=1}^\infty s_i^2 = 1$.

On the side of the observable, we first smooth $f$ to an element $f_\tau = K_\tau f \in \mathcal H_\tau$ where $K_\tau$ is the kernel integral operator induced by the reproducing kernel of $\mathcal H_\tau$,
\begin{displaymath}
    K_\tau f = \int_X k_\tau(\cdot, x)\, d\mu(x).
\end{displaymath}
Then, for $n \in \mathbb N$ we map the multiplication operator $M_{f_\tau} \in B(\mathcal H_\tau)$ (which is a bounded quantum observable by virtue of the RKHA structure of $\mathcal H_\tau$, as in \cref{sec:quantum_embedding_pure_point_spec}) to a quantum observable $A_{f, \tau, n} \in B(F(\mathcal H_\tau))$ on the Fock space given by the following amplification of $M_{f_\tau}$:
\begin{displaymath}
    A_{f,\tau,n} = \Delta_n M_{f_\tau} \Delta_n^*.
\end{displaymath}
Proposition~9 in \cite{GiannakisEtAl24} then shows that $\mathbb E_p (U^t f)$ is approximated to arbitrarily high accuracy as $\tau \to 0^+$ by the normalized quantum expectation
\begin{equation}
    \label{eq:tensor_approx}
    f^{(t)}_{\tau,n} = \frac{\mathbb E_{\rho_\tau}(\bm{\tilde U}^t_\tau A_{f,\tau,n})}{\mathbb E_{\rho_\tau}(\bm{\tilde U}^t_\tau A_{1_X,\tau,n})},
\end{equation}
where $\bm{\tilde U}^t_\tau \colon B(F(\mathcal H_\tau)) \to B(F(\mathcal H_\tau))$ is the adjoint operator acting on quantum observables on the Fock space via $\bm{\tilde U}^t_\tau A = \tilde U^t_\tau A \tilde U^{t*}_\tau$.

\subsubsection{Finite-rank approximation}

An interesting aspect of the approximation~\eqref{eq:tensor_approx} is that it converges as $\tau \to 0^+$ for any $n \in \mathbb N$. We can take advantage of this fact by projecting the state vector $\eta_{\tau}$ to a Fock subspace $F(\mathcal Z_{\tau,d}) \subset F(\mathcal H_\tau)$ generated by a finite-dimensional subspace $\mathcal Z_{\tau,d} \subset \mathcal H_\tau$ of dimension $2d + 1$, spanned by $d$ linearly independent eigenfunctions of $W_\tau$ with nonzero corresponding eigenvalue, their complex conjugates, and the constant eigenfunction $1_X$ such that $W_\tau 1_X = 0$. Even though $F(\mathcal Z_{\tau,d})$ is infinite-dimensional, the fact that the range of $A_{f,\tau,n}$ lies in the subspace $\mathcal H_\tau^{\otimes n}\subset F(\mathcal H_\tau)$ of fixed grading $n$ means that the quantum expectation of $\bm{\tilde U}^t_\tau A_{f,\tau,n}$ appearing in~\eqref{eq:tensor_approx} with respect to the projected state onto $F(\mathcal Z_{\tau,d})$ is equal to the expectation of a finite-rank operator  $\bm{\tilde U}^t_\tau A_{f,\tau,n,d}$ with respect to $\rho_\tau$, where $\ran A_{f,\tau,n,d} \subseteq \mathcal Z_{\tau,d}^{\otimes n} \subset F(\mathcal Z_{\tau,d})$.

Letting $f_{\tau,n,d}^{(t)} \in \mathbb C$ denote the analog of $f_{\tau,n}^{(t)}$ resulting from this approximation,
\begin{equation}
    \label{eq:tensor_approx_d}
    f^{(t)}_{\tau,n,d} = \frac{\mathbb E_{\rho_\tau}(\bm{\tilde U}^t_\tau A_{f,\tau,n,d})}{\mathbb E_{\rho_\tau}(\bm{\tilde U}^t_\tau A_{1_X,\tau,n,d})},
\end{equation}
it was shown \cite{GiannakisEtAl24}*{Theorem~12} that $f_{\tau,n,d}^{(t)}$ converges to $f_{\tau,n}^{(t)}$ as $n \to \infty$ for \emph{fixed} sufficiently large $d \in \mathbb N$.

Intuitively, this result signifies that due to the multiplicative nature of the underlying Koopman operator $U^t$, the correspondence between tensor products and pointwise products induced by the RKHA structure of $\mathcal H_\tau$, and the strong convergence $U^t_\tau \to U^t$ as $\tau \to 0^+$, the quantum evolution $t \mapsto \bm{\tilde U}^t_\tau A_{f,\tau,n,d}$ taking place in subspaces of $F(\mathcal Z_{\tau,d})$ of increasingly large dimension $(2d+1)^n$ captures sufficient information about the true Koopman evolution $t \mapsto U^t f$ so as to recover it fully as $\tau \to 0^+$ and $n\to \infty$ (at appropriate rates), even if $d$ is held fixed to a finite value.

Practically, this allows one to devote resources to compute a collection of approximate Koopman eigenfunctions of modest size $2d +1$, and algebraically amplify them to access tensor product approximation spaces $\mathcal Z_{\tau,d}^{\otimes n}$ of dimension $(2d + 1)^n$.

In \cite{GiannakisEtAl24}, it was shown that~\eqref{eq:tensor_approx_d} can be efficiently evaluated by a tensor network with tree structure, alleviating the cost of brute-force computations in $\mathcal Z_{\tau,d}^{\otimes n}$ that increases exponentially with $n$. Numerical experiments with low-dimensional dynamical systems (including the Stepanoff flow example studied in this paper) found significant improvement in prediction skill for data-driven implementations of the tensor network approach over conventional models based on linear combinations of Koopman eigenfunctions.

%###############################################################
\section{Second quantization framework}
\label{sec:overview}

Our second-quantization framework approximates the measure-preserving dynamical system introduced in \cref{sec:dynamical_system} by a rotation system on a family of tori embedded in the spectrum of a commutative Banach algebra, built as a weighted symmetric Fock space. In this section, we describe the main steps of our approach, using as a starting point an approximation of the skew-adjoint generator $V\colon D(V) \to H$ by a family of skew-adjoint, diagonalizable operators $V_\tau\colon D(V_\tau) \to H$ and their RKHS counterparts $W_\tau \colon D(W_\tau) \to \mathcal H_\tau$ satisfying \crefrange{prty:v1}{prty:v5} and~\eqref{eq:rkhs_conv}. As noted in \cref{sec:spectral_approx}, the numerical results in this paper are based in the method of \cite{GiannakisValva24b} (summarized in \cref{app:markov,app:spectral_approx}), but any other technique satisfying these properties may also be employed.

Similarly to the tensor network approximation scheme \cite{GiannakisEtAl24} outlined in \cref{sec:tensor_network}, to recover the Leibniz rule lost through regularization, we construct an amplification $\tilde W_\tau \colon D(\tilde W_\tau) \to \fk$ of $W_\tau$ to a Fock space, $\fk$, generated by $\mathcal H_\tau$ on which it acts as an algebra derivation. Two major differences between our approach and the spaces $F(\mathcal H_\tau)$ from \cite{GiannakisEtAl24} are that (i) $\fk$ is based on the symmetric tensor algebra of $\mathcal H_\tau$; and (ii) Hilbert space closure is taken with respect to a weighted inner product. As a result, $\fk$ becomes a Banach algebra with respect to the symmetric tensor product that is isomorphic to an RKHA of continuous functions on its spectrum. In addition, unlike \cite{GiannakisEtAl24}, we do not require that the $\mathcal H_\tau$ are RKHAs. This affords the second-quantization approach presented in this paper greater flexibility. The Fock space $\fk$, generator $\tilde W_\tau$, and related RKHAs to $\fk$ will be the foci of \crefrange{sec:overview_fock_space}{sec:fock_space_amplification}.

Next, in \cref{sec:top_models}, we employ the unitary evolution group generated by  $\tilde W_\tau$ to build a rotation system on the Banach algebra spectrum $\sigma(\fk)$ that will serve as an approximate topological model of the unitary Koopman evolution of observables generated by $V$. In \cref{sec:embedding_overview,sec:finite_dimensional_overview}, we couple this model with a procedure for embedding  observables of the original system into $\fk$ and its associated RKHAs. The outcome of this construction is an asymptotically consistent approximation of the unitary Koopman evolution of observables under potentially mixing dynamics by trigonometric (Fourier) polynomials of arbitrarily large degree on suitably chosen tori. A key aspect of this approximation is that it captures information from products of approximate Koopman eigenfunctions (the eigenfunctions of $W_\tau$). In addition, the approximation is positivity preserving.

\subsection{Weighted Fock space}
\label{sec:overview_fock_space}

The space $\fk$ is constructed as the closure of the symmetric tensor algebra $T^\vee(\mathcal H_\tau) := \mathbb C \oplus \mathcal H_\tau \oplus \mathcal H_\tau^{\vee 2} \oplus \ldots$ with respect to the inner product satisfying (cf.\ \eqref{eq:fock_innerp})
\begin{equation}
    \label{eq:weighted_sym_fock_innerp}
    \begin{gathered}
        \langle a, b \rangle_\fk = \bar a b, \quad a, b \in \mathbb C, \\
        \langle a, f \rangle_\fk = 0, \quad a \in \mathbb C, \quad f \in \mathcal H_\tau, \\
        \langle f_1 \vee \cdots \vee f_n, g_1 \vee \cdots \vee g_n \rangle_{\fk} = \frac{w^2(n)}{n!^2} \sum_{\sigma,\sigma'\in S_n} \prod_{i=1}^n \langle f_{\sigma(i)}, g_{\sigma'(i)} \rangle_{\mathcal H_\tau}, \quad f_i, g_i \in \mathcal H_\tau,
    \end{gathered}
\end{equation}
for a strictly positive weight function $w \colon \mathbb N_0 \to \mathbb R_{>0}$. Here, $\vee$ denotes the symmetric tensor product, defined as the average
\begin{displaymath}
    f_1 \vee \dots \vee f_n = \frac{1}{n!} \sum_{\sigma \in S_n} f_{\sigma(1)} \otimes \dots \otimes f_{\sigma(n)}, \quad f_i \in \mathcal H_\tau,
\end{displaymath}
over the $n$-element permutation group $S_n$, and $\mathcal H_\tau^{\vee n}$ is the closed subspace of $\mathcal H_\tau^{\otimes n}$ consisting of symmetric tensors. The map $f_1 \otimes \cdots \otimes f_n \mapsto f_1 \vee \cdots \vee f_n$ defines, by linear extension, the orthogonal projection from $\mathcal H_\tau^{\otimes n}$ to $\mathcal H_\tau^{\vee n}$.  We will use $\Omega \equiv 1 \in \mathbb C \subset \fk$ to denote the ``vacuum'' vector of the Fock space. By convention, we will always choose $w$ such that $w(0) = 1$.

In \cref{thm:quotient}, we show that if $w^{-2}$ is summable and subconvolutive,
\begin{equation}
    \label{eq:w_subonv}
    w^{-2} \in \ell^1(\mathbb N_0), \quad w^{-2} * w^{-2}(n) \leq C w^{-2}(n),
\end{equation}
$\fk$ becomes a unital Banach algebra with respect to the symmetric tensor product for a norm $\vertiii{\cdot}_{\fk}$ equivalent to the Hilbert space norm,
\begin{displaymath}
    \vertiii{f \vee g}_{\fk} \leq \vertiii{f}_{\fk} \vertiii{g}_{\fk}, \quad \forall f, g \in \fk,
\end{displaymath}
and with $\Omega$ as the unit. Moreover, associated with $\fk$ is a coproduct, i.e., a bounded operator $\Delta \colon \fk \to \fk \otimes \fk$ such that
\begin{displaymath}
    \Delta^*(f \otimes g) = f \vee g.
\end{displaymath}
Among many possible constructions, in this paper we use weights from the subexponential family
\begin{displaymath}
    w(n) = e^{\sigma n^p}, \quad \sigma > 0, \quad p \in (0, 1),
\end{displaymath}
which is a prototypical class of subconvolutive weights used in harmonic analysis on locally compact abelian groups (e.g., \cites{Feichtinger79,Grochenig07}).

Boundedness of $\Delta$ allows us to characterize the spectrum of $\fk$ (i.e., the set of nonzero multiplicative functionals $\chi\colon \fk \to \mathbb C$) as the set
\begin{displaymath}
    \sigma(\fk) = \left\{ \chi = \langle \xi, \cdot \rangle_{\fk}: \xi = \sum_{n=0}^\infty w^{-2}(n) \eta^{\vee n}: \eta \in (\mathcal H_\tau)_{R_w} \right\} \subset \fk^*,
\end{displaymath}
where $R_w$ is the radius of convergence of the series $\sum_{n=1}^\infty w^{-2}(n) z^n$, $z \in \mathbb C$; see \cref{prop:rkha_spec}. Since $w \in \ell^2(\mathbb N)$, we have $ R_w \geq 1 $ and the set of admissible vectors $\eta$ in the definition above includes the unit ball of $\mathcal H_\tau$.

The dual object to $\sigma(\fk)$ is the cospectrum of $\fk$, which is defined as
\begin{displaymath}
    \sigma_\text{co}(\fk) = \left\{ \xi \in \fk : \langle \xi, \cdot \rangle_{\fk} \in \sigma(\fk) \right\} \subset \fk.
\end{displaymath}
Equivalently, we have that $\sigma_\text{co}(\fk)$ is the subset of $\fk$ consisting of elements $\xi$ such that $\Delta \xi = \xi \otimes \xi$. We equip $\sigma(\fk)$ and $\sigma_\text{co}(\fk)$ with the weak-$^*$ topology on $\fk^*$ and the weak topology on $\fk$, respectively. With these topologies, they become compact Hausdorff spaces.

Next, let $\varphi \colon X \to \mathcal H_\tau$ be a continuous, bounded feature map. Assuming boundedness of $k_\tau$ (as per \cref{sec:markov_operators}), a prototypical example is the canonical feature map of $\mathcal H_\tau$, $\varphi(x) = k_\tau(x,\cdot)$. We will distinguish this feature map using the notation $\varphi_\tau(x) \equiv k_\tau(x,\cdot)$, but in what follows we will also consider more general examples $\varphi$. The weighted Fock space $\fk$ has an associated feature map $\tilde\varphi\colon X \to \fk$, where
\begin{equation}
    \label{eq:feature_map_fock}
    \tilde\varphi (x) = \sum_{n=0}^\infty \frac{w^{-2}(n)}{\varpi^n} \varphi(x)^{\vee n},
\end{equation}
and $\varpi$ is a constant chosen such that $\varpi \geq \sup_{x \in X} \lVert \varphi(x)\rVert_{\mathcal H_\tau}$. The range of $\tilde\varphi$ then lies in the cospectrum $\sigma_\text{co}(\fk)$, which implies that $\hat \varphi_\tau \colon X \to \sigma(\fk)$ with
\begin{equation}
    \label{eq:feature_map_rkha}
    \hat\varphi_\tau(x) = \langle \tilde\varphi(x), \cdot\rangle_{\fk}
\end{equation}
is a well-defined map of state space into the spectrum of the weighted Fock space. We include a $\tau$ subscript in our notation for $\hat\varphi_\tau$ to emphasize the fact this map depends canonically on the inner product of $\fk$. The choice of the RKHS family $\mathcal H_\tau$ and feature maps $\varphi$ in our examples will ensure that $\hat\varphi_\tau$ is injective on the support of the invariant measure.

Next, using the convention $f^{\vee 0} = \Omega$, the $\{ \zeta_{j,\tau} \}_{j \in \mathbb Z}$ eigenbasis of $\mathcal H_\tau(X_\mu)$ from \eqref{eq:approx_koopman_eigs} induces an orthonormal basis of $F_w(\mathcal H_\tau(X_\mu)) \subseteq \fk$ consisting of elements
\begin{equation}
    \label{eq:zeta_fock}
    \zeta_{\tau}^{\vee J} = \frac{\zeta_{0,\tau}^{\vee j_0} \vee \zeta_{-1,\tau}^{\vee j_{-1}} \vee \zeta_{1,\tau}^{\vee j_1} \vee \zeta_{-2,\tau}^{\vee j_{-2}} \vee \zeta_{2,\tau}^{\vee j_2} \vee \cdots}{w(\lVert J \rVert)\sqrt{j_0! \, j_{-1}! \, j_1! \, j_{-2}! \, j_{2!} \cdots}}.
\end{equation}
Here, $J = (j_0, j_{-1}, j_1, j_{-2}, j_2, \ldots)$ is a multi-index of non-negative integers $j_i$, finitely many of which are nonzero, and $\lVert J \rVert = \sum_{j_i \in j} j_i$. Henceforth, we will use the symbol $\mathbb J$ to denote the set of such multi-indices.

In second-quantization terminology, the basis $ \{ \zeta_\tau^{\vee J} \}_{J \in \mathbb J}$ is referred to as an occupation number basis. Intuitively, we think of each basis vector $\zeta_{i,\tau}$ as being associated with a ``particle'' of a certain type. A density operator $\rho= \langle \zeta_\tau^{\vee J }, \cdot \rangle_{\fk} \zeta_\tau^{\vee J} \in Q(\fk)$ then induces a pure quantum state (see \cref{sec:pure_point_spec}) comprising of $j_0$ particles of type $\zeta_{\tau,0}$, $j_{-1}$ particles of type $\zeta_{\tau,-1}$, and so on, for a total number of $\lVert J\rVert$ particles.

\subsection{Reproducing kernel Hilbert algebras}
\label{sec:overview_rkha}

The weighted Fock space $\fk$ has two associated unital RKHAs: One, denoted as $\hfk$, is a space of continuous functions on the spectrum $\sigma(\fk)$, and another one, denoted as $\tilde{\mathcal H}_\tau$, is a space of continuous functions on the state space $X$.

The space $\hfk$ is built using the Gelfand transform $\Gamma \colon \fk \to C(\sigma(\fk))$, defined as $(\Gamma f)\chi = \chi f$. We show, by \cref{cor:rkha} below, that the image $\hfk \subseteq C(\sigma(\fk))$ of $\fk$ under $\Gamma$ has RKHA structure for  the reproducing kernel $\hat k_\tau \colon \sigma(\fk) \times \sigma(\fk) \to \mathbb C$,
\begin{displaymath}
    \hat k_\tau(\chi_1, \chi_2) = \langle \xi_1, \xi_2 \rangle_{\fk}, \quad \chi_i = \langle \xi_i, \cdot \rangle_{\fk},
\end{displaymath}
and the Gelfand map $\Gamma$ acts as a Banach algebra isomorphism. Using the Gelfand map, we obtain an orthonormal basis $ \{ \Gamma \zeta_\tau^{\vee J} \}_{J \in \mathbb J}$ of $\hfk$ from the basis elements $\zeta_\tau^{\vee J}$ in~\eqref{eq:zeta_fock}.

Next, $\tilde{\mathcal H}_\tau$ is given by the pullback of $\hfk$ onto $X$ under the feature map~\eqref{eq:feature_map_rkha}, i.e., $\tilde{\mathcal H}_\tau = \{ f \circ \hat \varphi_\tau: f \in \hfk \} \subset C(X)$. One then readily verifies that $\tilde{\mathcal H}_\tau$ has the reproducing kernel $\tilde k_\tau \colon X \times X \to \mathbb R_{>0}$, where
\begin{displaymath}
    \tilde k_\tau(x, y) = \sum_{n=0}^\infty \frac{w^{-2}(n)}{\varpi^n} \langle \varphi(x), \varphi(y)\rangle^n_{\mathcal H_\tau}.
\end{displaymath}
Moreover, the maps $\hat \pi \colon \hfk \to \tilde{\mathcal H}_\tau$ and $\tilde \pi \colon \fk \to \tilde{\mathcal H}_\tau$ defined by $\hat \pi f = f \circ \hat \varphi_\tau$ and $\tilde \pi = \hat \pi \circ \Gamma$ are algebra homomorphisms,
\begin{displaymath}
    \hat \pi(fg) = (\hat\pi f)(\hat \pi g), \quad \tilde \pi(f \vee g) = (\tilde \pi f) (\tilde \pi g).
\end{displaymath}
We think of an element on the weighted Fock space in the preimage $\tilde \pi^{-1}(f)$ as lying ``above'' observable $f \in \tilde{\mathcal H}_\tau$. We can also interpret $\hat f \in \hat \pi^{-1}(f)$ as an extension of $f$ from $X \cong \hat\varphi_\tau(X) \subset \sigma(\fk)$ to $\sigma(\fk)$.

Similarly to $\mathcal H_\tau$, the RKHAs $\tilde{\mathcal H}_\tau$ have associated integral operators $\tilde K_\tau \colon H \to \tilde{\mathcal H}_\tau$, where
\begin{displaymath}
    \tilde K_\tau f = \int_X \tilde k_\tau(\cdot, x) f(x) \, d\mu(x),
\end{displaymath}
and $\tilde K_\tau^*$ implements the inclusion map from $\tilde{\mathcal H}_\tau$ into $H$.

Henceforth, we will assume that the feature vector $\varphi(x)$ is a real-valued element of $\mathcal H_\tau$ for every $x \in X$. Since the reproducing kernel of $\mathcal H_\tau$ is also real-valued, this implies that $\tilde k_\tau$ is real-valued and $\tilde{\mathcal H}_\tau$ is closed under complex conjugation of functions.

\subsection{Spectral tori}
\label{sec:spectral_tori}

Let $\mathbb A_w$ be the subset of $\ell^2(\mathbb N_0)$ consisting of vectors $a = (a_j)_j$ with norm $\lVert a\rVert_{\ell^2(\mathbb N)} \leq R_w$ and non-negative elements $a_j$. The structure of the weighted Fock spaces and RKHAs introduced above can be further characterized by defining, for each $a \in \mathbb A_w$ and each sequence $z = (z_j)_j \in \ell^\infty(\mathbb N)$ with unimodular elements $z_j$, the vectors
\begin{equation}
    \label{eq:xi_cospec}
    \xi_{\tau,a,z} = \sum_{n=0}^\infty w^{-2}(n) \left(a_0 +\sum_{j=1}^\infty a_j \frac{z_j \zeta_{j,\tau} + \overline{z_j \zeta_{j,\tau}}}{\sqrt 2}\right)^{\vee n} \in \sigma_\text{co}(\fk),
\end{equation}
and the subsets $\mathbb T_{\tau, a}$ of the spectrum $\sigma(\fk)$ as
\begin{displaymath}
    \mathbb T_{\tau, a} = \left\{ \chi_{\tau,a,z} \equiv \langle \xi_{\tau,a,z}, \cdot\rangle_{\fk}: z = (z_j)_j \in \ell^\infty(\mathbb N), \; z_j \in \mathbb T^1 \subset \mathbb C \right\}.
\end{displaymath}
Each set $\mathbb T_{\tau, a}$ has the topology of a torus of dimension equal to the number of nonzero elements of $(a_1, a_2, \ldots)$. In what follows, $\mathbb S_\tau = \bigcup_{a \in \mathbb A_w} \mathbb T_{\tau, a} \subset \sigma(\fk)$ will be the (disjoint) union of these tori.

For $x \in X$, let $a_\tau(x) = (a_{0,\tau}(x), a_{1,\tau}(x), \ldots) \in \mathbb A_w$ and $z_\tau(x) = (z_{\tau,1}(x), z_{\tau,2}(x), \ldots) \in \ell^\infty(\mathbb N)$, where
\begin{equation}
    \label{eq:feat_a}
    a_{0,\tau} = \frac{1}{\varpi} \lvert \langle \zeta_{j,\tau}, \varphi(x)\rangle_{\mathcal H_\tau}\rvert, \quad a_{j,\tau}(x) = \frac{\sqrt 2}{\varpi} \lvert \langle \zeta_{j,\tau}, \varphi(x) \rangle_{\mathcal H_\tau}\rvert, \quad z_{j,\tau}(x) = e^{i \arg \langle \zeta_{j,\tau}, \varphi(x)\rangle_{\mathcal H_\tau}},
\end{equation}
for $j \in \mathbb N$. Observe that since $ \{ \zeta_{j, \tau} \}_{j \in \mathbb Z}$ is an orthonormal basis of $\mathcal H_\tau(X_\mu)$ and $\varphi(x)$ is real-valued, we have
\begin{equation}
    \label{eq:feature_vec}
    \frac{\varphi(x)}{\varpi} = a_{0,\tau}(x) + \sum_{j=1}^\infty a_{j,\tau}(x) \frac{z_{j,\tau}(x)\zeta_{j,\tau} + \overline{z_{j,\tau}(x)\zeta_{j,\tau}}}{\sqrt 2}, \quad \forall x \in X_\mu,
\end{equation}
and $\lVert a_\tau(x)\rVert_{\ell^2(\mathbb N_0)} = \lVert\varphi(x)\rVert_{\mathcal H_\tau} / \varpi \leq 1$. Comparing the above expression with~\eqref{eq:xi_cospec}, it follows that if $x$ lies in the support of the invariant measure, the feature vector $\varphi(x)$ is equal to $\xi_{\tau,a_\tau(x),z_\tau(x)}$, and thus that $\hat\varphi_\tau(x)$ is a point in the torus $\mathbb T_{\tau, a_\tau(x)}$ with coordinates $z_\tau(x)$. Henceforth, we will use the abbreviated notation $\mathbb T_{\tau, x} \equiv \mathbb T_{\tau, a_\tau(x)}$ for $x \in X$. Letting $\hat X_{\mu,\tau} = \hat\varphi_\tau(X_\mu) \subseteq \sigma(\fk)$ be the image of the support of the invariant measure in the Fock space spectrum under $\hat\varphi_\tau$, we have that $\hat X_{\mu,\tau}$ is a subset of $\mathbb S_\tau$.

Suppose now that feature map $\varphi$ is injective. Then, $\varphi\rvert_M$ has a continuous inverse on the forward-invariant compact set $M \subseteq X$. As a result, $\hat\varphi_\tau\rvert_M$ has a continuous inverse on the image $\hat M_\tau = \hat\varphi_\tau(M) \subseteq \sigma(\fk)$ in the Fock space spectrum, and $\hat\varphi_\tau(x) \mapsto \hat\varphi_\tau(\Phi^t(x))$ defines a continuous (semi)flow $\hat \Phi_\tau^t \colon \hat M_\tau \to \hat M_\tau$. This flow is topologically conjugate to $\Phi^t$ in the sense of the following commuting diagram involving continuous maps:
\begin{displaymath}
    \begin{tikzcd}
        M \ar[r,"\Phi^t"] \ar[d,"\hat \varphi_\tau",swap] & M \ar[d,"\hat \varphi_\tau"] \\
        \hat M_\tau \ar[r,"\hat\Phi^t_\tau"] & \hat M_\tau
    \end{tikzcd}
\end{displaymath}

\subsection{Fock space amplification}
\label{sec:fock_space_amplification}

For $\tau>0$ we define the strongly continuous, one-parameter evolution group $ \{ \tilde U^t_\tau\colon \fk \to \fk \}_{t\in \mathbb R}$ by linear extension of
\begin{displaymath}
    \tilde U^t_\tau (f_1 \vee \cdots \vee f_n) = (U^t_\tau f_1) \vee \cdots \vee (U^t_\tau f_n), \quad f_i \in \mathcal H_\tau.
\end{displaymath}
The generator of this group is a skew-adjoint operator  $\tilde W_\tau \colon D(\tilde W_\tau) \to \fk$ satisfying
\begin{displaymath}
    \tilde W_\tau(f_1 \vee \cdots \vee f_n) = (W_\tau f_1) \vee f_2 \cdots \vee f_n + \ldots + f_1 \vee \cdots \vee f_{n-1} \vee (W_\tau f_n)
\end{displaymath}
for $f_1, \ldots, f_n \in D(W_\tau)$. Thus, $\tilde U^t_\tau$ acts multiplicatively on $\fk$, and $\tilde W_\tau$ satisfies the Leibniz rule on a suitable subspace. By virtue of the Leibniz rule, $\tilde W_\tau$ admits the eigendecomposition
\begin{displaymath}
    \tilde W_\tau \zeta_\tau^{\vee J} = i\tilde\omega_{J,\tau} \zeta_\tau^{\vee J}, \quad \tilde\omega_{J,\tau} = \sum_{i=-\infty}^\infty j_i \omega_{i,\tau}, \quad J = (j_0, j_{-1}, j_1, j_{-2}, j_2, \ldots) \in \mathbb J,
\end{displaymath}
and the point spectrum $\sigma_p(\tilde W_\tau)$ has the structure of a subgroup of $i \mathbb R$. Thus, lifting $U^t_\tau$ into the Fock space recovers key structural properties of one-parameter unitary Koopman groups that were lost through regularization.

\subsection{Topological models of regularized Koopman evolution}
\label{sec:top_models}

The unitary operators $U^t_\tau$ induce a flow $R_\tau^t \colon \sigma(\fk) \to \sigma(\fk)$, $t \in \mathbb R$, where $R^t_\tau(\chi) = \chi \circ \tilde U^t_\tau$. Each torus $\mathbb T_{\tau, a}$ is an invariant set under $R_\tau^t$. On these sets, $R^t_\tau$ takes the form of a rotation system generated  by the eigenfrequencies $\omega_{j,\tau}$:
\begin{displaymath}
    R_\tau^t(\chi_{\tau,a,z}) = \chi_{\tau,a, z^{(t)}},
\end{displaymath}
where $z = (z_j)_{j \in \mathbb N}$ and $z^{(t)} = (e^{-i \omega_{j,t}} z_j)_{j \in \mathbb N}$. Equivalently, we have
\begin{displaymath}
    R_\tau^t(\langle \xi_{\tau,a,z}, \cdot\rangle_\fk) = \langle \tilde U^{-t}_\tau\xi_{\tau,a,z}, \cdot\rangle_\fk,
\end{displaymath}
so the vector $\xi_{\tau,a,z}$ evolves under the adjoint (``Perron--Frobenius'') operators $\tilde U^{t*}_\tau = \tilde U^{-t}_\tau$.

The rotation system $R_\tau^t$ constitutes a topological model of the regularized Koopman dynamics $U^t_\tau$ as a rotation system on the spectrum of the weighted Fock space $\fk$. This is non-trivial since $U^t_\tau$ is not a composition operator induced by a flow on the original state space $X$. Since $\hat X_\tau$ is a subset of $\mathbb S_\tau \subset \sigma(\fk)$ (and $\mathbb S_\tau$ is invariant under $R^t_\tau$), the union of tori $\mathbb S_\tau$ provides a common topological setting for studying the dynamical system associated with the regularized Koopman operators $U^t_\tau$ (represented by $R_\tau^t$) in relation to the original dynamical system $\Phi^t$ (represented by $\hat\Phi^t_\tau$).

Letting $\hat U^t_\tau \colon C(\sigma(\fk)) \to C(\sigma(\fk))$, $\hat U^t_\tau f = f \circ R^t_\tau$, be the Koopman operator on continuous functions induced by $R^t_\tau$, we have
\begin{displaymath}
    \Gamma \circ \tilde U^t_\tau = \hat U^t_\tau \circ \Gamma.
\end{displaymath}
As a result, $\hfk$ is a $\hat U^t_\tau$-invariant subspace of $C(\sigma(\fk))$.

Our approach is to employ the rotation system $R^t_\tau$ in conjunction with an embedding of observables in $H$ into the weighted Fock space $\fk$ and associated RKHA $\hfk$ to build approximations of the unitary Koopman evolution of observables under $U^t$. Since the RKHS $\tilde{\mathcal H}_\tau$ spans a dense subspace of $H$ and for a general observable $f \in \mathcal H_\tau \subset \tilde{\mathcal H}_\tau$ there exist multiple elements $\tilde f \in \fk$ of the Fock space lying above it (i.e., $\tilde\pi \tilde f = f$) there are different possible strategies for representing observables in $H$ by elements of $\fk$ or, equivalently, $\hfk$.

One such strategy is to employ a continuous lifting map $\mathcal L \colon \mathcal H_\tau \to \fk$ satisfying
\begin{equation}
    \label{eq:lifting}
    \tilde \pi \circ \mathcal L f = f.
\end{equation}
Defining $\hat{\mathcal L} = \Gamma \circ \mathcal L$ then yields
\begin{displaymath}
    f^{(t)}_\tau = \hat \pi \circ \hat U^t_\tau \circ \hat{\mathcal L}
\end{displaymath}
as an approximation of the Koopman evolution $U^t f$ based on the rotation system on $\sigma(\fk)$. A basic choice in that direction is to choose $\mathcal L \colon \mathcal H_\tau \hookrightarrow \fk $ as the inclusion map. In \cref{sec:top_models_complex} we consider aspects of ``nonlinear'' approximations wherein $\mathcal L$ yields a representation of $f$ as a product--sum of functions, thus utilizing higher gradings of the Fock space.

Yet another approach is to relax the requirement that $\mathcal L$ yields an exact representation, and consider embeddings that satisfy~\eqref{eq:lifting} approximately while still utilizing higher gradings of the Fock space and converging to the true Koopman evolution in an appropriate asymptotic limit. The construction and analysis of such an approximation will be the focus of \cref{sec:embedding_overview,sec:finite_dimensional_overview} below.

\subsection{Fock space embedding of observables}
\label{sec:embedding_overview}

We use integral operators to represent $f \in H$ by elements of the RKHA $\hfk$, whose restrictions on finite-dimensional tori $\mathbb T_{\tau, a} \subset \mathbb S_\tau$ are polynomials of arbitrarily large degree $m \in \mathbb N$ of the coordinates $z_j$.

Let $\kappa: X \times X \to \mathbb R_{>0}$ be a strictly positive, bounded, continuous kernel function such that
\begin{equation}
    \label{eq:kappa_decay}
    \kappa(x, y) < \kappa(x, x), \quad \forall x, y \in X: \, x \neq y.
\end{equation}
For example, given a metric $d \colon X \times X \to \mathbb R$ that metrizes the topology of $X$, a prototypical kernel satisfying \eqref{eq:kappa_decay} is the radial Gaussian kernel,
\begin{equation}
    \label{eq:gaussian_kernel}
    \kappa(x,y) = \exp\left( - \frac{d^2(x,y)}{\varepsilon^2}\right), \quad \varepsilon > 0.
\end{equation}
For any such kernel $\kappa$ and $\tau > 0$, define the smoothed kernel $k_\tau \colon X \times X \to \mathbb R_{> 0}$, where $\kappa_\tau(\cdot, y) = K_\tau \iota \kappa(\cdot, y) \in \mathcal H_\tau$. Moreover, for $m \in \mathbb N$ define the integral operators $\mathcal K_{m,\tau} \colon H \to \fk$ and $\hat{\mathcal K}_{m,\tau} \colon H \to \hfk$, where
\begin{displaymath}
    \mathcal K_{m,\tau} f = \int_X \kappa_\tau^{\vee m}(\cdot, y) f(y) \, d\mu(y), \quad \hat{\mathcal K}_{m,\tau} f = \Gamma \mathcal K_{m,\tau} f = \int_X (\Gamma \kappa_\tau(\cdot, y))^m f(y) \, d\mu(y).
\end{displaymath}
Well-definition of these operators will be verified in \cref{sec:kappa_ops}. Note that for every $x \in X$ the pointwise power $y \mapsto (\kappa_\tau(x, y))^m$ lies in the RKHA $\tilde{\mathcal H}_\tau$ with reproducing kernel
\begin{displaymath}
    \tilde k_\tau(x, y) = \sum_{n=0}^\infty \frac{w^{-2}(n)}{\varpi_\tau^n} k_\tau(x, y)^n.
\end{displaymath}
Moreover, $\kappa_\tau^{\vee m}(\cdot, y) \in \fk$ lies above $\kappa_\tau(\cdot, y)^m \in \tilde{\mathcal H}_\tau$,
\begin{displaymath}
    \tilde \pi(\kappa_\tau^{\vee m}(\cdot, y)) = \kappa_\tau(\cdot, y)^m.
\end{displaymath}

We map observables $f \in H$ to elements of the RKHA $\hfk$ by means of the integral operators $\hat{\mathcal K}_{m,\tau}$; specifically, $\hat g_{m,\tau} := \hat{\mathcal K}_{m,\tau} \iota f$. We will also employ $\hat h_{m,\tau} := \hat{\mathcal K}_{m,\tau} 1_X$ for normalization purposes (cf.~\eqref{eq:tensor_approx}). The functions $\hat g_{m,\tau}$ and $\hat h_{m,\tau}$ evolve unitarily under the action of the Koopman operator $\hat U^t_\tau$ to
\begin{displaymath}
    \hat g^{(t)}_{m,\tau} := \hat U^t_\tau \hat g_{m,\tau}, \quad \hat h^{(t)}_{m,\tau} := \hat U^t_\tau \hat h_{m,\tau},
\end{displaymath}
respectively.

In order to render an approximation of the true Koopman evolution $U^t f$ using $\hat g^{(t)}_{m,\tau}$ and $\hat h^{(t)}_{m,\tau}$, we pull back these functions to state space $X$ by means of a feature map. Setting $\sigma > 0$ and $ \tau \in (0, \sigma/2]$, we let $\varphi_\sigma \colon X \to \mathcal H_\sigma$ be the canonical feature map into $\mathcal H_\sigma$,
\begin{displaymath}
    \varphi_\sigma(x) = k_\sigma(x, \cdot),
\end{displaymath}
and $\varphi^{(\mu)}_\sigma$ its projection onto the subspace $\mathcal H_\sigma(X_\mu) \subseteq \mathcal H_\sigma$,
\begin{displaymath}
    \varphi^{(\mu)}_\sigma = \proj_{\mathcal H_\sigma(X_\mu)} \circ \varphi_\sigma.
\end{displaymath}
We then define $\tilde\varphi^{(\mu)}_\sigma \colon X \to F_w(\mathcal H_\sigma)$ as (cf.~\eqref{eq:feature_map_fock})
\begin{displaymath}
    \tilde\varphi^{(\mu)}_\sigma (x) = \sum_{n=0}^\infty \frac{w^{-2}(n)}{\varpi_\sigma^{2n}} \varphi^{(\mu)}_\sigma(x)^{\vee n}, \quad \varpi_\sigma = \sup_{x \in X} \lVert \varphi_\sigma(x)\rVert_{\mathcal H_\sigma}.
\end{displaymath}
Since $\sigma > \tau$, we have $ F_w(\mathcal H_\sigma) \subset \fk $, so we can view $\tilde\varphi^{(\mu)}_\sigma$ as a feature map into $\fk$ that is based on the canonical feature vectors associated with $\mathcal H_\sigma(X_\mu) \subseteq \mathcal H_\sigma \subseteq \mathcal H_\tau$. In particular, for every $\sigma \geq \tau$ we have $\tilde\varphi^{(\mu)}_\sigma(x) \in \sigma_\text{co}(\fk)$ since $\varpi_\sigma = \lVert \varphi_\sigma(x)\rVert_{\mathcal H_\sigma} \geq \lVert \varphi_\sigma(x)\rVert_{\mathcal H_\tau}$ and $\lVert \varphi^{(\mu)}_\sigma(x)\rVert_{\mathcal H_\sigma} \leq \lVert \varphi_\sigma(x)\rVert_{\mathcal H_\sigma}$.

We therefore obtain a feature map $\hat\varphi^{(\mu)}_{\sigma,\tau} \colon X \to \sigma(\fk)$ mapping into the spectrum of $\fk$ via (cf.~\eqref{eq:feature_map_rkha})
\begin{displaymath}
    \hat\varphi^{(\mu)}_{\sigma,\tau}(x) = \langle \tilde\varphi^{(\mu)}_\sigma(x), \cdot\rangle_{\fk}.
\end{displaymath}
By construction, the image $\hat X^{(\mu)}_{\sigma,\tau} = \hat\varphi_{\sigma,\tau}(X)$ of state space under this feature map lies in in the union of tori $\mathbb S_\tau \subset \sigma(\fk)$, so we can dynamically evolve each point $\hat\varphi_{\sigma,\tau}(x)$ using the rotation system $R^t_\tau$ from \cref{sec:spectral_tori} restricted to $\mathbb S_\tau$.

\begin{rk}
    A reason for building the two-parameter family of feature maps $\hat\varphi^{(\mu)}_{\sigma,\tau}$ based on $\varphi^{(\mu)}_\sigma$ (as opposed to, say, the canonical feature maps $\varphi_\tau$ of $\mathcal H_\tau$) is to control the regularity of feature vectors as elements of $H$ when taking $\tau \to 0^+$ limits associated with Koopman operator approximation. Indeed, it follows from the semigroup \cref{prty:k3,prty:k4} and \eqref{eq:ran_k_tau} that for every $\tau \in (0, \sigma/2]$ and $x \in X$,
    \begin{equation}
        \label{eq:feat_sigma_tau}
        \varphi^{(\mu)}_\sigma(x) = K_\tau G_{\frac{\sigma}{2}-\tau} q_\sigma(x),
    \end{equation}
    for a continuous function $q_\sigma \colon X \to H$. The strong convergence~\eqref{eq:rkhs_conv} then implies
    \begin{equation}
        \label{eq:rkhs_conv_phi}
        \lim_{\tau\to 0^+} \left\lVert K_\tau^* U^t_\tau \varphi^{(\mu)}_\sigma(x) - U^t \iota \varphi^{(\mu)}_\sigma(x)\right\rVert_H = 0, \quad \forall x \in X.
    \end{equation}
    One further finds that
    \begin{equation}
        \label{eq:feat_sigma_tau_l2}
        q_\sigma(x) = K^*_{\sigma/2} \varphi_{\sigma/2}(x),
    \end{equation}
    which allows to relate $H$ inner products with $q_\sigma(x)$ to pointwise evaluation in the RKHS $\mathcal H_{\sigma/2}$,
    \begin{displaymath}
        \langle q_\sigma(x), f \rangle_H = \langle \varphi_{\sigma/2}, K_{\sigma/2}f\rangle_{\mathcal H_{\sigma/2}} = (K_{\sigma/2} f)(x), \quad \forall f \in H, \quad \forall x \in X.
    \end{displaymath}
\end{rk}

Next, let $\hat M_{\sigma,\tau}^{(\mu)} = \hat\varphi^{(\mu)}_{\sigma,\tau}(M) \subseteq \hat X^{(\mu)}_{\sigma,\tau}$ be the image of the compact set $M$ under the feature map $\hat\varphi^{(\mu)}_{\sigma,\tau}$. As we will show in \cref{sec:fock_consistency}, for sufficiently small $\tau$, $\hat h^{(t)}\rvert_{\hat M^{(\mu)}_{\sigma,\tau}}$ is bounded away from zero and thus has a multiplicative inverse in $C(\hat M^{(\mu)}_{\sigma,\tau})$. As a result,
\begin{equation}
    \label{eq:f_t_tau_m}
    \hat f_{m,\tau}^{(t)} := \frac{\hat g^{(t)}_{m,\tau}\rvert_{\hat M^{(\mu)}_{\sigma,\tau}}}{\hat h^{(t)}_{m,\tau}\rvert_{\hat M^{(\mu)}_{\sigma,\tau}}}, \quad f^{(t)}_{m,\sigma,\tau} := \hat f_{m,\tau}^{(t)} \circ \hat \varphi^{(\mu)}_{\sigma,\tau} \rvert_M \equiv \frac{\hat g^{(t)}_{m,\tau} \circ \hat\varphi^{(\mu)}_{\sigma,\tau}\rvert_M}{\hat h^{(t)}_{m,\tau} \circ \hat\varphi^{(\mu)}_{\sigma,\tau}\rvert_M}
\end{equation}
are well-defined continuous functions on $\hat M^{(\mu)}_{\sigma,\tau}$ and $M$, respectively. Our main approximation result is as follows.

\begin{thm}
    \label{thm:fock_consistency}
    With notation and assumptions as above, for every $f \in C(M)$, $f^{(t)}_{m,\sigma,\tau}$ converges to the true Koopman evolution $U^t \iota f$ in $H$ in the iterated limit of $m \to \infty$ after $\sigma\to 0^+$ after $\tau \to 0^+$,
    \begin{displaymath}
        \lim_{m\to\infty}\lim_{\sigma\to 0^+}\lim_{\tau\to 0^+} \left\lVert \iota f^{(t)}_{m,\sigma,\tau} - U^t \iota f\right\rVert_H = 0.
    \end{displaymath}
\end{thm}

\begin{proof}
    See \cref{sec:fock_consistency}.
\end{proof}

\subsection{Finite-dimensional approximation}
\label{sec:finite_dimensional_overview}

To arrive at a practical approximation algorithm, we approximate $f^{(t)}_{m,\sigma,\tau}(x) = \hat f^{(t)}_{m,\tau}(\hat \varphi_{\sigma,\tau}(x))$ by evaluating $\hat f^{(t)}_{m,\tau}$ on points in finite-dimensional tori in $\mathbb S_\tau$ lying near $\hat\varphi_{\sigma,\tau}(x)$.

To that end, for a parameter $d \in \mathbb N$, define the truncated feature map $\varphi_{\sigma,\tau,d} \colon X \to \mathcal H_\tau(X_\mu)$,
\begin{equation}
    \label{eq:feature_vec_d}
    \varphi_{\sigma,\tau,d}(x) = a_{0,\sigma,\tau}(x) + \sum_{j=1}^d a_{j,\sigma,\tau}(x) \frac{z_{j,\sigma,\tau}(x)\zeta_{j,\tau} + \overline{z_{j,\sigma,\tau}(x)\zeta_{j,\tau}}}{\sqrt 2},
\end{equation}
where the amplitudes $a_{j,\sigma,\tau}(x) \in \mathbb R_{\geq 0}$ and coordinates $z_{j,\sigma,\tau} \in S^1 \subset \mathbb C$ are given by specializing \eqref{eq:feat_a} to the feature map $\varphi^{(\mu)}_\sigma$ from \cref{sec:embedding_overview},
\begin{displaymath}
    a_{0,\sigma,\tau} = \frac{1}{\varpi_\sigma} \lvert \langle \zeta_{j,\tau}, \varphi^{(\mu)}_\sigma(x)\rangle_{\mathcal H_\tau}\rvert, \quad a_{j,\sigma,\tau}(x) = \frac{\sqrt 2}{\varpi_\sigma} \lvert \langle \zeta_{j,\tau}, \varphi^{(\mu)}_\sigma(x) \rangle_{\mathcal H_\tau}\rvert, \quad z_{j,\sigma,\tau}(x) = e^{i \arg\langle \zeta_{j,\tau}, \varphi^{(\mu)}_\sigma(x)\rangle_{\mathcal H_\tau}}.
\end{displaymath}
We then have that $\varphi_{\sigma,\tau,d}(x)$ converges to $\varphi^{(\mu)}_\sigma(x)$ in the norm of $\mathcal H_\tau$ for every $x \in X$. We also define truncated versions $\tilde\varphi_{\sigma,\tau,d} \colon X \to \fk$ and $\hat\varphi_{\sigma,\tau,d} \colon X \to \sigma(\fk)$ of $\tilde\varphi^{(\mu)}_\sigma$ and $\hat\varphi_{\sigma,\tau}$, respectively,
\begin{displaymath}
    \tilde\varphi_{\sigma,\tau,d}(x) = \sum_{n=0}^\infty \frac{w^{-2}(n)}{\varpi_\sigma^n} \varphi_{\sigma,\tau,d}(x)^{\vee n}, \quad \hat\varphi_{\sigma,\tau,d}(x) = \langle \tilde\varphi_{\sigma,\tau,d}(x), \cdot \rangle_\fk.
\end{displaymath}
Note that the image of $x$ under $\hat\varphi_{\sigma,\tau,d}$ lies in the finite-dimensional torus $ \mathbb T_{\tau,a_{\sigma,\tau}^{(d)}(x)} \equiv \mathbb T_{\sigma,\tau,d,x}$, where $a^{(d)}_{\sigma,\tau}(x) = (a_{0,\sigma,\tau}(x), \ldots a_{d,\sigma,\tau}(x), 0, \ldots) \in \mathbb A_w$. Moreover, from~\eqref{eq:feat_sigma_tau} and~\eqref{eq:feat_sigma_tau_l2} we get
\begin{displaymath}
    \langle \zeta_{j,\tau}, \varphi^{(\mu)}_\sigma(x)\rangle_{\mathcal H_\tau} = \langle \xi_{j,\tau}, T_\tau^* \varphi^{(\mu)}_\sigma(x)\rangle_H = (K_{\sigma+\frac{\tau}{2}} \xi_{j,\tau})(x), \quad \forall x \in X,
\end{displaymath}
so the coefficients $a_{j,\sigma,\tau}(x)$ and $z_{j,\sigma,\tau}(x)$ in~\eqref{eq:feature_vec_d} can be obtained by pointwise evaluation of kernel-smoothed approximate Koopman eigenfunctions $\xi_{j,\tau} \in H$. In what follows, we set
\begin{equation}
    \label{eq:zeta_sigma}
    \zeta_{j,\sigma,\tau} = K_{\sigma+\frac{\tau}{2}}\xi_{j,\tau} \equiv K_\sigma \iota \zeta_{j,\tau} \in \mathcal H_{\sigma + \frac{\tau}{2}},
\end{equation}
where $\zeta_{j,0,\tau} = \zeta_{j,\tau}$.

We approximate the prediction function $f^{(t)}_{m,\sigma,\tau}$ from~\eqref{eq:f_t_tau_m} by pulling back $\hat f^{(t)}_{m,\tau}$ from the tori $\mathbb T_{\sigma,\tau,d,x}$ to the state space $X$, viz.
\begin{displaymath}
    f^{(t)}_{m,\sigma,\tau,d} = \hat f^{(t)}_{m,\tau} \circ \hat \varphi_{\sigma,\tau,d} \in C(X).
\end{displaymath}
As $d \to \infty$ at fixed $\sigma$, $\tau$, and $m$, $f^{(t)}_{m,\sigma,\tau,d}$ converges to $f^{(t)}_{m,\sigma,\tau}$ pointwise on $X$ and in the norm of $H$,
\begin{equation}
    \label{eq:fock_consistency_d}
    \lim_{d\to\infty}\left\lVert \iota (f^{(t)}_{m,\sigma,\tau,d} - f^{(t)}_{m,\sigma,\tau})\right\rVert_H = 0;
\end{equation}
see \cref{sec:fock_consistency_d} for further details. Combining~\eqref{eq:fock_consistency_d} with \cref{thm:fock_consistency}, we deduce that the finite-dimensional Fock space approximation $f^{(t)}_{m,\sigma,\tau,d}$ converges to the true Koopman evolution $U^t f$ in $L^2$ sense in the iterated limit of $m \to \infty$ after $\sigma \to 0^+$ after $\tau \to 0^+$ after $d \to \infty$.

For simplicity of exposition, we will henceforth assume that the coefficients $a_{1,\sigma,\tau}(x), \ldots, a_{d,\sigma,\tau}(x)$ are all nonzero at the chosen $x \in X$, so that $\mathbb T_{\sigma,\tau,d,x}$ is $d$-dimensional. We then have that the restrictions of the functions $\hat g^{(t)}_{m,\tau}$ and $\hat h^{(t)}_{m,\tau}$ in~\eqref{eq:f_t_tau_m} on $\mathbb T_{\sigma,\tau,d,x}$ are polynomials of degree $m$ in the variables $z_1, \ldots, z_d, \overline{z_1}, \ldots, \overline{z_d}$.

To build these polynomials, define for $i \in \{ -d, \ldots, d \}$ the functions
\begin{equation}
    \label{eq:smoothed_eigenfuncs}
    \varrho_{j,\tau} = \int_X \kappa(\cdot, x) \zeta_{j,\tau}(x) \, d\mu(x) \in \mathcal H_\tau.
\end{equation}
Define also the index set $\mathbb J_{d,m} = \{ j = (j_{-d}, \ldots, j_d) \in \mathbb N_0^{2d +1} : \lvert j \rvert = m \}$, and for each $j \in \mathbb J_{d,m}$, the moment coefficients
\begin{equation}
    \label{eq:moms}
    C_{j,\tau, d}^{(g)} = \int_X f(x) \prod_{i=-d}^d \varrho_{i,\tau}^{j_i}(x) \, d\mu(x), \quad C_{j,\tau,d}^{(h)} = \int_X \prod_{i=-d}^d \varrho_{i,\tau}^{j_i}(x) \, d\mu(x),
\end{equation}
and the amplitude functions
\begin{displaymath}
    \alpha_{j,\sigma,\tau,d}(x) = \prod_{i=-d}^d a_{i,\sigma,\tau}(x).
\end{displaymath}
We then have that for every $\chi_{\tau,a_{\sigma,\tau}^{(d)}(x), z} \in \mathbb T_{\sigma,\tau,d,x}$,
\begin{equation}
    \label{eq:hat_gh_t_m_tau_d}
    \begin{aligned}
        \hat g_{m,\tau}^{(t)}\left(\chi_{\tau,a_{\sigma,\tau}^{(d)}(x),z}\right) &= w^{-2}(m) \sum_{j \in \mathbb J_{d,m}} \binom{m}{j_{-d} \cdots j_d} \alpha_{j,\sigma,\tau,d}(x) C_{j,\tau,d}^{(g)} \prod_{r=1}^d e^{-i (j_r - j_{-r})\omega_{r}t} z_r^{j_r - j_{-r}}, \\
        \hat h_{m,\tau}^{(t)}\left(\chi_{\tau,a_{\sigma,\tau}^{(d)}(x),z}\right) &= w^{-2}(m) \sum_{j \in \mathbb J_{d,m}} \binom{m}{j_{-d} \cdots j_d} \alpha_{j,\sigma,\tau,d}(x) C_{j,\tau,d}^{(h)} \prod_{r=1}^d e^{-i (j_r - j_{-r})\omega_{r}t} z_r^{j_r - j_{-r}}.
    \end{aligned}
\end{equation}
Thus, on $\mathbb T_{\sigma,\tau,d,x}$, $\hat g_{m,\tau}^{(t)}$ and $\hat h_{m,\tau}^{(t)}$ reduce to degree-$m$ polynomials in the coordinates $z_r$ and their conjugates. Correspondingly, the prediction function $\hat f^{(t)}_{m,\tau}$ is a rational function of the $z_r$.

Pulling back to $X$ via the feature map $\hat\varphi_{\sigma,\tau,d}$, we arrive at the approximation
\begin{equation}
    \label{eq:f_t_m_tau_d}
    f^{(t)}_{m,\sigma,\tau,d} = \frac{\hat g^{(t)}_{m,\tau} \circ \hat\varphi_{\sigma,\tau,d}}{\hat h^{(t)}_{m,\tau} \circ \hat\varphi_{\sigma,\tau,d}} = \frac{\sum_{j \in \mathbb J_{d,m}} \binom{m}{j_{-d} \cdots j_d} C_{j,\tau,d}^{(g)} \prod_{r=1}^d e^{i (j_{-r} - j_r)\omega_{r,\tau}t} \zeta_{r,\sigma,\tau}^{j_{-r} - j_r}}{\sum_{j \in \mathbb J_{d,m}} \binom{m}{j_{-d} \cdots j_d} C_{j,\tau,d}^{(h)} \prod_{r=1}^d e^{i (j_{-r} - j_r)\omega_{r,\tau}t} \zeta_{r,\sigma,\tau}^{j_{-r} - j_r}}.
\end{equation}
A key aspect of this approximation, which is evident from both~\eqref{eq:hat_gh_t_m_tau_d} and~\eqref{eq:f_t_m_tau_d}, is that it captures spectral information from the entire set of eigenfrequencies $ \{ \tilde\omega_{j,\tau} \}_{j \in \mathbb J_{d,m}}$, containing $\binom{m + 2d}{2d}$ elements. Thus, similarly to the tensor network scheme outlined in \cref{sec:tensor_network}, the second-quantization approach presented in this paper provides a route for algebraic amplification of $(2d+1)$-dimensional approximation spaces $\mathcal Z_{\tau,d} = \{ \zeta_{-d,\tau}, \ldots, \zeta_{d,\tau} \} \subseteq \mathcal H_\tau$. As noted in the beginning of this section, an advantageous aspect of the second-quantization approach is that it does not require $\mathcal H_\tau$ to have RKHA structure. In addition, it provides a framework for building topological models of regularized Koopman evolution as rotation systems on tori.

\section{Weighted Fock space}
\label{sec:fock_space}

In this section, we discuss in more detail the structure of the weighted Fock spaces $\fk$ as commutative Banach algebras (\cref{sec:symmetric_fock_space}), and establish the decomposition of their spectra in terms of the tori $\mathbb T_{\tau,a}$ (\cref{sec:spec_decomp}).

\subsection{Banach algebra structure}
\label{sec:symmetric_fock_space}

We will take an indirect approach to proving that $\fk$ is a commutative Banach algebra that is isomorphic to an RKHA of continuous functions on its spectrum. We will pass through the full weighted Fock space, $\fkf$, as a noncommutative Banach algebra with a comultiplication map and pass to the symmetric Fock space by taking a quotient.

In what follows, given a Hilbert space $\mathcal H$ (not necessarily an RKHS), and a weight $w \colon \mathbb N_0 \to \mathbb R_{>0}$ with $w(0) = 1$, $F_w^\otimes(\mathcal H)$ will be the weighted Fock space defined as the closure of the tensor algebra $T(\mathcal H) = \mathbb C \oplus \mathcal H \oplus \mathcal H^{\otimes 2} \oplus \ldots$ with respect to the inner product (cf.~\eqref{eq:weighted_sym_fock_innerp})
\begin{displaymath}
    \begin{gathered}
        \langle a, b \rangle_{F_w^\otimes(\mathcal H)} = \bar a b, \quad a, b \in \mathbb C, \\
        \langle a, f \rangle_{F_w^\otimes(\mathcal H)} = 0, \quad a \in \mathbb C, \quad f \in \mathcal H, \\
        \langle f_1 \otimes \cdots \otimes f_n, g_1 \otimes \cdots \otimes g_n \rangle_{F_w^\otimes(\mathcal H)} = w^2(n) \prod_{i=1}^n \langle f_i, g_i \rangle_{\mathcal H_\tau}, \quad f_i, g_i \in \mathcal H.
    \end{gathered}
\end{displaymath}
Moreover, we will use $\boxtimes$ to denote the usual tensor product between two Hilbert spaces in order to distinguish it from the tensor product $\otimes$ within $F_w^\otimes(\mathcal H)$. We define the symmetric Fock space $F_w(\mathcal H)$ similarly to $F_w(\mathcal H_\tau)$ from \cref{sec:overview_fock_space} as the closed subspace of $F_w^\otimes(\mathcal H)$ spanned by symmetric tensors. We let  $\text{Sym} \colon F_w^\otimes(\mathcal H) \to F_w(\mathcal H)$ be the orthogonal projection onto that subspace.

\begin{prop}
    \label{prop:rkha}
    With notation as above, let $w \colon \mathbb N_0 \to \mathbb R_{>0}$ be inverse square-summable, $w^{-1} \in \ell^2(\mathbb N)$. Then, there exists a unique bounded operator $\Delta \colon F_w^\otimes(\mathcal H) \to F_w^\otimes(\mathcal H) \boxtimes F_w^\otimes(\mathcal H)$ such that $\Delta^*(f \boxtimes g) = f \otimes g$ iff $w^{-2}$ is subconvolutive, $w^{-2} * w^{-2} \leq C w^{-2}$.
\end{prop}

\begin{proof}
    Fix an orthonormal basis $\{\phi_i\}_{i=0}^\infty \subset \mathcal H$, and write $\phi^{\otimes J} = \phi_{j_1} \otimes \cdots \otimes \phi_{j_m}$ for $j \in \mathbb J^\otimes := \bigcup_{m=0}^\infty \mathbb N_0^m$ with the convention that $\mathbb N_0^0 = \{ 0 \}$ and $f^{\otimes 0} = \Omega$ for every $f \in F_w^{\otimes}(\mathcal H)$ (cf.\ the basis vectors $\zeta_\tau^{\vee J}$ for the symmetric Fock space from \cref{sec:overview_fock_space}). Note that the multi-indices $J \in \mathbb J^\otimes$ form a graded semigroup with grading $\lvert J\rvert=m$ if $J \in \mathbb N^m_0$ and $J \cdot J'$ given by concatenation. Then $ \{\psi_J := w^{-1}(\lvert J \rvert) \phi^{\otimes J} \}_{J \in \mathbb J^\otimes}$ is an orthonormal basis of $F_w^\otimes(\mathcal H)$.

Define
\begin{displaymath}
    \Delta \psi_J
    = \sum_{J',J'' \in \mathbb J^\otimes} \left\langle \psi_{J'} \otimes \psi_{J''},\psi_J \right\rangle_{F_w^\otimes(\mathcal H)} \psi_{J'} \boxtimes \psi_{J''}
    =\sum_{J' \cdot J'' = J} \frac{w^{-1}(\lvert J'\rvert)w^{-1}(\lvert J''\rvert)}{w^{-1}(\lvert J \rvert)} \psi_{J'} \boxtimes \psi_{J''},
\end{displaymath}
and observe that
\begin{displaymath}
    \left\langle \Delta \psi_J, \Delta \psi_L \right\rangle_{F_w^\otimes(\mathcal H)} =
    \delta_{JL} \sum_{a+b=\lvert J \rvert} \frac{w^{-2}(a)w^{-2}(b)}{w^{-2}(a+b)} = \delta_{JL}  \dfrac{w^{-2} * w^{-2}(\lvert J \rvert)}{w^{-2}(|J|)}.
\end{displaymath}
Clearly, $\Delta$ extends to a bounded linear operator on $F_w^\otimes(\mathcal H)$ iff $w^{-2}$ is subconvolutive. Moreover, we have
\begin{multline*}
    \langle \psi_J, \Delta^*(\phi^{\otimes J'} \boxtimes \phi^{\otimes J''})\rangle_{F_w^\otimes(\mathcal H)} \\
    \begin{aligned}
        &= \sum_{K' \cdot K'' = J} \frac{\langle w^{-1}(\lvert K' \rvert) w^{-1}(\lvert K'' \rvert) \phi^{\otimes K'} \boxtimes \phi^{\otimes K''}, w^{-1}(\lvert K'\rvert) w^{-1}(\lvert K'' \rvert) \phi^{\otimes J'} \boxtimes \phi^{\otimes J''} \rangle_{F_w^\otimes(\mathcal H)\boxtimes F_w^\otimes(\mathcal H)}}{w^{-1}(\lvert J \rvert)}\\
        &=\delta_{J,J'\cdot J''} w(\lvert J \rvert)\\
        &= \langle  \psi_J, \phi^{\otimes J'} \otimes \phi^{\otimes J''}\rangle_{F_w^\otimes(\mathcal H)},
    \end{aligned}
\end{multline*}
which implies that $\Delta^*(f \boxtimes g) = f \otimes g$. This also implies that $\Delta$ is independent of the choice of basis $\{\phi_i\}_{i=0}^\infty \subset \mathcal H$.
\end{proof}

Boundedness of $\Delta$ implies that $(F_w^\otimes(\mathcal H), \otimes, 1_{\mathbb C})$ is a unital Banach algebra with respect to the operator norm (cf.\ \cref{sec:overview_rkha}),
\begin{displaymath}
    \vertiii{ f \otimes g }_{F_w^\otimes(\mathcal H)} \leq \vertiii{ f }_{F_w^\otimes(\mathcal H)} \vertiii{ g }_{F_w^\otimes(\mathcal H)},
\end{displaymath}
where
\begin{displaymath}
    \vertiii{ f}_{F_w^\otimes(\mathcal H)} = \sup_{h\in F^\otimes_w(\mathcal H) \setminus \{0\}} \frac{\lVert f \otimes h \rVert_{F^\otimes_w(\mathcal H)}}{\lVert h\rVert_{F_w^\otimes(\mathcal H)}}.
\end{displaymath}
Because $F^\otimes_w(\mathcal H)$ is unital, this norm is equivalent to the Hilbert space norm \cite{GiannakisMontgomery24},
\begin{displaymath}
    \lVert f \rVert_{F^\otimes_w(\mathcal H)} \leq \vertiii{ f}_{F_w^\otimes(\mathcal H)} \leq \lVert \Delta\rVert \lVert f\rVert_{F^\otimes_w(\mathcal H)}.
\end{displaymath}
Similarly to \cref{sec:rkha}, we will refer to $\Delta \colon F_w^\otimes(\mathcal H) \to F_w^\otimes(\mathcal H) \boxtimes F_w^\otimes(\mathcal H)$ as a coproduct.

\begin{prop}
    \label{prop:rkha_spec}
    The spectrum of $F_w^\otimes(\mathcal H)$ is given by
\begin{displaymath}
    \sigma(F_w^\otimes(\mathcal H)) =\left\lbrace \chi = \langle \xi, \cdot \rangle_{F_w^\otimes(\mathcal H)} : \xi = \sum_{n=0}^\infty w^{-2}(n) \eta^{\otimes n}: \eta \in (\mathcal H)_{R_w} \right\rbrace,
\end{displaymath}
where $R_w$ is the radius of convergence of $\sum_{n=0}^\infty w^{-2}(n)z^{2n}$, $z \in \mathbb C$, and the spectrum is endowed with the topology of weak-$^*$ convergence. Moreover, the weak-$^*$ topology on $\sigma(F_w^{\otimes}(\mathcal H))$ is equivalent to the weak topology on $(\mathcal H)_{R_w}$ with the identification of $\eta \mapsto \chi$ as above.
\end{prop}
\begin{proof}
    Let $\chi \colon F_w^\otimes(\mathcal H) \to \mathbb C$ be a nonzero multiplicative linear functional. Note that $\chi$ is necessarily unital since $F_w^\otimes(\mathcal H)$ is unital, $\chi \Omega = 1$. Since $\chi(\phi^{\otimes J}) = \prod_{i=1}^{|j|} \chi(\phi_{j_i})$, $\chi$ is determined by its restriction to $\mathcal H \subset F_w^\otimes(\mathcal H)$. Since characters are automatically continuous and the Hilbert space and Banach algebra norms, respectively $\lVert \cdot\rVert_{F_w^\otimes(\mathcal H)}$ and $\vertiii{ \cdot }_{F_w^\otimes(\mathcal H)}$, are equivalent, we have $\chi |_{\mathcal H} = \langle \eta, \cdot \rangle_{\mathcal H}$ for some $\eta \in \mathcal H$. As a consequence, $\chi = \langle \sum_{n=0}^\infty w^{-2}(n) \eta^{\otimes n}, \cdot \rangle_{F_w^\otimes(\mathcal H)}$. The restriction of $ \lVert \eta\rVert_{\mathcal H} \in [0, R_w]$ follows from finiteness of $\lVert \sum_{n=0}^\infty w^{-2}(n) \eta^{\otimes n} \rVert_{F_w^\otimes(\mathcal H)}$. The converse is easily verified by checking multiplicativity of $\langle \sum_{n=0}^\infty w^{-2}(n) \eta^{\otimes n}, \cdot \rangle_{F_w^\otimes(\mathcal H)}$.

Let $\tilde R_w = \sum_{n=0}^\infty w^{-2}(n)R_w^{2n}$. Since $w^{-2}$ is strictly positive and summable, weak convergence of $\eta_i \to \eta$ in $(\mathcal H)_{R_w}$ is equivalent to weak-$^*$ convergence of $\chi_i = \langle \xi_i, \cdot \rangle_{F_w^\otimes(\mathcal H)} \to \chi = \langle \xi, \cdot \rangle_{F_w^\otimes(\mathcal H)}$ in $\sigma(F^\otimes_w(\mathcal H))$, where $ \xi_i = \sum_{n=0}^\infty w^{-2}(n) \eta_i^{\otimes n}$ and $ \xi = \sum_{n=0}^\infty w^{-2}(n) \eta^{\otimes n}$ lie in $(F^\otimes_w(\mathcal H))_{\tilde R_w}$. Indeed, if $\eta_i \to \eta$ weakly, then for every $n \in \mathbb N_0$ and all $f_j \in \mathcal H$, $\chi_i(\bigotimes_{j=1}^n f_j) = \langle \eta_i^{\otimes n}, \bigotimes_{j=0}^n f_j \rangle_{F_w^\otimes(\mathcal H)} \to \langle \eta^{\otimes n}, \bigotimes_{j=1}^n f_j,\rangle_{F_w^\otimes(\mathcal H)} = \chi(\bigotimes_{j=1}^n f_j)$. Uniform boundedness of $\eta_i$ then implies weak-$^*$ convergence of $\chi_i$ to $\chi$. For the opposite direction, $\langle \eta_i, f\rangle_\mathcal H = \langle \sum_{n=0}^\infty w^{-2}(n) \eta_i^{\otimes n}, f \rangle_{F_w^\otimes(\mathcal H)} = \chi_i f \to \chi f =\langle \eta, f \rangle_\mathcal H$.
\end{proof}

By \cref{prop:rkha_spec}, for every $f, g \in F_w^\otimes(\mathcal H)$ and $\chi = \langle \xi, \cdot \rangle_{F_w^\otimes(\mathcal H)} \in \sigma(F_w^\otimes(\mathcal H))$,
$$\langle \Delta\xi, f\boxtimes g \rangle_{F_w^\otimes(\mathcal H) \boxtimes F_w^\otimes(\mathcal H)} =  \chi(\Delta^*(f \boxtimes w)) = (\chi f)(\chi g) = \langle \xi \boxtimes \xi, f \boxtimes g\rangle_{F_w^\otimes(\mathcal H) \boxtimes F_w^\otimes(\mathcal H)}.$$
Thus, we can equivalently characterize the spectrum of $F_w^\otimes(\mathcal H)$ as
\begin{equation}
    \label{eq:spec_comult}
    \sigma(F_w(\mathcal H)) = \left\{ \chi = \langle \xi, \cdot\rangle_{F_w^\otimes(\mathcal H)} : \Delta \xi = \xi \boxtimes \xi \right\}.
\end{equation}
The set of vectors $\xi$ in~\eqref{eq:spec_comult} is then the cospectrum $\sigma_\text{co}(F_w^\otimes(\mathcal H)) \subset F_w^\otimes(\mathcal H)$.

We may now pass to a quotient Banach algebra and Hilbert space to recover the symmetric Fock space $F_w(\mathcal H)$.

\begin{thm}
    \label{thm:quotient}
    With notation as above and $w^{-1} \in l^2(\mathbb N_0)$ strictly positive and $w^{-2}$ subconvolutive, we have the following isomorphisms as Hilbert spaces and Banach algebras:
$$F_w^\otimes(\mathcal H) \Big/ \bigcap_{\chi \in \sigma(F_w^\otimes(\mathcal H))} \ker(\chi) \cong \overline{\spn{ \left\lbrace\sum_{n=0}^\infty w^{-2}(n) \eta^{\otimes n}:\eta \in (\mathcal H)_{R_w} \right\rbrace }} \cong F_w(\mathcal H).$$
Moreover, $F_w^\otimes(\mathcal H)$ and $F_w(\mathcal H)$ have the same spectra as Banach algebras.
\end{thm}
\begin{proof}
Let $\mathcal X = \bigcap_{\chi \in \sigma(F_w^\otimes(\mathcal H)) } \ker(\chi) \subset F_w^\otimes(\mathcal H)$. The first isomorphism as Hilbert spaces follows from the orthogonal decomposition
$$F_w^\otimes(\mathcal H) = \mathcal X \oplus \overline{\spn{ \left\lbrace\sum_{n=0}^\infty w^{-2}(n) \eta^{\otimes n}: \eta \in (\mathcal H)_{R_w} \right\rbrace }}.$$
The quotient is also equipped with a Banach algebra structure with the norm
$$\vertiii{f+\mathcal X} = \inf_{g \in \mathcal X} \vertiii{f+g}_{F_w^\otimes(\mathcal H)},$$
which is equivalent to the Hilbert space norm
$$\lVert f+\mathcal X\rVert = \inf_{g \in \mathcal X} \lVert f+g\rVert_{F_w^\otimes(\mathcal H)}.$$

The second isomorphism as Hilbert spaces follows from two observations. For $f \in (\mathcal H)_{R_w}$ and $z \in \mathbb C$ with $\lvert z \rvert \leq 1$ we can isolate individual tensor powers using
$$f^{\otimes n} = \left. \frac{1}{w^{-2}(n) n!} \frac{d^n}{dz^n}\right\rvert_{z=0} \left(\sum_{n=0}^\infty w^{-2}(n)(zf)^{\otimes n}\right).$$
Second, the symmetric part of $\mathcal H^{\otimes n}$ is densely spanned by the vectors $f^{\otimes n}$ for $f \in \mathcal H$. Let  $n_i \in \mathbb N$ for $i = 1, \ldots ,m$ and $\{\phi_j\}_{j=1}^\infty \subset \mathcal H$ an orthonormal basis. Then by a combinatorial argument,
$$\frac{1}{n_1! \cdots n_m!}\sum_{\sigma \in S_n} \phi_{j_{\sigma(1)}} \otimes \cdots \otimes \phi_{j_{\sigma(n)}}  = \frac{1}{n_1! \cdots n_m!} \left. \frac{\partial^{n_1} \cdots \partial^{n_m}}{\partial z_1^{n_1} \cdots \partial z_m^{n_m}} \right|_{z_i = 0} \left( \sum_{i=1}^m z_i \phi_i \right)^{\otimes n},$$
where $(j_1, \ldots, j_n)$ is a vector with $n_l$ instances of $l$ and $\sum_{i=1}^m n_i =n$. The left hand side is a sum of every unique reordering of $\phi_{j_1} \otimes \cdots \otimes \phi_{j_n}$ while the right hand side is a sum of all sequences of $n$ numbers drawn from $1,\ldots,m$ with $n_l$ instances of $l$. Since the left hand side forms a basis for $\mathcal H^{\vee n} \subset \mathcal H^{\otimes n}$, we conclude that
$$\overline{\spn{ \left\lbrace\sum_{n=0}^\infty w^{-2}(n) \eta^{\otimes n}:\eta \in (\mathcal H)_{R_w} \right\rbrace }} \cong F_w(\mathcal H).$$

Meanwhile, the algebra structure on $F_w(\mathcal H)$ from the quotient and the symmetric tensor product are both given by $\text{Sym} \circ \Delta^*$. Combining the above facts, we conclude that the spaces $F_w^\otimes(\mathcal H) / \mathcal X$, $ \overline{\spn{ \left\lbrace\sum_{n=0}^\infty w^{-2}(n) \eta^{\otimes n}:\eta \in (\mathcal H)_{R_w} \right\rbrace }}$, and $F_w(\mathcal H)$ are isomorphic as Hilbert spaces and Banach algebras where all of the Banach algebra norms and Hilbert space norms are equivalent.

Finally, we identify the spectrum of $F_w(\mathcal H)$. By density of the linear span of $\sigma_\text{co}(F_w^\otimes(\mathcal H))$ in $F_w^\otimes(\mathcal H)$, the restriction of the coproduct to $F_w(\mathcal H)$ is well-defined as a linear map
$$\Delta|_{F_w(\mathcal H)} \colon F_w(\mathcal H) \to F_w(\mathcal H) \boxtimes F_w(\mathcal H) \subset F_w^\otimes(\mathcal H) \boxtimes F_w^\otimes(\mathcal H).$$
Thus, we have $\left(\Delta|_{F_w(\mathcal H)}\right)^* =\text{Sym} \circ \Delta^*$, where the adjoint in the left-hand side is taken on $B(F_w(\mathcal H), F_w(\mathcal H) \boxtimes F_w(\mathcal H))$ and in the right-hand side it is taken on $B(F_w^\otimes(\mathcal H), F_w^\otimes(\mathcal H) \boxtimes F_w^\otimes(\mathcal H))$. Now, for every $\chi \in \sigma(F_w(\mathcal H))$, there exists $\xi \in F_w(\mathcal H)$ such that $\chi=\left\langle \xi, \cdot \right\rangle_{F_w(\mathcal H)}$, giving
\begin{align*}
    \left\langle \xi \boxtimes  \xi, f\boxtimes g \right\rangle_{F_w^\otimes(\mathcal H) \boxtimes F_w^\otimes(\mathcal H)}
    &= \langle \xi, f \rangle_{F_w(\mathcal H)} \langle \xi, g\rangle_{F_w(\mathcal H)} =(\chi f)(\chi g) \\
    &= \left\langle \xi, f \vee g \right\rangle_{F_w(\mathcal H)} = \left\langle \Delta \xi, f \boxtimes g \right\rangle_{F_w^\otimes(\mathcal H) \boxtimes F_w^\otimes(\mathcal H)}.
\end{align*}
Thus, we have $\Delta \xi = \xi \boxtimes \xi$ and it follows from~\eqref{eq:spec_comult} that $\sigma(F_w(\mathcal H)) \subseteq \sigma(F_w^\otimes(\mathcal H))$. Since $\sigma(F_w^\otimes(\mathcal H)) \subseteq \sigma(F_w(\mathcal H))$ by the subspace relation $F_w(\mathcal H) \subset F_w^\otimes(\mathcal H)$, we deduce that $\sigma(F_w(\mathcal H)) = \sigma(F_w^\otimes(\mathcal H))$.
\end{proof}

Letting $\Gamma \colon F_w(\mathcal H) \to C(\sigma(F_w(\mathcal H)))$ denote the Gelfand map on the commutative Banach algebra $F_w(\mathcal H)$, the following is a corollary of \cref{thm:quotient} and equality of the spectra of $F_w^\otimes(\mathcal H)$ and $F_w(\mathcal H)$.

\begin{cor}
    \label{cor:rkha}
    The Gelfand map $\Gamma \colon F_w(\mathcal H) \to C(\sigma(F_w(\mathcal H)))$ is injective. As a result, the space $\hat F_w(\mathcal H) := \Gamma(F_w(\mathcal H))$ equipped with the inner product $\langle \hat f, \hat g \rangle_{\hat F_w(\mathcal H)} = \langle \Gamma^{-1} \hat f, \Gamma^{-1} \hat g \rangle_{F_w(\mathcal H)}$ is a unital RKHA with reproducing kernel $\hat k \colon \sigma(F_w(\mathcal H)) \times \sigma(F_w(\mathcal H)) \to \mathbb C$ defined as
    \begin{displaymath}
        \hat k(\chi_1, \chi_2) = \langle \xi_1, \xi_2 \rangle_{F_w(\mathcal H)}, \quad \chi_i = \langle \xi_i, \cdot \rangle_{F_w(\mathcal H)}.
    \end{displaymath}
    Moreover, $F_w(\mathcal H)$ and $\hat F_w(\mathcal H)$ are isomorphic Banach algebras under the Gelfand map.
\end{cor}

\begin{proof}
    Suppose that $\Gamma f = \Gamma g$ for some $f, g \in F_w(\mathcal H)$. Then $f-g \in \ker \chi$ for every $\chi \in \sigma(F_w(\mathcal H))$. Since $\sigma(F_w(\mathcal H)) = \sigma(F_w^\otimes(\mathcal H))$, it follows that $f-g$ lies in $\mathcal X \cap F_w(\mathcal H)$ and thus $f-g = 0$ since $F_w(\mathcal H) \cong F_w^\otimes(\mathcal H) / \mathcal X$ by \cref{thm:quotient}. This proves injectivity of $\Gamma$ and thus well-definition of $(\hat F_w(\mathcal H), \langle \cdot, \cdot \rangle_{\hat F_w(\mathcal H)})$ as a Hilbert space.

    Next, for every $\hat f = \Gamma f \in \hat F_w(\mathcal H)$ and $\chi = \langle \xi, \cdot \rangle_{F_w(\mathcal H)} \in \sigma(F_w(\mathcal H))$ we have
    \begin{displaymath}
        \hat f(\chi) = \chi f = \langle \xi, f \rangle_{F_w(\mathcal H)} = \langle  \Gamma \xi, \hat f\rangle_{\hat F_w(\mathcal H)} = \langle \hat k_\chi, \hat f \rangle_{\hat F_w(\mathcal H)}.
    \end{displaymath}
    Thus, $\hat F_w(\mathcal H)$ is an RKHS with reproducing kernel $\hat k$, as claimed.

    To verify that $\hat F_w(\mathcal H)$ is an RKHA, observe that
    \begin{displaymath}
        \lVert \hat k_\chi \otimes \hat k_\chi \rVert_{\hat F_w(\mathcal H)}^2 = \langle \xi, \xi\rangle^2_{F_w(\mathcal H)} = \langle \Delta \xi, \Delta \xi\rangle_{F_w(\mathcal H) \boxtimes F_w(\mathcal H)} \leq \lVert \Delta^2\rVert \lVert \xi\rVert^2_{F_w(\mathcal H)} = \lVert \Delta\rVert^2 \lVert \hat k_\chi \rVert^2_{\hat F_w(\mathcal H)}.
    \end{displaymath}
    This implies that $\hat \Delta \colon \hat k_\chi \mapsto \hat k_\chi \boxtimes k_\chi$ extends to a bounded linear operator $\Delta \colon \hat F_w(\mathcal H) \to \hat F_w(\mathcal H) \boxtimes \hat F_w(\mathcal H)$, so $\hat F_w(\mathcal H)$ is an RKHA according to \cref{def:rkha}. This RKHA is unital since $\Gamma \Omega = 1_{\sigma(F_w(\mathcal H))}$.

    Finally, for every $f, g \in F_w(\mathcal H)$ we have
    \begin{displaymath}
        \Gamma(f \otimes g)(\chi) = \chi( f \otimes g) = (\chi f)(\chi g) = (\Gamma f)(\Gamma g)
    \end{displaymath}
    and we conclude that $\Gamma$ is a Banach algebra isomorphism.
\end{proof}

Suppose now that $\mathcal H$ is an RKHS with reproducing kernel $k \colon X \times X \to \mathbb C$ and feature map $\varphi \colon X \to \mathcal H$. Then we may define feature maps $ \tilde \varphi \colon X \to F_w(\mathcal H)$ and $\hat \varphi \colon X \to \hat F_w(\mathcal H)$ analogously to~\eqref{eq:feature_map_fock} and~\eqref{eq:feature_map_rkha}, respectively. Pulling back $\hat F_w(\mathcal H)$ to $X$ then defines an RKHA $\tilde{\mathcal H} = \{ \hat f \circ \hat \varphi: \hat f \in \hat F_w(\mathcal H) \}$ with kernel
\begin{displaymath}
    \tilde k(x, y) = \hat k(\hat\varphi(x), \hat\varphi(y)) = \sum_{n=0}^\infty \frac{w^{-2}(n)}{\varpi^{2n}}k(x,y)^n.
\end{displaymath}
Since $\tilde{k}(x,y) - (\varpi w(1))^{-2}k(x,y)$ is a positive-definite kernel, we have the inclusion $\mathcal H \subset \tilde{\mathcal H}$ as shown in \cite{Aronszajn50}. Note that $\tilde{\mathcal H}_w$ is unital even if $1_X$ is not an element of $\mathcal H$.

\subsection{Decomposition of the spectrum and $^*$-structure}
\label{sec:spec_decomp}

Let $\mathcal H$ be a Hilbert space and $J \colon \mathcal H \to \mathcal H$ be a conjugate-linear unitary such that $J^2=\Id_\mathcal H$. Then we may also define a conjugate-linear unitary and Banach algebra involution $^*$ for $F^{\otimes}_w(\mathcal H)$ and $F_w(\mathcal H)$ by extending the following definition by conjugate linearity:
$$\Omega^*=\Omega, \quad (f_1 \otimes \cdots\otimes f_n)^* = J(f_n) \otimes \cdots  \otimes J(f_1), \quad (f_1 \vee \cdots\vee f_n)^* = J(f_n) \vee \cdots  \vee J(f_1).$$
Since $\langle J(v),J(u) \rangle_\mathcal H = \overline{\langle v, u \rangle_\mathcal H}$, these involutions satisfy $\langle \xi^*, \eta^*\rangle_{F_w^\otimes(\mathcal H)} = \overline{\langle \xi,\eta\rangle_{F_w^\otimes(\mathcal H)}}$ and $\langle \xi^*, \eta^*\rangle_{F_w(\mathcal H)} = \overline{\langle \xi,\eta\rangle_{F_w(\mathcal H)}}$.

In what follows, $\mathbb J \subset \mathbb \mathbb \mathbb N_0^{\mathbb N}$ will be the index set for basis vectors of $F_w(\mathcal H)$ introduced in \cref{sec:overview_fock_space}, and $\mathbb A_w \subset \ell^2(\mathbb N_0)$ the set of square-integrable sequences with positive elements defined in \cref{sec:spectral_tori}.

\begin{prop} \label{prop:interpolation}
    Fix an orthonormal basis $\{\phi_j\}_{j=0}^\infty \subset \mathcal H$, a sequence $a \in \mathbb A_w$, and consider the subset
\begin{displaymath}
    \mathbb T_a^\text{complex} = \left\lbrace\langle \xi, \cdot \rangle_{F_w(\mathcal H)} : \xi =\sum_{n=0}^\infty w^{-2}(n)\left(\sum_{j=0}^\infty a_jz_j\phi_j\right)^{\vee n}: (z_j)_{j} \in l^\infty(\mathbb N_0), \; \lvert z_j \rvert = 1 \right\rbrace \subset \sigma(F_w(\mathcal H)).
\end{displaymath}
Then, $\hat{F}_w(\mathcal H)|_{\mathbb T_a^\text{complex}}$ is an RKHA of continuous functions on $\mathbb T_a^\text{complex}$, and if $a_j >0$, then $f|_{\mathbb T_a^\text{complex}}$ uniquely determines $f \in \hat{F}_w(\mathcal H)$.
\end{prop}

\begin{proof}
    Recall from \cref{sec:notation} that $\hat F_w(\mathcal H)\rvert_{\mathbb T_a^\text{complex}}$ is isomorphic as a Hilbert space to the norm closure of $\spn \{ \hat k_\chi : \chi \in \mathbb T_a^\text{complex} \}$ in $\hat F_w(\mathcal H)$, which is in turn isomorphic to the norm closure of $ \Xi_a := \spn \{ \xi \in F_w(\mathcal H): \langle \xi, \cdot\rangle_{F_w(\mathcal H)} \in \mathbb T_a^\text{complex} \}$ in $F_w(\mathcal H)$ by \cref{cor:rkha}. Thus, to prove the proposition it suffices to show that $\Xi_a$ is a dense subspace of $F_w(\mathcal H)$. To that end, we will show that for every sequence $(j_0, j_1, \ldots) \in \mathbb J$ we can recover the symmetric tensor $\phi_{0}^{\vee j_0} \cdots \phi_{i_0}^{\vee j_{i_0}}$ in the closure $\overline{\Xi_a}^{\lVert \cdot \rVert_{F_w(\mathcal H)}}$, where $j_i$ is constantly zero after $i_0$. Since the collection of such tensors forms an orthogonal basis of $F_w(\mathcal H)$ (see \cref{sec:overview_fock_space}), we can then deduce that $\overline{\Xi_a}^{\lVert \cdot \rVert_{F_w(\mathcal H)}} = F_w(\mathcal H)$.

    Let $m=\sum_i j_i$ and $\nu_N$ denote the Haar measure on $\mathbb T^N$. Pick $N > i_0$ large enough such that $\left\lVert\sum_{i=N+1}^\infty a_i z_i \phi_{i}\right\rVert_{F_w(\mathcal H)} = \delta_N <1/2$. Then,
\begin{multline*}
    \int_{\mathbb T^N} z_0^{-j_0} \cdots z_N^{-j_N} \sum_{n=0}^\infty w^{-2}(n)\left(\sum_{i=0}^\infty a_iz_i \phi_{i}\right)^{\vee n} \, d\nu_N(z_0,\ldots,z_N) \\
    \begin{aligned}
        &=\int_{\mathbb T^N} z_0^{-j_0} \cdots z_N^{-j_N} \sum_{n=0}^\infty w^{-2}(n)\sum_{k=0}^n \binom{n}{k} \left(\sum_{i=0}^N a_iz_i \phi_{i}\right)^{\vee k} \vee \left(\sum_{i=N+1}^\infty a_iz_i \phi_{i}\right)^{\vee n-k}\, d\nu_N(z_0,\ldots,z_N) \\
        &=\sum_{n=m}^\infty w^{-2}(n)a_0^{j_0} \cdots a_{i_0}^{j_{i_0}} \binom{n}{m} \left(\phi_{0}^{\vee j_0} \cdots \phi_{N}^{\vee j_N}\right)\vee \left(\sum_{i=N+1}^\infty a_iz_i \phi_{i}\right)^{\vee n-m} \\
        &=w^{-2}(m)a_0^{j_0} \cdots a_{i_0}^{j_{i_0}}\left(\phi_{0}^{\vee j_0} \cdots \phi_{i_0}^{\vee j_{i_0}}\right) + r_N,
    \end{aligned}
\end{multline*}
where
\begin{displaymath}
    r_N = a_0^{j_0} \cdots a_{i_0}^{j_{i_0}} \left(\phi_{0}^{\vee j_0} \cdots \phi_{i_0}^{\vee j_{i_0}}\right)\vee \left(\sum_{i=N+1}^\infty a_iz_i \phi_{i}\right) \vee \sum_{n=m+1}^\infty w^{-2}(n) \binom{n}{m}  \left(\sum_{i=N+1}^\infty a_iz_i \phi_{i}\right)^{\vee n-m-1}.
\end{displaymath}
If we can show that the remainder $r_N$ above converges to zero as $N$ increases, then we will be able to isolate $\phi_{0}^{\vee j_0} \cdots \phi_{N}^{\vee j_{i_0}}$. Since $\binom{n}{m} \leq n^m/m!$,
\begin{align*}
    \lVert r_N\rVert^2_{F_w(\mathcal H)}
    &= \left\lVert \sum_{n=m+1}^\infty w^{-2}(n) \binom{n}{m}  \left(\sum_{i=N+1}^\infty a_iz_i \phi_{i}\right)^{\vee n-m-1}\right\rVert^2_{F_w(\mathcal H)} \\
    &\leq \sum_{n=m+1}^\infty w^{-2}(n)\frac{n^{2m}}{(m!)^2}\delta_N^{2(n-m-1)} \\
    &\leq \frac{\|w^{-2}\|_{\ell^\infty(\mathbb N_0)}}{(m!)^2} \sum_{n=m+1}^\infty\frac{n^{2m}}{4^{n-m-1}} = \text{Const}(w,m) <\infty.
\end{align*}
By boundedness of the symmetric product in $F_w(\mathcal H)$ and summability of $w^{-2}$, the remainder $r_N$ converges to zero as $N \to \infty$.

Therefore, by strict positivity of $w^{-2}(m)$ and all of the $a_i$'s, $\phi_{0}^{\vee j_0} \cdots \phi_{i_0}^{\vee j_{i_0}} \in \overline{\Xi_a}^{\lVert \cdot\rVert_{F_w(\mathcal H)}}$ for every sequence of nonnegative integers $(j_i)_i \in \mathbb J$, making $\Xi_a$ a dense subspace of $F_w(\mathcal H)$.
\end{proof}

In fact, $\sigma(F_w(\mathcal H)) = \bigcup_{a \in \mathbb A_w} \mathbb T_a^\text{complex}$, and we can identify $F_w(\mathcal H) \cong \hat F_{w,a}(\mathcal H)$  for any $a \in \mathbb A_w$ with strictly positive elements. Here, $\hat F_{w,a}(\mathcal H)  \subset C(\mathbb T_a)$ is the RKHA with reproducing kernel $\hat k^{(a)} \colon \mathbb T_a^\text{complex} \times \mathbb T_a^\text{complex} \to \mathbb C$,
$$\hat k^{(a)}(\vec{z}, \vec{y}) = \sum_{n=0}^\infty w^{-2}(n) \left( \sum_{j=0}^\infty a_j^2 \overline{z_j}y_j \right)^n,$$
where $\vec z \equiv \chi_{a,z} = \langle \xi_{a,z}, \cdot \rangle_{F_w(\mathcal H)}$ and $\vec y \equiv \chi_{a,z} = \langle \xi_{a,y}, \cdot \rangle_{F_w(\mathcal H)}$ for $\xi_{a,z}, \xi_{a,y} \in \sigma_\text{co}(F_w(\mathcal H))$.
The isomorphism $\upsilon_a: \hat F_{w,a}(\mathcal H) \to F_w(\mathcal H)$ realizing this identification is defined by linear extension of
\begin{displaymath}
    \upsilon_a \left(\hat k^{(a)}(\vec z, \cdot)\right) = \xi_{a, z}.
\end{displaymath}

In \cref{sec:overview}, we specialize to tori in $\sigma(F_w(\mathcal H_\tau))$ with real kernel. If $\mathcal H \subset C(X)$ is an RKHS with a real kernel $k \colon X \times X \to \mathbb R$ and a self-conjugate orthonormal basis $ \{ \phi_j \}_{j \in \mathbb Z}$ containing the unit function $1_X = \phi_0$, then
\begin{displaymath}
    k_x = a_0(x)1_X +\sum_{j=1}^\infty a_j(x)\frac{z_j(x) \phi_j + \overline{z_j(x) \phi_j}}{\sqrt 2}
\end{displaymath}
analogously to~\eqref{eq:feature_vec}. The associated tori
\begin{displaymath}
    \mathbb T_{a(x)} = \left\{ \langle \xi, \cdot\rangle_{F_w(\mathcal H)}: \xi = \sum_{n=0}^\infty w^{-2}(n) \left(a_0(x) 1_X +\sum_{j=1}^\infty a_j(x)\frac{z_j \phi_j + \overline{z_j \phi_j}}{\sqrt{2}}\right)^{\vee n}: \lvert z_j \rvert = 1 \right\},
\end{displaymath}
have real-valued kernels,
\begin{displaymath}
    k^{(a)}(\vec{z}, \vec{y}) = \sum_{n=0}^\infty w^{-2}(n) \left(a_0^2(x) + \sum_{j=1}^\infty  a_j^2(x) \Real(\overline{z_j}y_j) \right)^n.
\end{displaymath}
Even between points in two different tori, $ \chi = \langle \xi, \cdot \rangle_{F_w(\mathcal H)} \in \mathbb T_a $ and $\psi = \langle \eta, \cdot \rangle_{F_w(\mathcal H)} \in \mathbb T_b$, the kernel $\hat k(\chi, \psi) = \langle \xi, \eta \rangle_{F_w(\mathcal H)}$ is real-valued.

If we use $J(\phi_i) = \overline{\phi_i}$ to define a $^*$-structure, then for every $\langle\xi, \cdot \rangle_{F_w(\mathcal H)} \in \mathbb T_a$, $\xi^* = \xi$. Since the $^*$-structure is a conjugate-linear unitary involution, for $\hat f = \Gamma f$  and $\chi = \langle \xi, \cdot \rangle_{F_w(\mathcal H)}$ we have
$$\widehat{f^*}(\chi) = \langle \xi, f^* \rangle_{F_w(\mathcal H)} = \langle \xi^*, f^* \rangle_{F_w(\mathcal H)} = \overline{\langle \xi, f \rangle_{F_w(\mathcal H)}} = \overline{\hat{f}(\chi)}.$$

\section{Fock space embedding of observables and dynamics}
\label{sec:fock_emb}

This section addresses the remaining material on Fock space representation of observables and dynamics from \crefrange{sec:top_models}{sec:finite_dimensional_overview}. We begin in \cref{sec:kappa_ops} with a discussion on the construction of the various integral operators introduced in \cref{sec:embedding_overview}. In \cref{sec:fock_consistency}, we prove the $L^2$ convergence of the Fock space approximation $f^{(t)}_{m,\sigma,\tau}$ to the true Koopman evolution $U^t f$ claimed in \cref{thm:fock_consistency}. Our proof makes use of an approximation result for the identity on $L^2(\mu)$ by kernel integral operators, which is stated and proved in \cref{app:approx_id}. In \cref{sec:fock_consistency_d}, we verify the pointwise and $L^2$ convergence of the approximations $f^{(t)}_{m,\sigma,\tau,d}$ on finite-dimensional tori from \cref{sec:finite_dimensional_overview}. In \cref{sec:top_models_complex}, we discuss embedding approaches utilizing the complex tori $\mathbb T^\text{complex}_a$ from \cref{prop:interpolation}.

\subsection{Kernel integral operators}
\label{sec:kappa_ops}

The well-definition of the integral operators from \cref{sec:embedding_overview} is a consequence of the following lemma.

\begin{lem}
    \label{lem:kappa_ops} Let $X$ be a Hausdorff topological space, $\mathbb H$ a Hilbert space, and $\beta \colon X \to \mathbb H$ a norm-continuous function. Let also $\mu$ be a Borel probability measure with compact support $X_\mu \subseteq X$. Then, $\mathsf K \colon \mathbb H \times L^2(\mu) \to \mathbb C$ where
    \begin{displaymath}
        \mathsf K(u, v) = \int_X \langle u, \beta(x) \rangle_{\mathbb H} v(x) \, d\mu(x)
    \end{displaymath}
    is a bounded sesquilinear form. As a result, there exists a unique bounded linear map $\mathcal K \colon L^2(\mu) \to \mathbb H$ such that
    \begin{displaymath}
        \mathsf K(u, v) = \langle u, \mathcal K v \rangle_{\mathbb H}, \quad \forall u \in \mathbb H, \quad \forall v \in L^2(\mu).
    \end{displaymath}
\end{lem}

\begin{proof}
    By compactness of $X_\mu$ and norm-continuity of $\beta$, we have
    \begin{displaymath}
        \lvert \mathsf K(u, v)\rvert \leq \lVert \langle u, \beta(\cdot)\rangle_{\mathbb H}\rVert_{C(X_\mu)} \lVert v \rVert_{L^1(\mu)} \leq C \lVert u \rVert_{\mathbb H} \lVert v \rVert_{L^2(\mu)},
    \end{displaymath}
    where $C$ is the $C(X_\mu)$ norm of the function $x \mapsto \lVert \beta(x) \rVert_{\mathbb H}$. This proves boundedness of $\mathsf K$. The well-definition and boundedness of $\mathcal K \colon L^2(\mu) \to \mathbb H$ then follows from standard results in Hilbert space theory.
\end{proof}

Symbolically, we express the action of $\mathcal K \colon L^2(\mu) \to \mathbb H$ as an $\mathbb H$-valued integral,
\begin{displaymath}
    \mathcal K v = \int_X \beta(x) v(x) \, d\mu(x), \quad v \in \mathbb L^2(\mu).
\end{displaymath}

By continuity of the kernel $\kappa$, we have that $\beta_{m,\tau} \colon X \to \mathbb \mathcal H_\tau^{\vee m} \subset \fk$ with $\beta_{m,\tau}(x) = \kappa_\tau(\cdot, x)^{\vee m}$ is a norm-continuous map. Well-definition and boundedness of $\mathcal K_{m,\tau} \colon H \to \fk$ then follows from \cref{lem:kappa_ops} with $\mathbb H = \fk$ and $\beta = \beta_{m,\tau}$. Well-definition and boundedness of $\hat{\mathcal K}_{m,\tau} \colon H \to \hfk$ follows similarly by setting $\beta = \Gamma \circ \beta_{m,\tau}$ and $\mathbb H = \hfk$.

\subsection{Proof of \cref{thm:fock_consistency}}
\label{sec:fock_consistency}

For $t \in \mathbb R$ and $\tau>0$ define the continuous, bounded kernels $\kappa^{(t)}, \kappa^{(t)}_\tau \colon X \times X \to \mathbb R_{>0}$, where
\begin{displaymath}
    \kappa^{(t)}(\cdot,y) = \kappa(\Phi^t(\cdot), y), \quad \kappa^{(t)}_\tau(\cdot, y) = K_\tau \kappa^{(t)}(\cdot, y).
\end{displaymath}
Define also $g^{(t)}_{m,\sigma,\tau} = \hat g^{(t)}_{m,\sigma,\tau} \circ \hat\varphi^{(\mu)}_{\sigma,\tau} \in \tilde{\mathcal H_\tau}$, $h^{(t)}_{m,\sigma,\tau} = \hat h^{(t)}_{m,\sigma,\tau} \circ \hat\varphi^{(\mu)}_{\sigma,\tau} \in \tilde{\mathcal H_\tau}$. In addition, set $g^{(t)}_{m,\sigma}, h^{(t)}_{m,\sigma} \in \tilde{\mathcal H}_\sigma$, where
\begin{displaymath}
    g^{(t)}_{m,\sigma}(x) = \int_X (\kappa_\sigma^{(t)}(x, y))^m f(y) \, d\mu(y), \quad h^{(t)}_{m,\sigma}(x) = \int_X (\kappa_\sigma^{(t)}(x, y))^m \, d\mu(y),
\end{displaymath}
and $g^{(t)}_m, h^{(t)}_m \in C(M)$, where
\begin{displaymath}
    g^{(t)}_m(x) = \int_X (\kappa^{(t)}(x, y))^m f(y) \, d\mu(y), \quad h^{(t)}_m(x) = \int_X (\kappa^{(t)}(x, y))^m \, d\mu(y).
\end{displaymath}

Using the normalization function $d_m \colon X \to \mathbb R_{>0}$, $d_m := \int_X \kappa(\cdot, y)^m \, d\mu(y)$, we also define the kernel  $p_m \colon X \times X \to \mathbb R_{>0}$, where
\begin{displaymath}
    p_m(x, y) = \frac{\kappa(x,y)^m}{d_m(x)}.
\end{displaymath}
This kernel is Markovian with respect to $\mu$, and induces an averaging operator $P_m \colon H \to C(M)$ by
\begin{displaymath}
    P_m f = \int_X p_m(\cdot, y) f(y) \, d\mu(y).
\end{displaymath}

Our approach to proving \cref{thm:fock_consistency} will be to first show that the prediction function $f^{(t)}_{m,\sigma,\tau}$ from~\eqref{eq:f_t_tau_m} converges to $U^t P_m f$ in the iterated limit of $\sigma \to 0^+$ after $\tau \to 0^+$. We will then deduce that $U^t P_m f$ converges to $U^t f$ as $m \to \infty$ by an application of \cref{lem:kappa_ops} in \cref{app:markov}.

\subsubsection{Convergence of $g^{(t)}_{m,\sigma,\tau}$}
\label{sec:conv_g}

We begin by showing that $g^{(t)}_{m,\sigma,\tau}$ converges to $g^{(t)}_{m,\sigma}$ as $\tau \to 0^+$. For $x \in X$, let $\chi_{\sigma,\tau,x} = \hat\varphi^{(\mu)}_{\sigma,\tau}(x) \in \hat X^{(\mu)}_{\sigma,\tau}$. Using multiplicativity of $\chi_{\sigma,\tau,x}$ and $\hat U^t_\tau$, we get
\begin{align*}
    g^{(t)}_{m,\sigma,\tau}(x) &=
    \hat g^{(t)}_{m,\tau}(\chi_{\sigma,\tau,x})
    = \int_X \chi_{\sigma,\tau,x}(\hat U^t_\tau (\Gamma \kappa_\tau(\cdot, y))^m) f(y) \, d\mu(y) \\
    &= \frac{w(m)^{-2}}{\varpi_\sigma^m} \int_X \langle \varphi^{(\mu)}_\sigma(x), U^t_\tau \kappa_\tau(\cdot, y)\rangle_{\mathcal H_\tau}^m f(y) \, d\mu(y)\\
    &= \frac{w(m)^{-2}}{\varpi_\sigma^m} \int_X \langle K_\tau^* U^{-t}_\tau\varphi^{(\mu)}_\sigma(x), \iota \kappa(\cdot, y)\rangle_{H}^m f(y) \, d\mu(y).
\end{align*}
Defining $v_{x,\tau} \in C(X_\mu)$ by $v_{x,\tau}(y) = \langle K_\tau^* U^{-t}_\tau\varphi^{(\mu)}_\sigma(x), \iota \kappa(\cdot, y)\rangle_{H}$, if follows from~\eqref{eq:rkhs_conv_phi} that as $\tau \to 0^+$ $v_{x,\tau}$ converges pointwise to $v_x \in C(X_\mu)$, where
\begin{align*}
    v_x(y)
    &= \langle U^{-t} \iota \varphi^{(\mu)}_\sigma(x), \iota \kappa(\cdot, y)\rangle_H  = \langle K_\sigma^* \varphi_\sigma(x), \kappa^{(t)}(\cdot, y)\rangle_H \\
    &= \langle \varphi_\sigma(x), \kappa_\sigma^{(t)}(\cdot, y)\rangle_{\mathcal H_\sigma} = \kappa^{(t)}_\sigma(x, y).
\end{align*}
Note that we used the facts that $\iota \varphi^{(\mu)}_\sigma(x) = K_\sigma^* \varphi^{(\mu)}_\sigma(x)$ and $\varphi_\sigma(x) - \varphi^{(\mu)}_\sigma(x) \in \ker K_\sigma^*$ to obtain the second equality in the first line. Moreover, since all of $K_\tau^* \colon \mathcal H_\tau \to H$, $U^{-t}_\tau \colon \mathcal H_\tau \to \mathcal H_\tau$, and $\iota \colon C(M) \to H$ have operator norm equal to 1, the family $ \{ v_{x,\tau} \}$ is uniformly bounded over $x \in X$ and $\tau > 0$,
\begin{displaymath}
    \lVert v_{x, \tau}\rVert_{C(X_\mu)} \leq \lVert \varphi^{(\mu)}_\sigma(x) \rVert_{\mathcal H_\sigma} \lVert \kappa\rVert_{\infty} \leq \varpi_\sigma^2 \lVert \kappa \rVert_\infty.
\end{displaymath}
Correspondingly, $ \{ v_{x,\tau}^m f \}_{\tau>0}$ is a uniformly bounded family that converges pointwise to $v_x^m f$. Thus, by the dominated convergence theorem,
\begin{align*}
    \lim_{\tau\to 0^+} g^{(t)}_{m,\sigma,\tau}(x)
    &= \lim_{\tau \to 0^+} \frac{w(m)^{-2}}{\varpi_\sigma^m}\int_X v_{x,\tau}^m(y) f(y) \, d\mu(y) \\
    &= \frac{w(m)^{-2}}{\varpi_\sigma^m}\int_X \lim_{\tau \to 0^+} v_{x,\tau}^m(y) f(y) \, d\mu(y) \\
    &= \frac{w(m)^{-2}}{\varpi_\sigma^m}\int_X v_x^m(y) f(y) \, d\mu(y) \\
    &= \frac{w(m)^{-2}}{\varpi_\sigma^m}\int_X \kappa_\sigma^{(t)}(x, y) f(y) \, d\mu(y) = \frac{w(m)^{-2}}{\varpi_\sigma^m} g^{(t)}_{m,\sigma}(x).
\end{align*}

Restricting to $x,x' \in M$, we have
\begin{displaymath}
    \lVert v_{x, \tau} - v_{x',\tau}\rVert_{C(X_\mu)} \leq \left\lVert \varphi^{(\mu)}_\sigma(x) - \varphi^{(\mu)}_\sigma(x') \right\rVert_{\mathcal H_\sigma} \lVert \kappa\rVert_{\infty},
\end{displaymath}
so the convergence of $v_{x,\tau}$ to $v_x$ as $\tau \to 0^+$ is uniform over $x \in M$ by uniform continuity of $\varphi^{(\mu)}_\sigma$ on that compact set. By uniform boundedness of $ v_{x,\tau}^m f$, it follows that the convergence of $v^m_{x,\tau} f$ to $v^m_x f$ is uniform over $x \in X$, and thus the convergence of $g^{(t)}_{m,\sigma,\tau}\rvert_M$ is also uniform,
\begin{equation}
    \label{eq:conv_g_tau}
    \lim_{\tau \to 0^+} \left\lVert w(m)^2\varpi_\sigma^m g^{(t)}_{m,\sigma,\tau} - g^{(t)}_{m,\sigma} \right\rVert_{C(M)} = 0.
\end{equation}

Next, we have
\begin{displaymath}
    \iota g^{(t)}_{m,\sigma} = \int_X G_\sigma \kappa^{(t)}(\cdot, y) f(y) \, d\mu(y),
\end{displaymath}
and by strong convergence of $G_\sigma$ to the identity (\cref{prty:k4}), we deduce convergence of $g^{(t)}_{m,\sigma}$ in $H$-norm,
\begin{equation}
    \label{eq:conv_g_sigma}
    \lim_{\sigma\to 0^+} \left\lVert \iota (g^{(t)}_{m,\sigma} - g^{(t)}_m)\right\rVert_H = 0.
\end{equation}

\subsubsection{Convergence of $h^{(t)}_{m,\sigma,\tau}$}
\label{sec:conv_h}

We now examine the behavior of the normalization function $h^{(t)}_{m,\sigma,\tau}$. Proceeding analogously to \cref{sec:conv_g}, we obtain
\begin{equation}
    \label{eq:conv_h_tau_sigma}
    \lim_{\tau \to 0^+} \left\lVert w(m)^2\varpi_\sigma^m h^{(t)}_{m,\sigma,\tau} - h^{(t)}_{m,\sigma} \right\rVert_{C(M)} = 0, \quad \lim_{\sigma\to 0^+} \left\lVert \iota (h^{(t)}_{m,\sigma} - h^{(t)}_m)\right\rVert_H = 0.
\end{equation}
Let $b>0$ be a lower bound for $\kappa$ on the compact set $M \times X_\mu$. Since the kernel $k_\sigma$ is Markovian with respect to $\mu$, for every $x \in M$ and $y \in X_\mu$ we have
\begin{displaymath}
    \kappa_\sigma^{(t)}(x, y) = \int_X k_\sigma(x,\cdot) \kappa^{(t)}(\cdot, y) \, d\mu(y) \geq b.
\end{displaymath}
As a result, $h^{(t)}_{m,\sigma}(x) \geq \beta^m$ for every $x \in M$, which implies that this function has a multiplicative inverse $1/h^{(t)}_{\mu,\sigma} \in C(M)$. Moreover, the uniform convergence $w(m)^2\varpi_\sigma^m h^{(t)}_{m,\sigma,\tau} \to h^{(t)}_{m,\sigma}$ implies that for every $\delta \in (0, b^m ]$ there exists $\tau_\delta>0$ such that for all $ \tau \in (0, \tau_\delta]$ and $x \in M$, $h^{(t)}_{m,\sigma,\tau}(x) \geq (b^m - \delta) w(m)^{-2}/ \varpi_\sigma^m > 0$. The multiplicative inverse $1/h^{(t)}_{m,\sigma,\tau} \in C(M)$ thus exists for $\tau \geq \tau_\delta$, and we have
\begin{displaymath}
    \lim_{\tau\to 0^+} \left\lVert \frac{1}{w(m)^2 \varpi_\sigma^m h^{(t)}_{m,\sigma,\tau}} - \frac{1}{h^{(t)}_{m,\sigma}}\right\rVert_{C(M)} = 0.
\end{displaymath}

\subsubsection{Convergence of $f^{(t)}_{\mu,\sigma,\tau}$}

By the results of \cref{sec:conv_h}, $f^{(t)}_{m,\sigma,\tau} = g^{(t)}_{m,\sigma,\tau}/h^{(t)}_{m,\sigma,\tau}$ is well-defined as an element of $C(M)$ for small-enough $\tau$, and $f^{(t)}_{m,\sigma} = g^{(t)}_{m,\sigma}/h^{(t)}_{m,\sigma}$ lies in $C(M)$ for all $m \in \mathbb N$ and $\sigma > 0$. Moreover, we have
\begin{align*}
    f^{(t)}_{m,\sigma,\tau} - f^{(t)}_{m,\sigma}
    &= \frac{g^{(t)}_{m,\sigma,\tau}h^{(t)}_{m,\sigma} - h^{(t)}_{m,\sigma,\tau} g^{(t)}_{m,\sigma}}{h^{(t)}_{m,\sigma,\tau}h^{(t)}_{m,\sigma}} \\
    &= \frac{\left(w(m)^2 \varpi^m_\sigma g^{(t)}_{m,\sigma,\tau} - g^{(t)}_{m,\sigma}\right) h^{(t)}_{m,\sigma} + \left(w(m)^2 \varpi^m h^{(t)}_{m,\sigma,\tau} - h^{(t)}_{m,\sigma}\right) g^{(t)}_{m,\sigma}}{w(m)^2 \varpi^m_\sigma h^{(t)}_{m,\sigma,\tau}h^{(t)}_{m,\sigma}}.
\end{align*}
Therefore, by \eqref{eq:conv_g_tau}, \eqref{eq:conv_h_tau_sigma}, and uniform boundedness of $h^{(t)}_{m,\sigma,\tau}$ and $h^{(t)}_{m,\sigma}$ away from zero, we obtain
\begin{displaymath}
    \lim_{\tau\to 0^+}\left\lVert f^{(t)}_{m,\sigma,\tau} - f^{(t)}_{m,\sigma}\right\rVert_{C(M)} = 0.
\end{displaymath}
Similarly, defining $f^{(t)}_m = g^{(t)}_m / h^{(t)}_m \in C(M)$, we have
\begin{displaymath}
    f^{(t)}_{m,\sigma} - f^{(t)}_m = \frac{\left(g^{(t)}_{m,\sigma} - g^{(t)}_m\right) h^{(t)}_m + \left(h^{(t)}_{m,\sigma} - h^{(t)}_m\right) g^{(t)}_m}{h^{(t)}_{m,\sigma}h^{(t)}_m},
\end{displaymath}
and thus
\begin{displaymath}
    \lim_{\sigma\to 0^+} \left\lVert \iota ( f^{(t)}_{m,\sigma} - f^{(t)}_m)\right\rVert_H = 0
\end{displaymath}
by \eqref{eq:conv_g_sigma}, \eqref{eq:conv_h_tau_sigma}, and uniform boundedness of $h^{(t)}_{m,\sigma}$ and $h^{(t)}_m$ away from zero.

Observe now that
\begin{displaymath}
    f^{(t)}_m = \frac{\int_X \kappa(\Phi^t(\cdot), y)^mf(y)\, d\mu(y)}{\int_X \kappa(\Phi^t(\cdot), \tilde y)^m\, d\mu(\tilde y)} = U^t \frac{\int_X \kappa(\cdot, y)^mf(y)\, d\mu(y)}{\int_X \kappa(\cdot, \tilde y)^m\, d\mu(\tilde y)} = U^t P_m f,
\end{displaymath}
since the Koopman operator $U^t$ acts as an algebra homomorphism on $C(M)$. Since the kernel $\kappa$ is continuous and satisfies~\eqref{eq:kappa_decay}, it follows from \cref{lem:kappa_ops} that $P_m f$ converges as $m \to \infty$ to $f$ $\mu$-a.e.\ and in the norm of $H$. As a result, $f^{(t)}_m \to U^t f$ $\mu$-a.e.\
\begin{displaymath}
    \lim_{m\to\infty}\left\lVert \iota(f^{(t)}_m - U^t f)\right\rVert_H = 0.
\end{displaymath}

In summary, we have shown $\mu$-a.e.\ and $H$-norm convergence of $f^{(t)}_{m,\sigma,\tau}$ to $U^t f$ in the iterated limit of $m \to \infty$ after $\sigma \to 0^+$ after $\tau \to 0^+$, as claimed. This completes our proof of \cref{thm:fock_consistency}.

\subsection{Approximation on finite-dimensional tori}
\label{sec:fock_consistency_d}

Recall that for every $x \in X$ the feature vectors $\varphi_{\sigma,\tau,d}(x)$ from~\eqref{eq:feature_vec_d} lie in $(\mathcal H_\tau)_1 \subseteq (\mathcal H_\tau)_{R_w}$, and norm-converge in to $\varphi^{(\mu)}_\sigma(x)$ as $d \to \infty$. By \cref{prop:rkha_spec} and \cref{thm:quotient}, the corresponding spectral elements $\hat\varphi_{\sigma,\tau,d}(x)$ converge to $\hat\varphi_{\sigma,\tau}(x)$ in the weak-$^*$ topology of $\sigma(F_w(\mathcal H))$. This implies pointwise convergence of the prediction function,
\begin{equation}
    \label{eq:fock_pointwise_consistency_d}
    \lim_{d \to \infty}f^{(t)}_{m,\sigma,\tau,d}(x) = f^{(t)}_{m,\sigma,\tau}(x), \quad \forall x \in X,
\end{equation}
as claimed in \cref{sec:finite_dimensional_overview}.

Next, the square residual $\lvert f^{(t)}_{m,\sigma,\tau,d}(x) - f^{(t)}_{m,\sigma,\tau}(x)\rvert^2 = \lvert \hat f^{(t)}_{m,\tau}(\hat\varphi_{\sigma,\tau,d}(x)) - \hat f^{(t)}_{m,\tau}(\hat\varphi^{(\mu)}_\sigma(x)) \rvert^2$ is dominated by the constant function on $X$ equal to $\lVert \hat f^{(t)}_{m,\sigma,\tau}\rVert_{C(\sigma(F_w(\mathcal H_\tau)))}^2$. This function is $\mu$-integrable since $\mu$ is a probability measure. Hence, by the dominated convergence theorem, the pointwise convergence~\eqref{eq:fock_pointwise_consistency_d} implies norm convergence in $H = L^2(\mu)$, verifying~\eqref{eq:fock_consistency_d}.

\subsection{Topological models on complex tori}
\label{sec:top_models_complex}

The topological models of regularized Koopman dynamics from~\cref{sec:top_models} were defined on the union of tori $\mathbb S_\tau \subset \sigma(\hfk)$ that contain the image $\hat X_\tau \subset \mathbb S_\tau$ of state space under the feature map $\hat\varphi_\tau$. In this subsection, we consider an alternative formulation that uses \cref{prop:interpolation} to build a topological model that is defined on a single complex torus $\mathbb T_{\tau,a}^\text{complex} \subset \sigma(\hfk)$.

First, we choose the approximate Koopman eigenfunctions $\zeta_{j,\tau}$ as the orthonormal basis functions of $\mathcal H_\tau$, indexed by $j \in \mathbb Z$. We also choose the complex torus $\mathbb T_a^\text{complex}$ associated with this basis and an arbitrary weight vector $a = (a_j)_{j \in \mathbb Z} \in \mathbb A_w$ with strictly positive elements.

Next, consider a lifting map $\mathcal L \colon \tilde{\mathcal H}_\tau \to \fk$; i.e., a norm-continuous map satisfying $\tilde \pi \circ \mathcal L = \Id_{\tilde{\mathcal H_\tau}}$. Since $\mathcal H_\tau$ is a subspace of $\tilde{\mathcal H}_\tau$, a possible choice for $\mathcal L$ is the inclusion map $\mathcal H_\tau \hookrightarrow \tilde{\mathcal H}_\tau$. More generally, $\tilde f = \mathcal L f$ could yield a representation of $f \in \tilde{\mathcal H}_\tau$ as a product--sum of functions $g_{i,j} \in \mathcal H_\tau$, i.e., $\tilde f = \sum_{i=1}^N c_i \bigvee_{j=0}^{M_i} g_{i,j} \in \fk$ where $f = \sum_{i=1}^N c_i\prod_{j=0}^{M_i} g_{i,j}$. The embedding via $\mathcal K_m$ from \cref{sec:embedding_overview} can be viewed as an approximate version of such an $\mathcal L$ up to normalization.

With any such lifting map, define $\hat{\mathcal L}_a \colon \tilde{\mathcal H}_\tau \to \hat F_{w,a}(\mathcal H_\tau)$ as $\hat{\mathcal L}_a = \upsilon_a^{-1} \circ \Gamma \circ \mathcal L$ and let $\hat f_a = \hat{\mathcal L}_a f$. Due to \cref{prop:interpolation}, we can evolve $\hat f_a$ by the torus rotation $R_\tau^t \rvert_{\mathbb T_a^\text{complex}}$ to get $\hat f^{(t)}_{\tau, a} = \hat f_a \circ R^t_\tau \rvert_{\mathbb T_a^\text{complex}}$, and $\hat f^{(t)}_\tau = \upsilon_a \hat f^{(t)}_a$ will be the unique function in $\hat F_w(\mathcal H_\tau)$ such that $\hat f^{(t)}_\tau|_{\mathbb T_a^\text{complex}} = \hat f^{(t)}_{\tau,a}$. Hence, for any lifting of $f$, there is a unique approximate Koopman evolution $f^{(t)}_\tau = \hat\pi \hat f^{(t)}_\tau$ induced from the rotation system on $\mathbb T_a$ such that $f^{(0)}_\tau = f$.

\section{Data-driven formulation}
\label{sec:data_driven_overview}

The finite-dimensional approximation scheme described in \cref{sec:finite_dimensional_overview} is amenable to data-driven approximation using kernel methods \cites{BerryEtAl15,GiannakisEtAl15,Giannakis19,DasEtAl21,GiannakisValva24,GiannakisValva24b}. In this subsection, we give a high-level overview of the data-driven scheme, relegating a detailed convergence analysis to these references.

We consider training data $y_0, y_1, \ldots, y_{N-1}$ in a data space $Y$. The data is sampled along a dynamical trajectory $x_0, x_1, \ldots, x_{N-1} \in X$ with $x_i = \Phi^{i \,\Delta t}(x_0)$, for an initial condition $x_0$ in the forward-invariant set $M$, sampling interval $\Delta t >0$, and observation map $F \colon X \to Y$ such that $y_i = F(x_i)$. Letting $\mu_N = \sum_{i=0}^{N-1} \delta_{x_i} / N$ be the sampling measure associated with the training dataset, we have
\begin{equation}
    \label{eq:ergodic_av}
    \lim_{N\to\infty}\int_M f \, d\mu_N = \int_M f \, d\mu, \quad \forall f \in C(M),
\end{equation}
for $\mu$-a.e.\ $x \in M$ by ergodicity. For systems with so-called physical invariant measures \cite{Blank17}, \eqref{eq:ergodic_av} holds for initial conditions $x_0$ drawn from a set of positive ambient measure (e.g., Lebesgue measure) on $M$.

Based on the above, our data-driven methods replace the Hilbert space $H=L^2(\mu)$ associated with the invariant measure with finite-dimensional Hilbert spaces $H_N := L^2(\mu_N)$ associated with the sampling measures. The various kernel integral operators employed in previous sections, including $K_\tau$ and $G_\tau$ from \cref{sec:spectral_approx}, $\tilde K_\tau$ from \cref{sec:overview_rkha}, $\mathcal K_m$ and $\hat{\mathcal K}_m$ from \cref{sec:embedding_overview}, and $P_m$ from \cref{sec:fock_consistency} are then approximated by analogous (finite-rank) operators defined on $H_N$. We will denote these operators using $N$ subscripts; e.g., $G_{\tau,N} \colon H_N \to H_N $ is the data-driven analog of $G_\tau$.

In our constructions, the reproducing kernel $k_\tau$ of $\mathcal H_\tau$ is built by Markov normalization of an un-normalized kernel function (see \cref{app:markov}). In the data-driven setting, normalization leads to a data-dependent kernel $k_{\tau,N} \colon X \times X \to \mathbb R_+ $ that is continuous, strictly positive-definite, and Markov-normalized with respect to $\mu_N$. The kernels $k_{\tau,N}$ thus have associated RKHSs $\mathcal H_{\tau,N} \subset C(X)$, and the data-driven analog of $K_\tau \colon H \to \mathcal H_\tau$ is an integral operator $K_{\tau,N} \colon H_N \to \mathcal H_{\tau,N}$ associated with the data-dependent kernel $k_{\tau,N}$; that is,
\begin{displaymath}
    K_{\tau,N} f = \int_X k_{\tau,N}(\cdot, x) f(x) \, d\mu_N(x) \equiv \frac{1}{N} \sum_{i=0}^{N-1} k_{\tau,N}(\cdot, x_i)f(x_i).
\end{displaymath}
In the large data limit, $N \to \infty$, the restrictions of $k_{\tau,N}$ to $M \times M$ converge uniformly to $k_\tau$ for every initial condition $x_0 \in M$ satisfying~\eqref{eq:ergodic_av}.

Using an eigendecomposition of $G_{\tau,N} = K_{\tau,N}^* K_{\tau,N}$,
\begin{displaymath}
    G_{\tau,N} \phi_{j,N} = \lambda_{j,\tau,N} \phi_{j,N}, \quad j \in \mathbb N_0,
\end{displaymath}
we have, by analogous properties to \ref{prty:k1}--\ref{prty:k4}, that the eigenvalues $\lambda_{j,\tau,N}$ lie in the interval $(0, 1]$ and can be ordered as $1 = \lambda_{0,\tau,N} > \lambda_{1, \tau, N} \geq \lambda_{2,\tau,N} \geq \cdots \lambda_{N-1,\tau,N} > 0$. Moreover, the corresponding eigenvectors $\phi_{j,N}$ form an orthonormal basis of $H_N$ that does not depend on $\tau$ by a semigroup property analogous to \ref{prty:k3}. These eigenvectors induce orthonormal vectors $\psi_{j,\tau,N} \in \mathcal H_\tau$, given by
\begin{displaymath}
    \psi_{j,\tau,N} = \frac{1}{\sqrt{\lambda_{j,\tau,N}}} K_{\tau,N} \phi_{j,N}.
\end{displaymath}

For $l \in \mathbb N$ such that $2l + 1 \leq N$, define the subspace $\Psi_{\tau,l,N} = \spn \{ \psi_{0,\tau,N}, \ldots, \psi_{2l,\tau,N} \} \subseteq \mathcal H_{\tau, N}$. We approximate the regularized Koopman generator $W_\tau \colon D(W_\tau) \to \mathcal H_{\tau,N}$ by skew-adjoint, finite-rank operators $W_{\tau,l,N} \colon \mathcal H_{\tau,N} \to \mathcal H_{\tau,N}$ with $\ran W_{\tau,l,N} \subseteq \Psi_{\tau,L,N}$. These operators have associated eigendecomposition
\begin{displaymath}
    W_{\tau,l,N} \zeta_{j,\tau,l,N} = i \omega_{j,\tau,l,N} \zeta_{j,\tau,l,N}, \quad j \in \{ -l, \ldots, l \},
\end{displaymath}
where $\omega_{j,\tau,l,N}$ are real eigenfrequencies satisfying $\omega_{-j,\tau,l,N} = - \omega_{j,\tau,l,N}$, and $\zeta_{j,\tau,l,N} \in \mathcal H_{\tau,l,N}$ are normalized eigenvectors that form complex-conjugate pairs, $\zeta_{-j,\tau,l,N} = \zeta_{j,\tau,l,N}^*$, with $\zeta_{0,\tau,l,N} = 1_X$. Computationally, we obtain the eigenpairs $(\omega_{j,\tau,l,N}, \zeta_{j,\tau,l,N})$ by solving an $l \times l$ matrix eigenvalue problem giving eigenpairs $(\omega_{j,\tau,l,N}, \xi_{j,\tau,l,N})$ for an approximate generator $V_{\tau,l,N}\colon H_N \to H_N$, and then mapping the eigenvectors $\xi_{j,\tau,l,N} \in H_N$ to $\zeta_{j,\tau,l,N} = T_{\tau,N} \xi_{j,\tau,l,N}$ using the isometry $T_{\tau, N} : H_N \to \mathcal H_{\tau,N}$ defined via the polar decomposition of $K_{\tau,N}$ as in~\eqref{eq:polar_decomp}. We order the computed eigenpairs in order of increasing Dirichlet energy as in \cref{sec:spectral_approx}. Further details can be found in \cref{app:spectral_approx}. The results converge to the eigendecomposition  of $W_\tau$ in an iterated limit of $l\to\infty$ after $N\to \infty$.

Having obtained $W_{\tau,l,N}$ and its eigendecomposition, the rest of the data-driven scheme proceeds in a very similar manner to what was described in \crefrange{sec:overview_rkha}{sec:finite_dimensional_overview}:
\begin{itemize}[wide]
    \item We build the weighted symmetric Fock space $\fkn$ analogously to $F_w(\mathcal H_\tau)$, and lift the unitaries $U^t_{\tau,l,N} \colon \mathcal H_{\tau,N} \to \mathcal H_{\tau,N}$ to unitaries $\tilde U^t_{\tau,l,N} \colon \fkn \to \fkn$ that act multiplicatively with respect to the symmetric tensor product. Embedded in the spectrum $\sigma(F_w(\mathcal H_{\tau,N}))$ is a union of tori $\mathbb S_{\tau,N}$ analogous to $\mathbb S_\tau$, in which the state space $X$ is mapped by means of a feature map $\hat \varphi^{(\mu_N)}_{\sigma,\tau} \colon X \to \sigma(\fkn)$ analogous to $\hat \varphi^{(\mu)}_{\sigma,\tau} \colon X \to \sigma(\fkn)$.
    \item We define a rotation system $R^t_{\tau,l,N} \colon \sigma(\fkn) \to \sigma(\fkn)$ on the spectrum of $\fkn$ that is induced by the unitaries $\tilde U^t_{\tau,l,N}$, i.e., $R^t_{\tau, l, N}(\chi) = \chi \circ \tilde U^t_{\tau,l,N} $. This will serve as an approximation of the rotation system $R^t_\tau$ on $\sigma(\fk)$.
    \item Letting $\hfkn \subset C(\sigma(\fkn))$ be the the Gelfand image of $\fkn$ as an RKHA on the spectrum of $\fkn$, we embed observables $f \in C(M)$ into elements $\hat g_{m,\tau,N} \in \hfkn$, $m \in \mathbb N$, by means of data-driven integral operators $\hat{\mathcal K}_{m,N} \colon H_N \to \hfkn$ analogous to $\hat{\mathcal K}_m \colon H_N \to \hfk$. The restrictions of these functions to $d$-dimensional tori in $\mathbb S_{\tau,N}$, $d \leq \lfloor (N -1) / 2 \rfloor $, are degree-$m$ polynomials in the coordinates $z_1, \ldots, z_d$ and their complex conjugates. These polynomials evolve under the rotation system $R^t_{\tau,l,N}$ analogously to $\hat g^{(t)}_{m,\tau}$ in~\eqref{eq:hat_gh_t_m_tau_d}.
    \item Proceeding similarly to the construction of $f^{(t)}_{m,\sigma,\tau,d}$ in~\eqref{eq:f_t_m_tau_d}, we approximate the Koopman evolution $U^t f$ by means of the function
        \begin{equation}
            \label{eq:f_t_m_tau_d_n}
            f^{(t)}_{m,\sigma,\tau,d,l,N} = \frac{g^{(t)}_{m,\sigma,\tau,l,N}}{h^{(t)}_{m,\sigma,\tau,l,N}} = \frac{\sum_{j \in \mathbb J_{d,m}} \binom{m}{j_{-d} \cdots j_d} C_{j,\tau,d,l,N}^{(g)} \prod_{r=1}^d e^{i (j_{-r} - j_r)\omega_{r,\tau,l,N}t} \zeta_{r,\sigma,\tau,l,N}^{j_{-r} - j_r}}{\sum_{j \in \mathbb J_{d,m}} \binom{m}{j_{-d} \cdots j_d} C_{j,\tau,d,l,N}^{(h)} \prod_{r=1}^d e^{i (j_{-r} - j_r)\omega_{r,\tau,l,N}t} \zeta_{r,\sigma,\tau,l,N}^{j_{-r} - j_r}} \in C(M).
        \end{equation}
        Here, $g^{(t)}_{m,\sigma,\tau,l,N}$ and $h^{(t)}_{m,\sigma,\tau,l,N}$ are elements of the RKHA $\tilde{\mathcal H}_{\tau,N}$ on $X$ defined analogously to $\tilde{\mathcal H}_\tau$ using the kernel $k_{\tau,N}$. Moreover, $C_{j,\tau,d,l,N}^{(g)}$ and $C_{j,\tau,d,l,N}^{(h)}$ are moment coefficients computed using similar formulas to~\eqref{eq:moms}, replacing the invariant measure $\mu$ by the sampling measure $\mu_N$ and the regularized generator eigenfunctions $\zeta_{j,\tau}$ by their data-driven counterparts $\zeta_{j,\tau,l,N}$. Meanwhile, $\zeta_{r,\sigma,\tau,l,N} \in \mathcal H_{\sigma+\frac{\tau}{2},N}$ are kernel-smoothed versions of $\zeta_{r,\tau,l,N}$ defined similarly to~\eqref{eq:zeta_sigma}.
\end{itemize}
The final approximation $f^{(t)}_{m,\sigma,\tau,d,l,N}$ converges to $U^t f$ in the norm of $H$ in the iterated limits of $\lim_{m\to\infty}\lim_{\sigma\to 0^+}\lim_{\tau\to 0^+}\lim_{d\to\infty}\lim_{l\to\infty}\lim_{N\to\infty}$, taken in that order.

\section{Numerical experiments}
\label{sec:experiments}

We apply the second quantization framework presented in \crefrange{sec:prelims}{sec:data_driven_overview} to two measure-preserving dynamical systems: a Stepanoff flow on $\mathbb T^2$ \cite{Oxtoby53} and the L63 system on $\mathbb R^3$. Both systems exhibit dynamical complexity from an operator-theoretic standpoint: The Stepanoff flow is a topologically weak-mixing system that is characterized by absence of continuous nonconstant Koopman functions, and the L63 system is measure-theoretically mixing \cite{LuzzattoEtAl05} with an associated continuous spectrum of the Koopman operator on $H$. A primary difference between the two examples is that the ergodic invariant measure $\mu$ is the Haar probability measure on $\mathbb T^2$ (and therefore has a smooth density in local coordinates), whereas in the L63 case $\mu$ is an SRB measure supported on a fractal set of zero Lebesgue measure (the Lorenz attractor) \cite{Tucker99}.

Both systems satisfy the assumptions laid out in \cref{sec:dynamical_system}: For the Stepanoff flow, the state space $X$, forward-invariant compact manifold $M$, and support of the invariant measure $\mu$ all coincide,  $ \supp\mu = M = X = \mathbb T^2$. For the L63 system, we have $\supp \mu \subset M \subset X = \mathbb R^3$ where $\supp \mu$ is the Lorenz attractor and $M$ an absorbing ball containing the attractor \cite{LawEtAl14}.

The Stepanoff flow was also employed in numerical experiments in \cite{GiannakisEtAl24} using the tensor network approximation scheme summarized in \cref{sec:tensor_network}. Koopman spectral computations for the two systems using the methods outlined in \cref{sec:spectral_approx,app:markov,app:spectral_approx} were performed in \cite{GiannakisValva24b}.

\subsection{Experimental procedure}
\label{sec:exp_procedure}

We generate (unobserved) states $x_0, x_1, \ldots, x_{N-1} \in X$ and associated training data $y_n = F(x_n)$ in a Euclidean data space $Y$ using a smooth observation map $F \colon X \to Y$. In the case of the Stepanoff flow, the states $x_n$ are sampled on a uniform grid on $\mathbb T^2$, and $F$ is the flat embedding of the 2-torus into $Y = \mathbb R^4$,
\begin{displaymath}
    F(x) = (\cos x^1, \sin x^1, \cos x^2, \sin x^2), \quad x = (x^1, x^2), \quad x^1,x^2 \in [0, 2\pi).
\end{displaymath}
The observation map for the L63 system is the identity map on $X = \mathbb R^3$.

We choose a prediction observable $f \in C(X)$ for each system and consider that the values $f(x_n) \in \mathbb R$ on the training states are known. The prediction observable for the Stepanoff experiments is a von Mises probability density function,
\begin{equation}
    \label{eq:von_mises}
    f(x) = \frac{e^{-\gamma (\cos x^1 + \cos x^2)}}{I_0^2(\gamma)}, \quad x = (x^1, x^2) \in [0, 2\pi)^2,
\end{equation}
where $\gamma>0$ is a concentration parameter and $I_0 \in C^\infty(\mathbb R)$ the modified Bessel function of the first kind of order 0. For the L63 experiments, $f$ is the component function for the first component component of the state vector,
\begin{equation}
    \label{eq:l63_x1}
    f(x) = x^1, \quad x = (x^1, x^2, x^3) \in \mathbb R^3.
\end{equation}

We compute approximate Koopman eigenfunctions $\zeta_{j,\tau,l,N} \in \mathcal H_{\tau,N}$ and corresponding eigenfrequencies $\omega_{j,\tau,l,N} \in \mathbb R$ by processing the training data $y_n$ with the Koopman spectral approximation technique from \cite{GiannakisValva24b}, summarized in \cref{sec:spectral_approx,sec:data_driven_overview,app:spectral_approx}. Note that this scheme makes use of equations of motion through the dynamical vector field $\vec V$ (given in~\eqref{eq:vec_stepanoff} and \eqref{eq:vec_l63} below for the two systems). In both of the Stepanoff and L63 experiments, we build the RKHS $\mathcal H_{\tau,N}$ using a bistochastic kernel normalization, described in \cref{app:markov}. The associated unitary evolution $U^t_{\tau,l,N}$ defines a rotation system on the tori $\mathbb T_{\sigma,\tau,d,x} \subset \sigma(\fkn)$ that we will use for our dynamical prediction experiments. In all cases, we use the torus dimension $d = 50$. We select the tori $\mathbb T_{\sigma,\tau,d,x}$ using the complex-conjugate pairs of eigenfunctions $\zeta_{j,\tau,l,N}$ with the $d$ largest $L^2$ projection coefficient amplitudes $c_j := \lvert \langle \xi_{j,z,\tau,l,N}, f \rangle_{H_N}\rvert$ with respect to the eigenbasis $\xi_{j,z,\tau,l,N}$ (that is, the sequence $a^{(d)}_{\sigma,\tau}(x)$ parameterizing $\mathbb T_{\sigma,\tau,d,x} \equiv T_{\tau, a_{\sigma,\tau}^{(d)}(x)}$ has nonzero elements only at the positions corresponding to the constant eigenfunction $\zeta_{0,\tau,l,N} = 1_X$ and eigenfunctions $\zeta_{j,\tau,k,N}$ with the 50 largest corresponding values of $c_j$, for a total of 101 nonzero elements). Note that the weight function $w$ cancels from our final formula~\eqref{eq:f_t_m_tau_d_n} for the time-evolved prediction observable, so we do not need to specify it explicitly for numerical computations.

Next, to embed the prediction observable $f$ into the Fock space $\fkn$, we use the grading parameter $m=4$ and a Gaussian smoothing kernel $\kappa$ from~\eqref{eq:gaussian_kernel} induced by the Euclidean metric $d(x,x') = \lVert F(x) - F(x')\rVert_2$ on $X$. We compute smoothed eigenfunctions $\varrho_{j,\tau,l,N}$ and the associated moments $C^{(g)}_{j,\tau,d,l,N}$ and $C^{(h)}_{j,\tau,d,l,N}$ using data-driven analogs of~\eqref{eq:smoothed_eigenfuncs} and~\eqref{eq:moms}. Note that the moment computations utilize the values $\zeta_{j,\tau,l,N}(x_n)$ and $f(x_n)$ of the eigenfunctions and prediction observable, respectively, on the states $x_n$. For our chosen parameter values of $d = 50$ and $m = 4$ the number of moments is $\binom{m + 2d}{2d} = \text{4,598,126}$. We then build the prediction function $f^{(t)}_{m,\sigma,\tau,d,l,N}$ using \eqref{eq:f_t_m_tau_d_n}. A listing of the dataset attributes and numerical parameters used in the Stepanoff and L63 experiments is included in \cref{tab:params}.

\begin{table}
    \centering
    \caption{Dataset attributes and numerical parameters used in the Stepanoff flow and L63 experiments.}
    \label{tab:params}
    \begin{tabular}{lcc}
        & Stepanoff flow & Lorenz 63 system \\
        \hline
        System parameters & $\alpha = \sqrt{20}$ & $(\beta, \rho, \sigma) = (8/3, 28, 10)$ \\
        Dimension of state space $X$ & 2 & 3 \\
        Dimension of data space $Y$ & 4 & 3 \\
        Number of training samples $N$ & $256 \times 256 = \text{65,536}$ & 80,000 \\
        Training sampling timestep $\Delta t$ & N/A & 5.0 \\
        Prediction observable $f$ & von Mises, $\gamma = 1$ & first state vector component \\
        Regularization parameter $\sigma$ & $10^{-4}$ & $2 \times 10^{-6}$ \\
        Regularization parameter $\tau$ & $10^{-4}$ & $5 \times 10^{-7}$ \\
        Koopman approximation space dimension $l$ & 4096 & 2048 \\
        Torus dimension $d$ & 50 & 50 \\
        Smoothing kernel bandwidth $\varepsilon$ & 0.05 & 0.1 \\
        Fock space grading $m$ & 4 & 4 \\
        \hline
    \end{tabular}
\end{table}

We test the accuracy of the scheme by comparing $f^{(t)}_{m,\sigma,\tau,d,l,N}(x_n)$ with a high-fidelity numerical approximation of the true Koopman evolution $U^t f(x_n) = f(\Phi^t(x_n))$ based on a ordinary differential equation solver to approximate the flow map $\Phi^t(x_n)$. In addition, we compare the Fock space approach with a conventional RKHS approximation \cite{DasEtAl21}, utilizing the same eigenfunctions and eigenfrequencies as those generating the $d$-dimensional rotation system on $\mathbb T_{\sigma,\tau,d,x}$. Hereafter, we refer to this method as classical approximation. The associated prediction function $f^{(t)}_{\text{cl},\tau,d,l,N} \in \mathcal H_\tau$ is defined as
\begin{displaymath}
    f^{(t)}_{\text{cl},\tau,d,l,N} = \sum_{j=0}^{2d} C^{(f)}_{j,\tau,l,N} e^{i\omega_{j,\tau,l,N}} \zeta_{j,\tau,l,N}, \quad C^{(f)}_{j,\tau,l,N} = \int_X \overline{\zeta_{j,\tau,l,N}} f \, d\mu_N.
\end{displaymath}
The convergence properties of $f^{(t)}_{\text{cl},\tau,d,l,N}$ are analogous to those of $f^{(t)}_{m,\sigma,\tau,d,l,N}$; i.e., we have convergence in the norm of $H$ in the iterated limit $\lim_{\tau\to 0^+} \lim_{d\to\infty} \lim_{l \to \infty} \lim_{N\to \infty}$; see \cite{DasEtAl21} for further details.

Let $f^{(t)}$ stand for either $f^{(t)}_{\text{cl}, \tau, d, l, N}$ or $f^{(t)}_{m,\sigma,\tau,d,l,N}$. To quantitatively assess prediction skill, we compute normalized root mean square error (RMSE) and anomaly correlation (AC) scores, defined as
\begin{align*}
    \text{RMSE}(t) = \frac{\lVert f^{(t)} - U^t f\rVert_{H_N}}{\lVert U^t f \rVert_{H_N}}, \quad \text{AC}(t) = \frac{\langle g^{(t)}, h^{(t)}\rangle_{H_N}}{\lVert g^{(t)}\rVert_{H_N} \lVert h^{(t)}\rVert_{H_N}}
\end{align*}
where $g^{(t)}, h^{(t)} \in H_N$ are the predicted and true anomalies relative to the empirical mean,
\begin{displaymath}
    g^{(t)} = f^{(t)} - \int_X f^{(t)} \, d\mu_N, \quad h^{(t)} = U^t f - \int_X U^t f \, d\mu_N,
\end{displaymath}
respectively. By construction, $\text{AC}(t)$ lies in the interval $[0,1]$. Values of $\text{RMSE}(t)$ close to 0 and $\text{AC}(t)$ close to 1 indicate high prediction skill.

For the remainder of \cref{sec:experiments}, we will use the shorthand notations $f^{(t)}_\text{cl} \equiv f^{(t)}_{\text{cl},\tau,d,l,N}$ and $f^{(t)}_\text{Fock} \equiv f^{(t)}_{m,\sigma,\tau,d,l,N}$ for the classical and Fock space approximations, respectively.

\subsection{Stepanoff flow}
\label{sec:stepanoff}

The dynamical vector field $\vec V\colon \mathbb T^2  \to T \mathbb T^2$ generating the Stepanoff flow has the coordinate representation $\vec V(x) = (V^1(x), V^2(x))$, where
\begin{equation}
    \label{eq:vec_stepanoff}
    V^1(x) = V^2(x) + (1- \alpha) (1 - \cos x^2), \quad V^2(x) = \alpha(1 - \cos(x^1 - x^2)),
\end{equation}
$x = (x^1, x^2) \in [0, 2\pi)^2$, and $\alpha$ is a real parameter. This vector field has zero divergence with respect to the Haar measure $\mu$,
\begin{displaymath}
    \divr_\mu \vec V = \frac{\partial V^1}{\partial x^1} + \frac{\partial V^2}{\partial x^2} = 0,
\end{displaymath}
which implies that $\mu$ is an invariant measure under the associated flow $\Phi^t$. In addition, the system has a fixed point at $x=0$, $\vec V(0) = 0$. In \cite{Oxtoby53} it is shown that the normalized Haar measure is the unique invariant Borel probability measure of this flow that assigns measure 0 to the singleton set $ \{ 0 \} \subset \mathbb T^2$ containing the fixed point.

Since any continuous, nonconstant Koopman eigenfunction of an ergodic flow induces a semiconjugacy with a circle rotation of nonzero frequency, the existence of the fixed point at $x=0$ implies that the system has no continuous Koopman eigenfunctions; i.e., it is topologically weak-mixing. In fact, \cite{Oxtoby53} shows that when $\alpha$ is irrational the Stepanoff flow is topologically conjugate to a time-reparameterized irrational rotation on $\mathbb T^2$ with frequency parameters $(1, \alpha)$, for a time-reparameterization function that vanishes at $x=0$. While, to our knowledge, there are no results in the literature on the measure-theoretic mixing properties of Stepanoff flows, the existence of fixed points is a necessary condition for a smooth ergodic flow on $\mathbb T^2$ to be (strong) mixing \cite{Kocergin75} and the vector field~\eqref{eq:vec_stepanoff} meets at least that requirement.

Here, we work with the parameter value $\alpha = \sqrt{20}$. The resulting dynamical vector field is depicted as a quiver plot in \eqref{eq:vec_stepanoff}. As prediction observable, we consider the von Mises density function $f$ from \eqref{eq:von_mises} with concentration parameter $\gamma=1$. This is shown as a heat map on $\mathbb T^2$ in the top-left panel of \cref{fig:evo_stepanoff}. Note that $f$ has a peak at $(\pi, \pi)$ which is well-separated from the fixed point at $(0, 0)$. The remaining panels in the left-hand column of \cref{fig:evo_stepanoff} show the Koopman evolution $U^t f$ at times $t \in \{ 0.5, 1, 2, 4 \} $. As $t$ increases, the initially radially symmetric structure of $f$ becomes stretched and folded, developing sharp gradients that are suggestive of sensitive dependence on initial conditions. The fact that $f$ is strictly positive makes this observable suitable for testing the ability of approximation schemes to preserve positivity.

\begin{figure}
    \centering
    \includegraphics[width=0.5\textwidth]{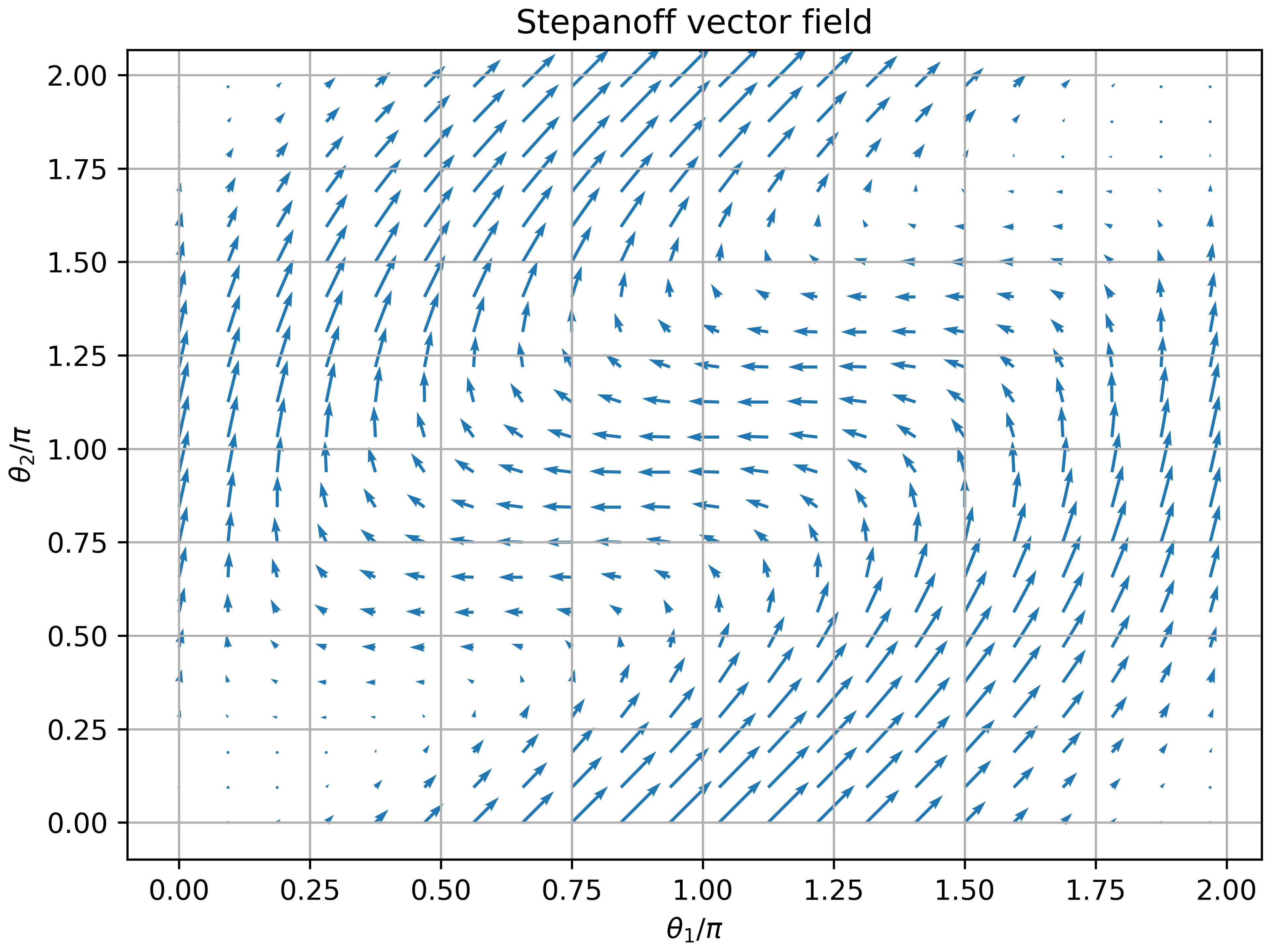}
    \caption{Generating vector field~\eqref{eq:vec_stepanoff} of the Stepanoff flow for $\alpha = \sqrt{20}$.}
    \label{fig:vec_stepanoff}
\end{figure}

\begin{figure}
    \centering
    \includegraphics[width=.93\linewidth]{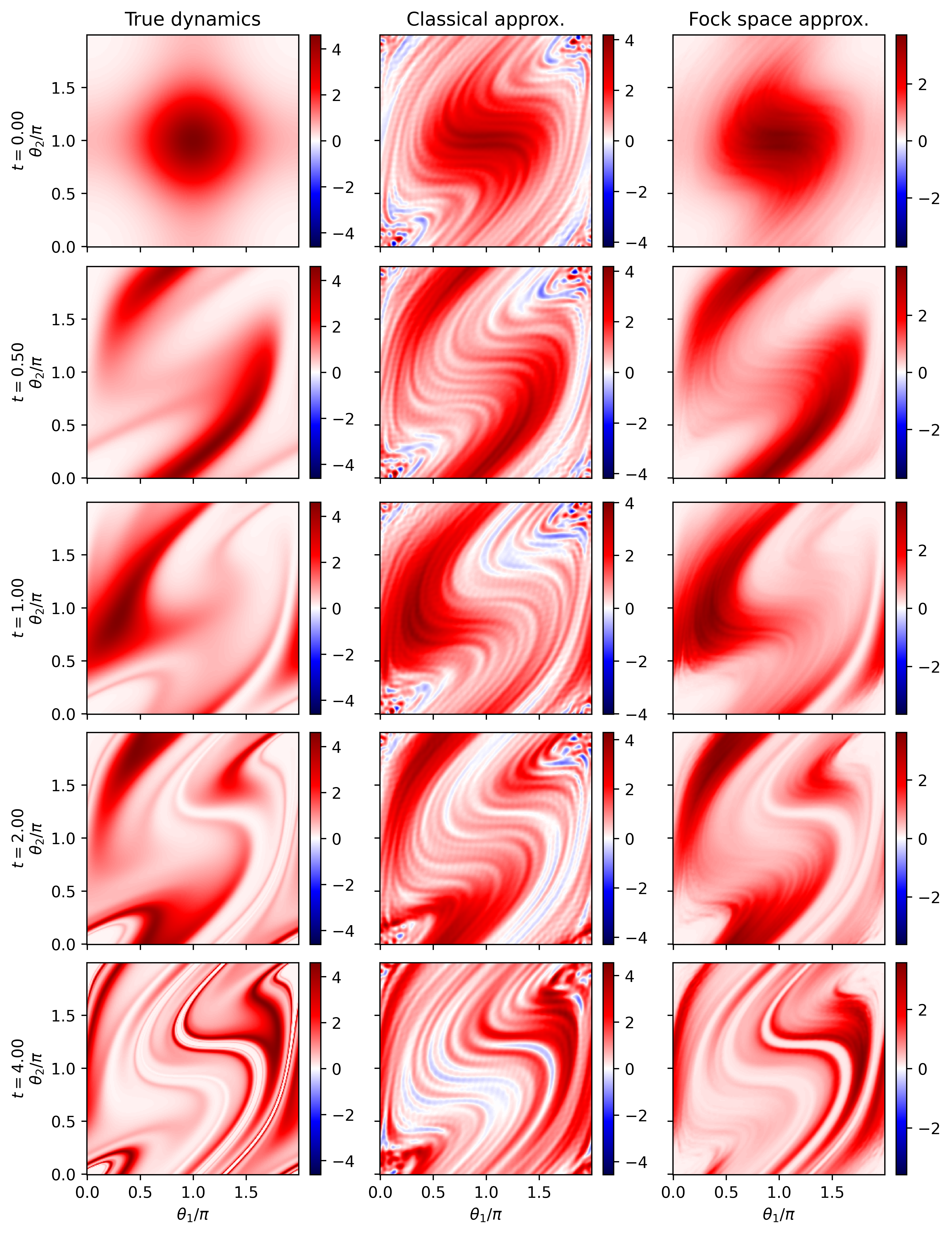}
    \caption{Time evolution of the von Mises density function from~\eqref{eq:von_mises} with $\gamma = 1$ under the Stepanoff flow ($U^{t}f$; left column), the classical approximation based on $2d+1$ eigenfunctions ($f^{(t)}_\text{cl}$; center column), and the Fock space approximation for torus dimension $d = 50$ and degree $m = 4$ ($f^{(t)}_\text{Fock}$; right column). Rows from top to bottom show snapshots at the evolution times $t=0, 0.5, 1, 2, 4$, respectively.}
    \label{fig:evo_stepanoff}
\end{figure}

Our training data is sampled on a uniform $256 \times 256$ grid $ \{ x_n \} \subset \mathbb T^2$, amounting to a total of $N = \text{65,536}$ points. We build the RKHS $\mathcal H_{\tau,N}$ and the $l$-dimensional approximation space $\Psi_{\tau,l,N} \subset \mathcal H_{\tau,N}$ using the regularization parameter $\tau = 10^{-4}$ and dimension parameter $l = 4096$ (see \cref{sec:data_driven_overview}). Representative eigenfrequencies $\omega_{j,\tau,l,N}$ and the corresponding eigenfunctions are shown in \cref{fig:spec,fig:evecs_stepanoff}, respectively. Notice the characteristic ``S-shaped'' pattern of the eigenfunction level sets that is broadly aligned with the dynamical vector field (see again \cref{fig:vec_stepanoff}). Also noteworthy in the eigenfunction plots is the oscillatory nature of the eigenfunctions near the fixed point. This behavior is qualitatively consistent with the slowing down of dynamical trajectories near the fixed point, which needs to be compensated by small-scale spatial oscillations in order to maintain a consistent evolution with the periodic phase evolution implied by the corresponding eigenfrequency, $\zeta_{j,\tau,l,N}(\Phi^t(x)) \approx e^{i\omega_{j,\tau,l,N}} \zeta_{j,\tau,l,N}(x)$.

\begin{figure}
    \centering
    \includegraphics[width=.48\linewidth]{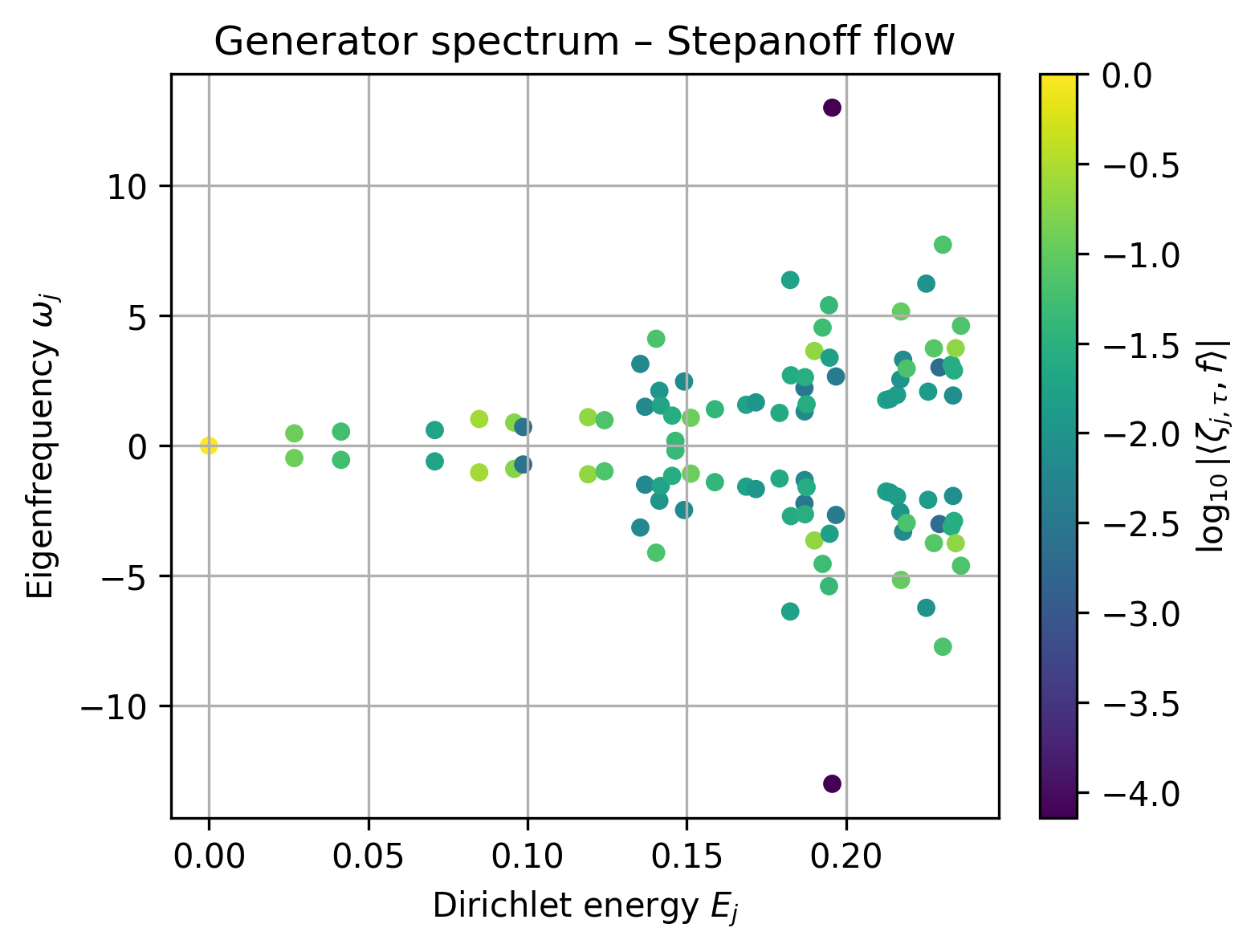}
    \hfill
    \includegraphics[width=.48\linewidth]{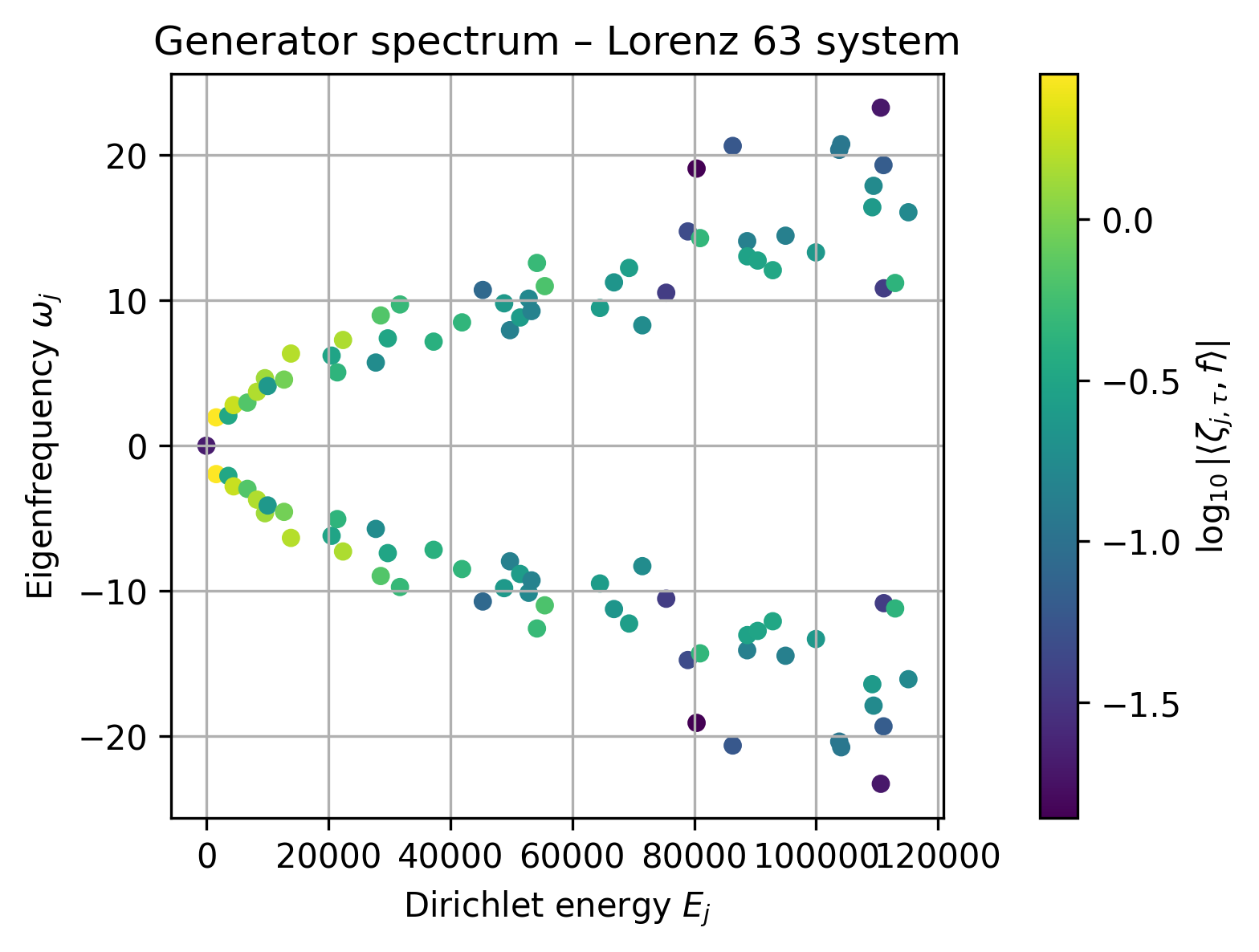}
    \caption{Spectrum of the regularized generator $W_{\tau,l,N}$ for the Stepanoff flow (left) and L63 system (right). Eigenfrequencies $\omega_{j,\tau,l,N}$ are plotted versus their corresponding Dirichlet energies from the RKHS $\mathcal H_\tau$. The plotted points are colored by the logarithms of the expansion coefficient amplitudes $\lvert \langle f, \xi_{j,\tau,l,N}\rangle_{H_N}\rvert$ of the prediction observable $f$ in the eigenbasis $\xi_{j,\tau,l,N}$.}
    \label{fig:spec}
\end{figure}

\begin{figure}
    \centering
    \includegraphics[width=\linewidth]{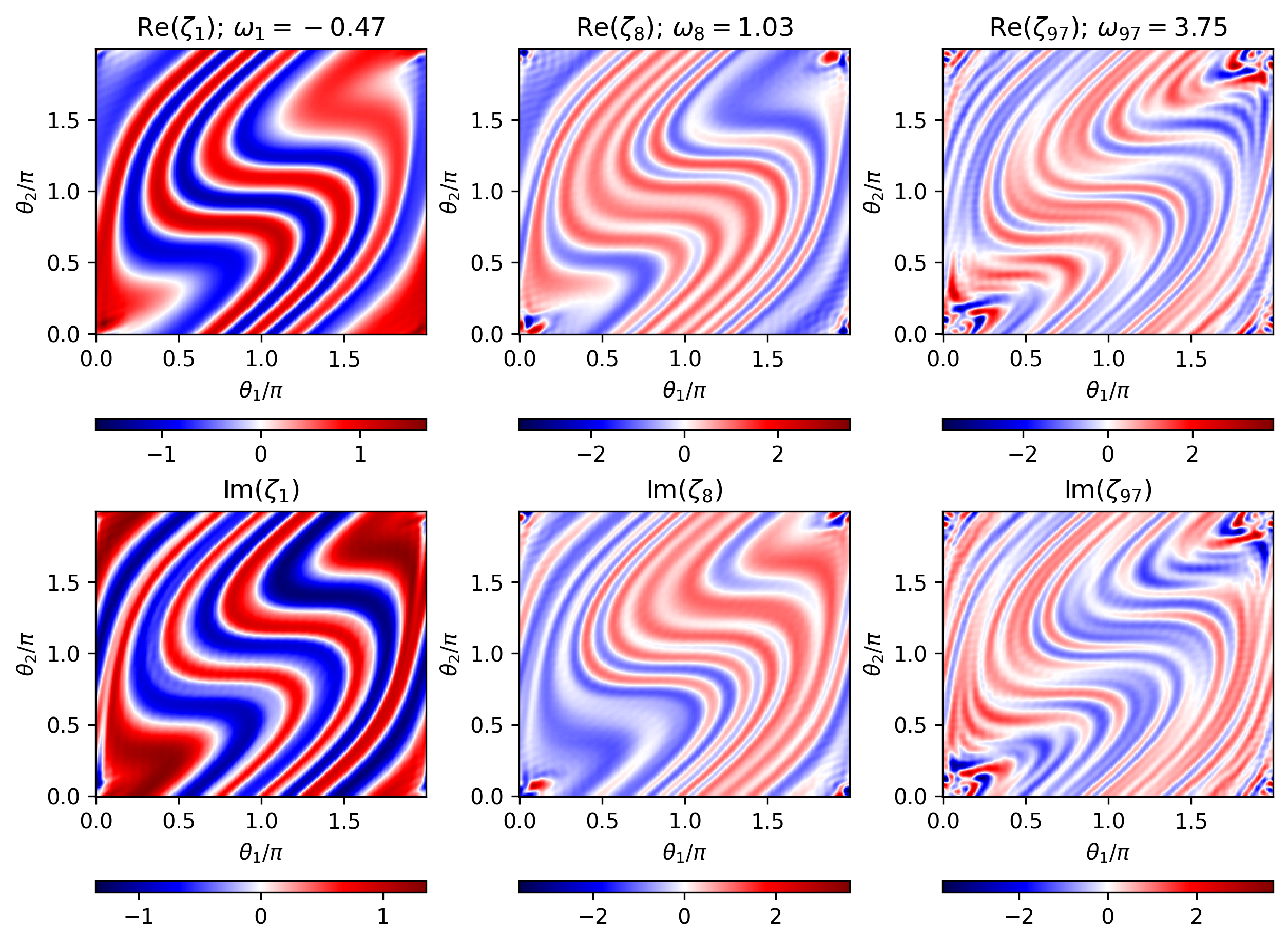}
    \caption{Real and imaginary parts of representative eigenfunctions $\zeta_{j,\tau,l,N}$ for the Stepanoff flow. The eigenfunctions shown have index $j=1, 8, 97$ with respect to the Dirichlet energy ordering, and are members of the complex-conjugate (nonconstant) eigenfunction pairs with the 7th, 1st, 5th largest projection amplitudes $\lvert \langle \xi_{j,\tau,l,N}, f \rangle_{H_N}\rvert$, respectively.}
    \label{fig:evecs_stepanoff}
\end{figure}

Next, in the center column of \cref{fig:evo_stepanoff} we show the evolution $f^{(t)}_\text{cl}$ under the classical Koopman approximation obtained for $d = 50$. The evolution times are the same as those shown for the true evolution in the left-hand column (i.e., $ t \in \{ 0, 0.5, 1, 2, 4 \}$) and $f^{(t)}_\text{cl}$ is plotted at the gridpoints $x_n$. \Cref{fig:err_stepanoff}(center column) shows the approximation error of $f^{(t)}_\text{cl}$ relative to the true evolution $U^t f$ at the gridpoints.

\begin{figure}
    \centering
    \includegraphics[width=.93\linewidth]{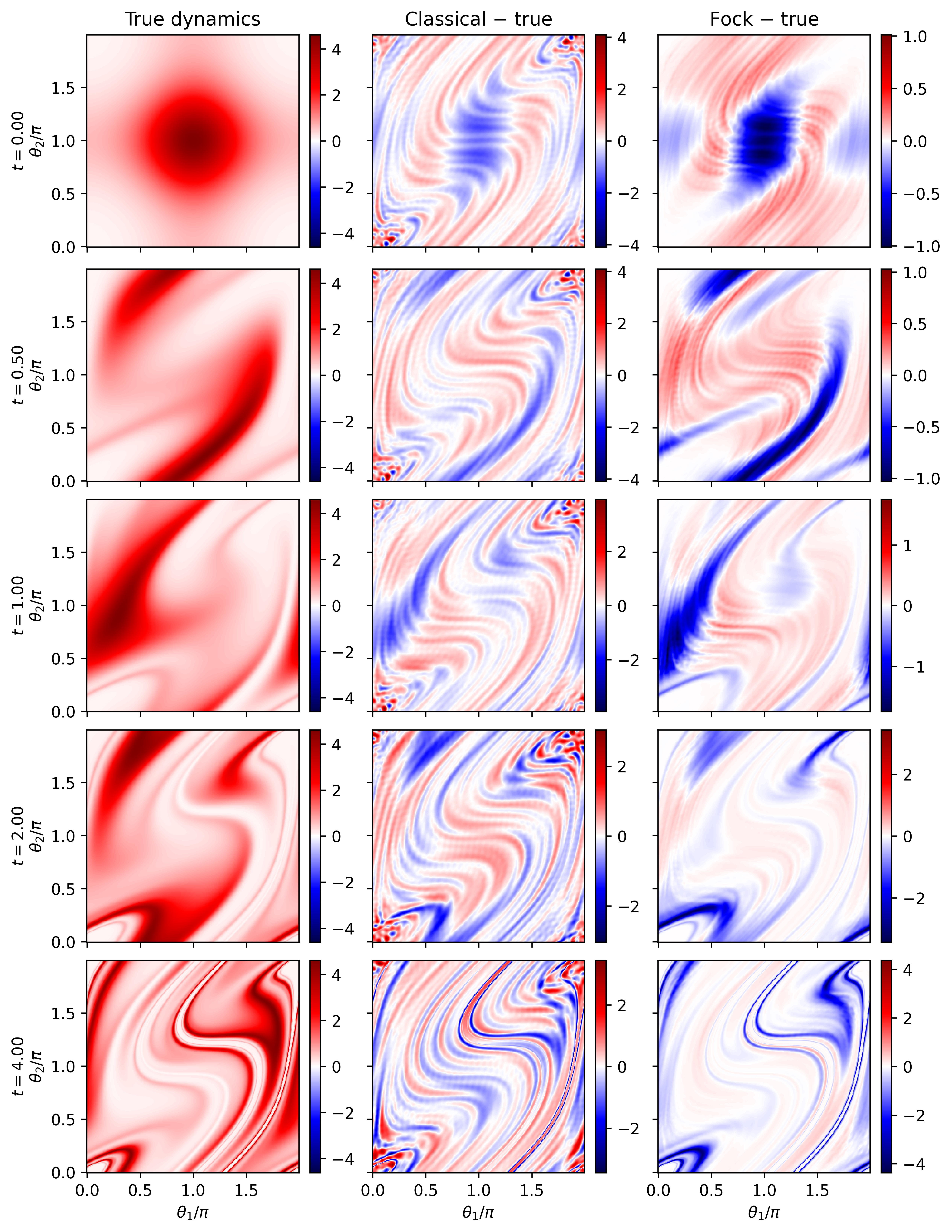}
    \caption{Error in the classical and Fock space approximations from \cref{fig:evo_stepanoff} (center and right columns, respectively) relative to the true Stepanoff evolution. The true evolution is plotted in the left column for reference.}
    \label{fig:err_stepanoff}
\end{figure}

Starting from time $t = 0$, it is clear that $f^{(0)}_\text{cl}$ carries imprints of the characteristic S-shaped spatial structure of the eigenfunctions $\zeta_{j,\tau,d,l,N}$ employed in the approximation (see \cref{fig:evecs_stepanoff}), and as a result provides a poor reconstruction of the radial symmetry of the von Mises density function. Moreover, $f^{(0)}_\text{cl}$ exhibits strongly oscillatory behavior near the fixed point with significant departures to negative values. This is a manifestation of the fact that, being an orthogonal subspace projection method, the classical approximation is not positivity-preserving. At later times, the classical approximation appears to track some of the large-scale qualitative features of the true evolution such as the large-amplitude bands seen at $t = 0.5, 1$ due to stretching of the originally radial von Mises density function by the Stepanoff flow. However, $f^{(t)}_\text{cl}$ continues to exhibit significant oscillatory behavior and departures to negative values, particularly in regions near the fixed point.

We now consider the Fock space approximation $f^{(t)}_\text{Fock}$ for torus dimension $d=50$, grading $m=4$, and smoothing kernel bandwidth parameter $\varepsilon = 0.05$. As noted in \cref{sec:exp_procedure}, with this choice of parameters the approximation space dimension is amplified to \text{4,598,126} so we expect higher reconstruction accuracy than the 101-dimensional classical approximation. Moreover, with $m$ being even, the Fock space approximation is positivity preserving. Indeed, as is evident from the top-right panels in \cref{fig:evo_stepanoff,fig:err_stepanoff}, the Fock space scheme leads to significant increase of reconstruction accuracy at $t=0$ compared to the classical approximation. While imprints from the eigenfunctions are still visible, the radial symmetry of the von Mises density is better represented, and the small-scale oscillations and negative values near the fixed point are eliminated.

\Cref{fig:scores}(left) shows RMSE and AC scores for the Stepanoff flow experiments, computed for the same prediction times $t$ as in \cref{fig:evo_stepanoff,fig:err_stepanoff}. As the evolution time $t$ increases, the error in $f^{(t)}_\text{Fock}$ remains consistently lower than $f^{(t)}_\text{cl}$ at least through $t \simeq 2.2$ and $t \simeq 3$ for the RMSE and AC metrics, respectively. At later times, the error in $f^{(t)}_\text{Fock}$ is somewhat higher than the classical approximation and appears to be predominantly concentrated in narrow bands where the true evolution $U^t f$ exhibits sharp gradients. As illustrated by the $t=4$ results in the bottom row of \cref{fig:err_stepanoff}, the maximal pointwise error in $f^{(t)}_\text{Fock}$ eventually outgrows that in $f^{(t)}_\text{cl}$.

\begin{figure}
    \centering
    \includegraphics[width=0.4\linewidth]{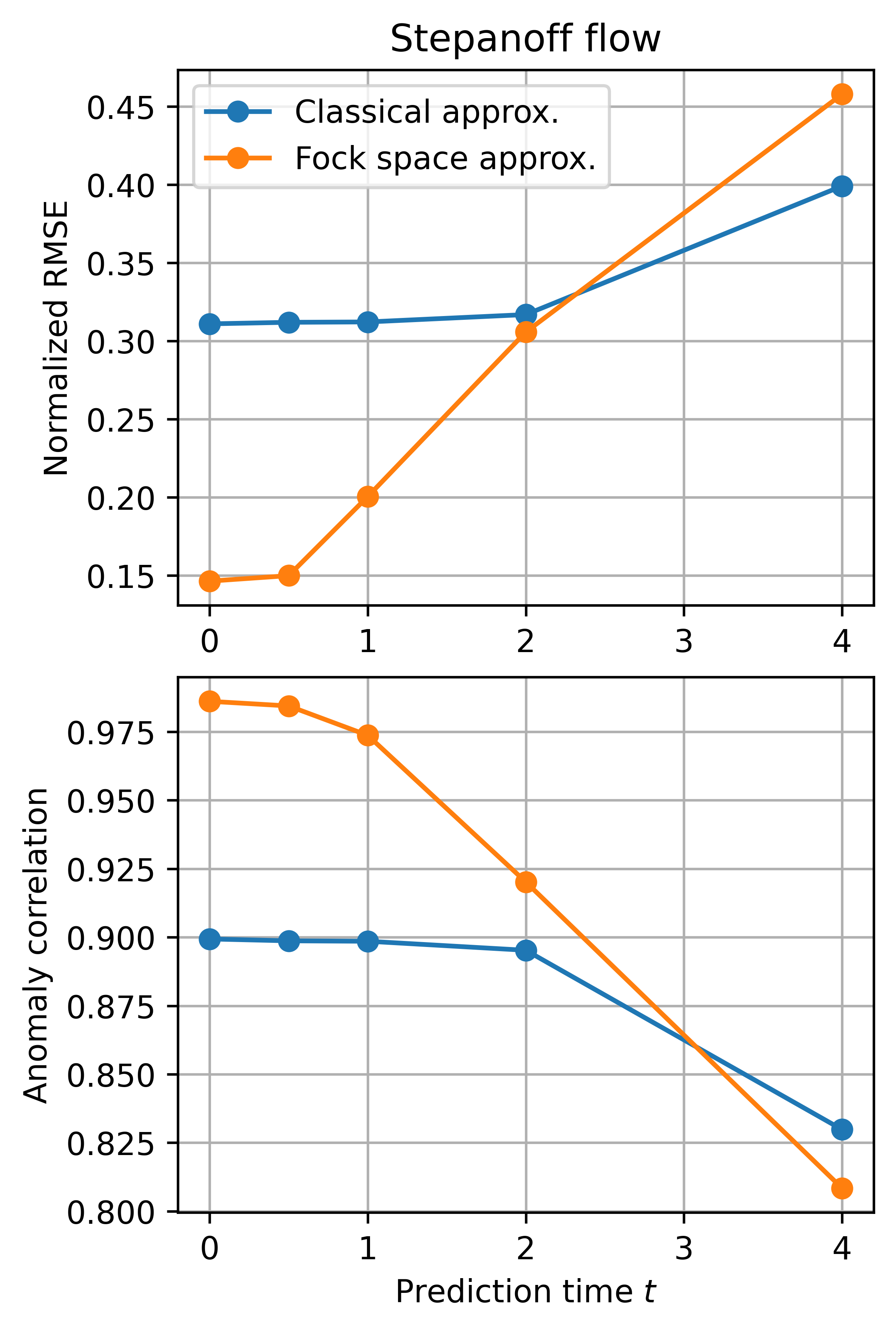}
    \includegraphics[width=0.4\linewidth]{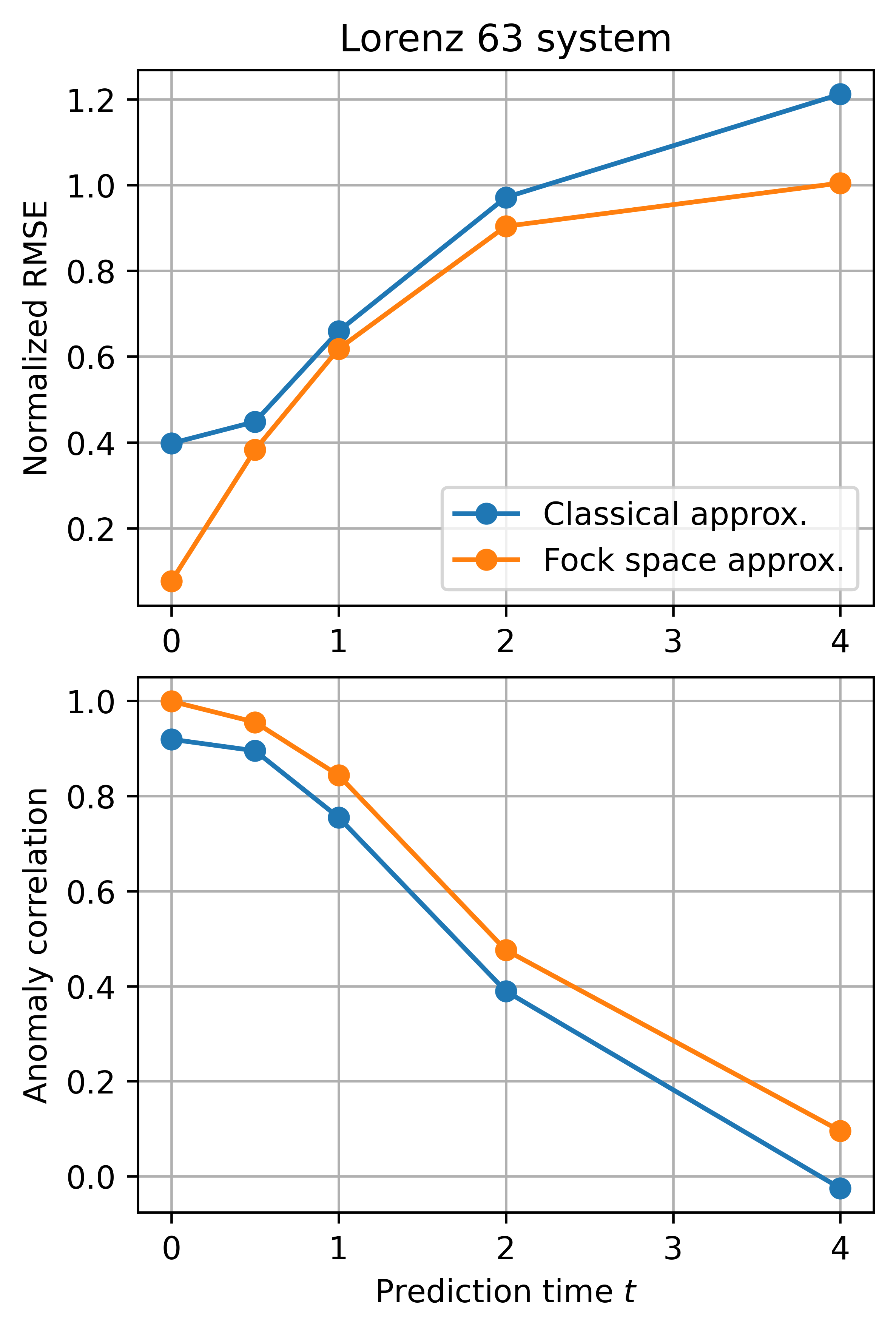}
    \caption{Normalized RMSE (top row) and anomaly correlation scores (bottom row) as a function of prediction time $t$ for the Stepanoff (left column) and L63 experiments (right column).}
    \label{fig:scores}
\end{figure}

\subsection{Lorenz 63 system}
\label{sec:l63}

The L63 system on $\mathbb R^3$ is generated by the vector field $\vec V \colon \mathbb R^3 \to T \mathbb R^3 \cong \mathbb R^3$, where
\begin{equation}
    \label{eq:vec_l63}
    \vec V(x) = (V^1(x), V^2(x), V^3(x)) = (-\sigma(x^2 - x^1), x^1(\rho - x^3) - x^2, x^1 x^2 - \beta x^3)
\end{equation}
and $x = (x^1, x^2, x^3)$. We use the standard parameter values $\beta = 8/3$, $\rho = 28$, and $\sigma = 10$, leading to the chaotic Lorenz ``butterfly'' attractor.

The left-hand column in \cref{fig:evo_l63} shows the evolution $U^t f$ of first state vector component~\eqref{eq:l63_x1} for evolution times $t \in \{ 0, 0.5, 1, 2, 4 \}$. Due to chaotic mixing between the two lobes of the attractor, as $t$ increases $U^t f$ exhibits increasingly finer-scale oscillations. We recall the approximate numerical value $\Lambda \simeq 0.91$ for the positive exponent of the standard L63 system \cite{Sprott03}, which implies that the pointwise predictability horizon of $f$ is of order $1/ \Lambda \simeq 1.10$ model time units.

\begin{figure}
    \centering
    \includegraphics[width=.95\linewidth]{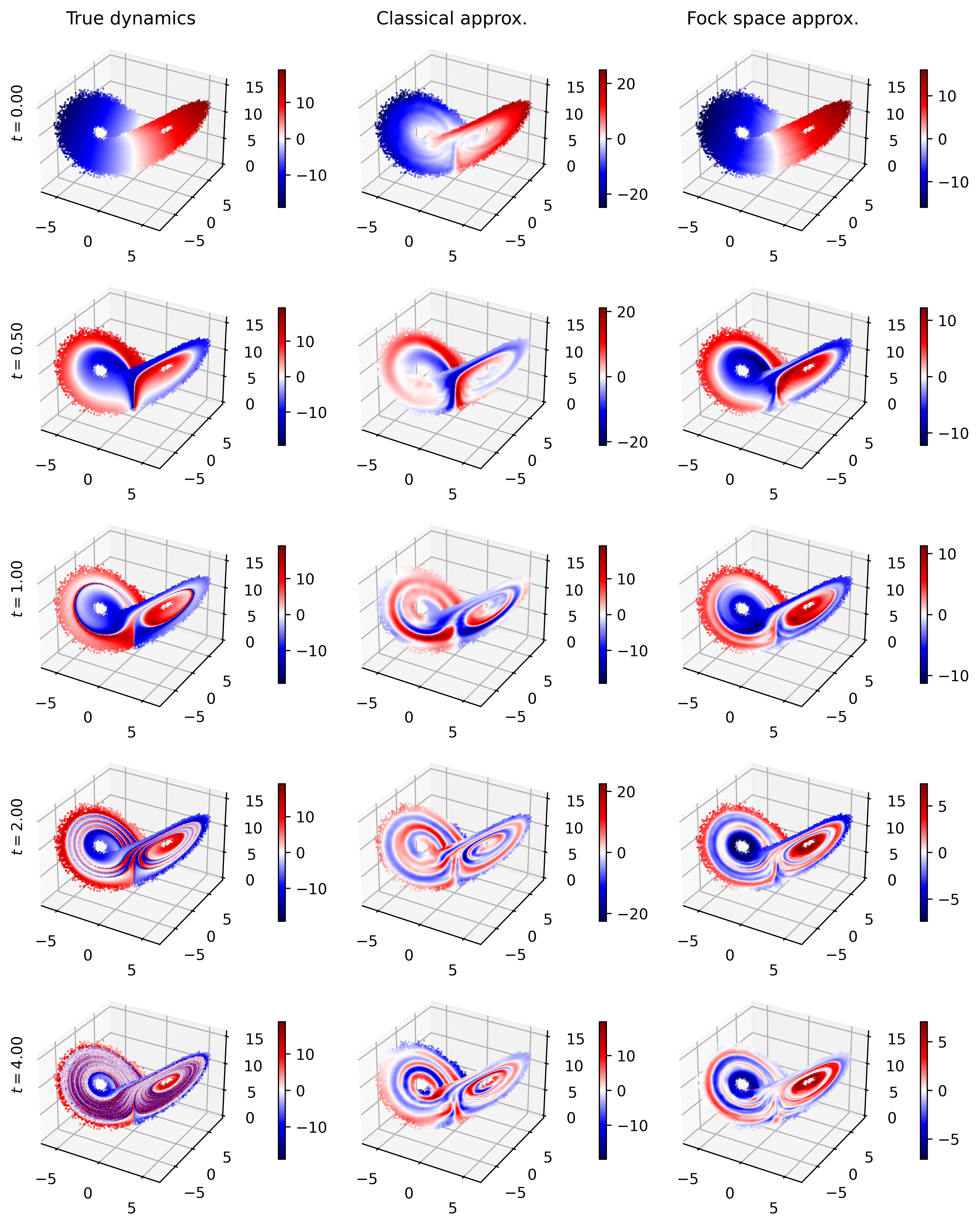}
    \caption{Time evolution of the $x$ component of the L63 state vector under the true dynamics (left column), a classical Koopman eigenfunction approximation using $2d+1$ eigenfunctions (center column), and a Fock space approximation for torus dimension $d=50$ and degree $m=4$. Rows from top to bottom show snapshots at the evolution times $t=0, 0.5, 1, 2, 4$, respectively.}
    \label{fig:evo_l63}
\end{figure}

The training dataset used in our L63 experiments consists of $N = \text{80,000}$ samples $x_n$, taken at a sampling interval of $\Delta t = 5.0$ model time units along a numerical L63 trajectory starting from an arbitrary initial condition $x_0$ near the Lorenz attractor. The sampling interval $\Delta t$ was chosen to be larger than the Lyapunov timescale $1/\Lambda$ to reduce correlations between the samples. This should aid in the convergence of the sampling measure $\mu_N$ to the invariant measure $\mu$. As noted in \cref{app:spectral_approx}, we are able to work with large sampling intervals by making use of known equations of motion to evaluate the action of the Koopman generator on functions. This is in contrast to, e.g., finite-difference methods that require small $\Delta t$ to yield accurate approximations of the generator \cites{Giannakis19,DasEtAl21}.

We compute approximate Koopman eigenfrequencies $\omega_{j,\tau,l,N}$ and eigenfunctions $\zeta_{j,\tau,l,N}$ similarly to the Stepanoff flow experiments, setting the regularization parameter and approximation space dimension to $\tau = 5 \times 10^{-7}$ and $l = 2048$, respectively. Eigenfrequency spectra and representative corresponding eigenfunctions are displayed in \cref{fig:spec,fig:evecs_l63}, respectively. Using the eigenpairs $(\omega_{j,\tau,l,N}, \zeta_{j,\tau,l,N})$ and the samples $f(x_n)$ of the prediction observable, we build classical and Fock space approximations $f^{(t)}_\text{cl}$ and $f^{(t)}_\text{Fock}$, respectively, again via a similar procedure to that used for the Stepanoff flow. The torus dimension, Fock space grading, regularization parameter, and kernel bandwidth parameters are $d = 50$, $m = 4$, $\sigma = 2 \times 10^{-6}$, and $\varepsilon = 0.1$, respectively. Scatterplots of $f^{(t)}_\text{cl}(x_n)$ and $f^{(t)}_\text{Fock}(x_n)$ on the training data are displayed in the center and right columns of \cref{fig:evo_l63}, respectively, and the corresponding errors relative to the true evolution $U^t f$ are shown in \cref{fig:err_l63}.

\begin{figure}
    \centering
    \includegraphics[width=\linewidth]{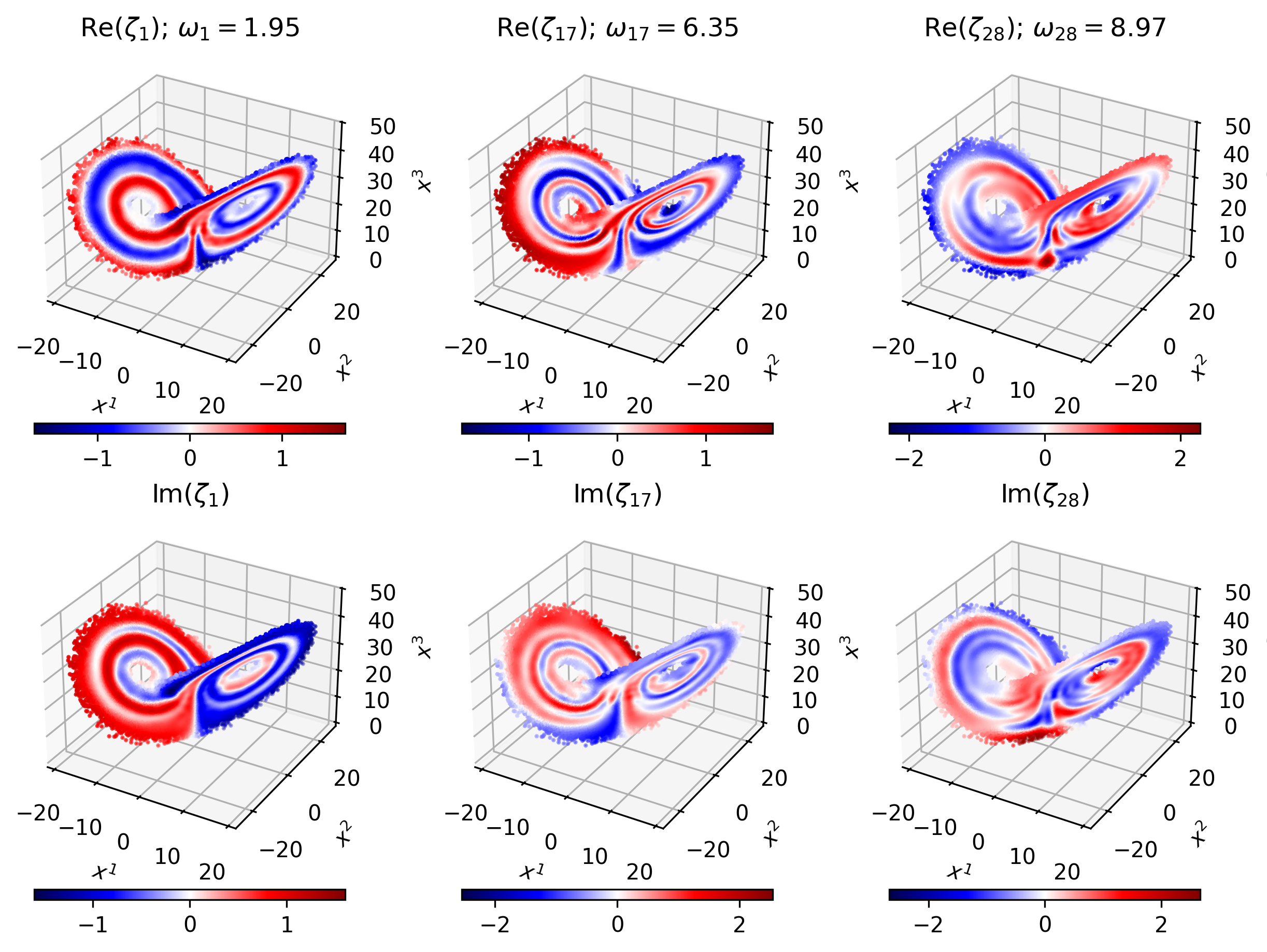}
    \caption{Real and imaginary parts of representative eigenfunctions $\zeta_{j,\tau,l,N}$ for the L63 system. The eigenfunctions shown have index $j=1, 5, 22$ with respect to the Dirichlet energy ordering, and are members of the complex-conjugate (nonconstant) eigenfunction pairs with the 1st, 2nd, and 6th largest projection amplitudes $\lvert \langle \xi_{j,\tau,l,N}, f \rangle_{H_N}\rvert$, respectively.}
    \label{fig:evecs_l63}
\end{figure}

\begin{figure}
    \centering
    \includegraphics[width=.95\linewidth]{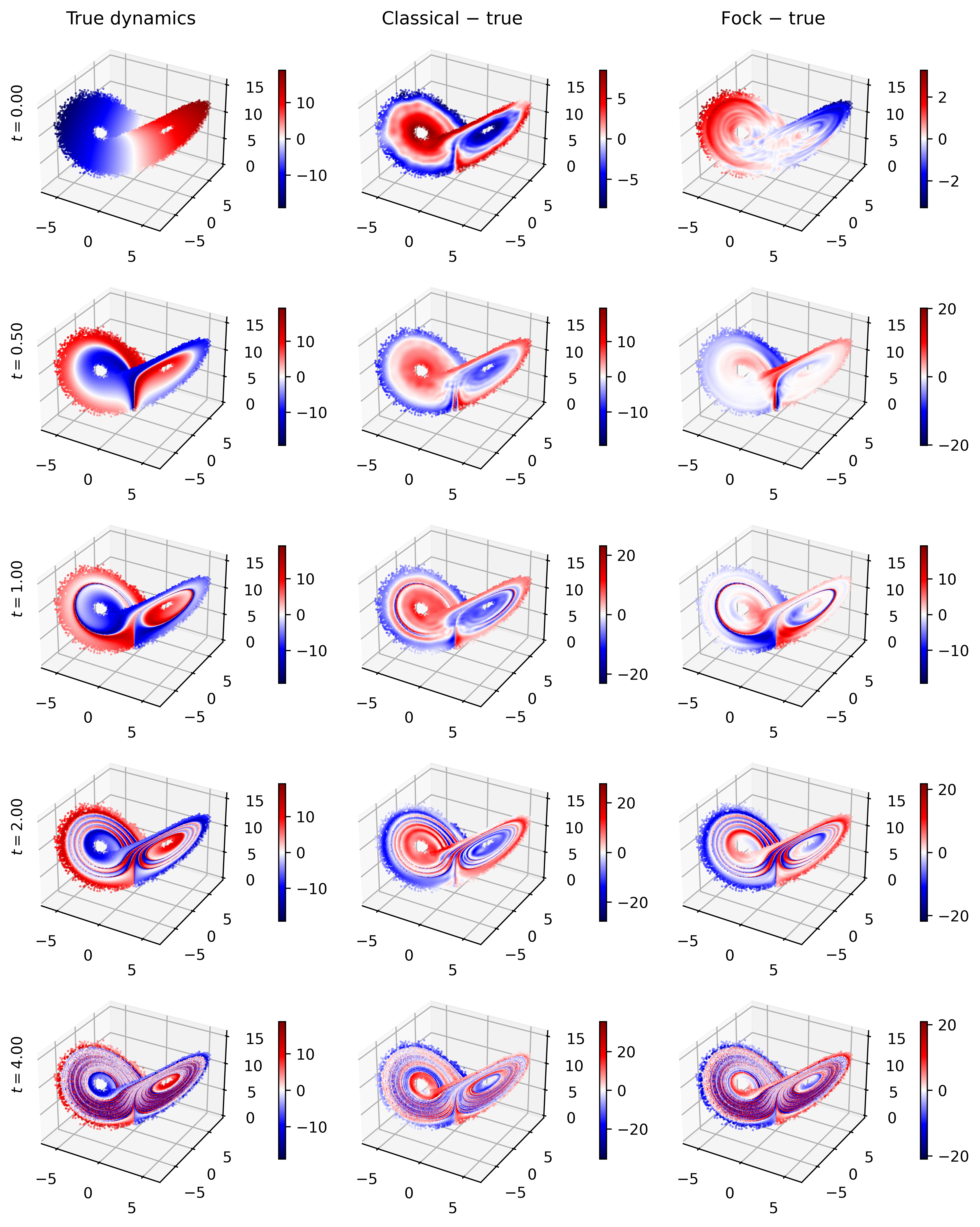}
    \caption{Errors in the classical and Fock space approximations from \cref{fig:evo_stepanoff} (center and right columns, respectively) relative to the true L63 evolution. The true evolution is plotted in the left column for reference.}
    \label{fig:err_l63}
\end{figure}

The general conclusions stemming from these results are broadly similar to what was previously observed for the Stepanoff flow. That is, amplification to the Fock space markedly improves prediction skill over a classical approximation utilizing the same number of eigenfunctions. At initialization time, $t=0$ (top rows in \cref{fig:evo_l63,fig:err_l63}), the Fock space reconstruction $f^{(0)}_\text{Fock}$ captures more accurately the linear dependence of $f$ on the $x^1$ component of the state vector than the classical reconstruction $f^{(0)}_\text{cl}$. At later times, $f^{(t)}_\text{Fock}$ provides a more accurate representation of the oscillations seen in the true evolution $U^tf$ along the radial direction in each lobe of the attractor due to chaotic mixing. Of course, since both methods utilize approximation spaces of fixed, finite dimension, they eventually fail to capture the small-scale spatial oscillatory behavior of $U^t f$ under the true dynamics. Still, even at $t=4$ (bottom rows in \cref{fig:evo_l63,fig:err_l63}), $f^{(t)}_\text{Fock}$ is arguably closer to the spatial structure of $U^t f$ than $f^{(t)}_\text{cl}$. \Cref{fig:scores}(right) shows the RMSE and AC scores for the same prediction times as in \cref{fig:evo_l63,fig:err_l63}. Over the examined time interval, the Fock space approximation outperforms the classical approximation with respect to both the RMSE and AC metrics, particularly at early times, $t \lesssim 0.5$.

\section{Discussion}
\label{sec:discussion}

Combining techniques from many-body quantum theory, RKHS theory, and Koopman operator theory, we have developed a framework for consistently approximating measure-preserving ergodic flows as infinite-dimensional topological rotation systems. A principal element of our approach is a family of weighted symmetric Fock spaces, $\fk$, generated from a 1-parameter family of RKHSs $\mathcal H_\tau$, $\tau > 0$, and endowed with coalgebra structure and commutative Banach algebra with respect to the symmetric tensor product. On the RKHSs $\mathcal H_\tau$, we build skew-adjoint, diagonalizable approximations $W_\tau$ of the Koopman generator on $L^2$ (which is generally a non-diagonalizable operator with non-trivial continuous spectrum). We dilate the unitary evolution groups generated by $W_\tau$ into infinite-dimensional rotation systems on the Banach algebra spectra of $\fk$, equipped with the weak-$^*$ topology making them compact Hausdorff spaces.

On the one hand, this construction realizes regularized approximations $e^{t W_\tau}$ of Koopman operators (which are not composition operators by flows on the original state space) as topological rotation systems on the spectrum of an algebra, with associated Koopman composition operators acting on continuous functions on the spectrum. This construction is therefore similar in spirit to the Halmos--von Neumann theorem that realizes a measure-preserving ergodic flow with pure point spectrum (i.e., diagonalizable generator) as a topological rotation system on the spectrum of an abelian $C^*$-algebra generated by the eigenspaces of the Koopman operator.

Simultaneously, the second quantization approach provides a flexible framework for approximating the Koopman evolution of observables of measure-preserving ergodic flows. The resulting approximations are positivity-preserving, amenable to data-driven numerical implementation, and consistent with the true Koopman evolution in appropriate asymptotic limits.

In more detail, we have shown that the spectrum $\sigma(\fk)$ can be partitioned into a collection of tori (of potentially infinite dimension) that are invariant under the rotational dynamics induced by $W_\tau$. Moreover, $\fk$ is isomorphic to a reproducing kernel Hilbert algebra (RKHA) $\hfk$ of continuous functions (i.e., an RKHS that is simultaneously a coalgebra and  a Banach algebra with respect to pointwise function multiplication) on appropriately chosen such tori. We developed a scheme for lifting continuous observables of the original dynamical system to polynomials of arbitrarily large degree $m$ in the Fourier bases of the spectral tori. A key aspect of this scheme is that it allows to systematically generate high-dimensional approximation spaces from tensor products of a given collection of approximate Koopman eigenfunctions (eigenfunctions of $W_\tau$), capturing information from the product structure of these eigenfunctions. Conventional subspace projection methods for Koopman operator can be viewed as $m=1$ versions of this scheme.

We numerically applied our approach to two examples of non-integrable measure-preserving ergodic systems: a Stepanoff flow on $\mathbb T^2$ as an example with smooth invariant measure and topological weak mixing (absence of continuous nonconstant Koopman eigenfunctions) and the L63 system on $\mathbb R^3$ as an example with invariant measure supported on a fractal attractor and measure-theoretic mixing. Both examples demonstrated improved prediction skill resulting from the Fock space construction over conventional subspace projection methods for Koopman operator approximation.

Previous work \cite{GiannakisEtAl24}, developed tensor network approximation techniques for Koopman and transfer operators that similarly yield algebraic amplification of regularized Koopman operators to Fock spaces. From a high-level standpoint the tensor network approach can be viewed as being ``dual'' to the second quantization scheme presented in this paper: the former is based on a lift of \emph{states} to the Fock space whereas the latter focuses on a lift of \emph{observables}. An advantageous aspect of the tensor network scheme is high computational efficiency when implemented classically. On the other hand, the second quantization scheme in this paper offers greater flexibility as it can be implemented for larger families of RKHSs $\mathcal H_\tau$, and it also yields a topological model that provides an arena for different types of approximation schemes (the polynomial lifts employed in this paper being just an example). Another noteworthy difference between the two methods is that the tensor network approach employs the full Fock space without algebra structure, whereas a centerpiece of the approach presented in this paper is a symmetric weighted Fock space with coalgebra and abelian Banach algebra structure.

In terms of future work, a natural next step would be to develop gate-based implementations of our scheme running on quantum computing platforms. One possibility in that direction would be to employ the quantum algorithm for pure point spectrum systems developed in \cite{GiannakisEtAl22} as a black box algorithm to simulate the evolution of the polynomial observables under the rotational dynamics on the tori within $\sigma(\fk)$. Essentially, this would involve defining an RKHA $\mathfrak A$ on each torus of interest, and using $\mathfrak A$ to build a quantum mechanical representation of the rotation system as described in \cref{sec:pure_point_spec}. A more direct (but likely more challenging) approach would be to build quantum computational algorithms on the RKHAs $\hfk$ without invoking auxiliary algebras such as $\mathfrak A$. Such an approach would hinge upon availability of efficient algorithms for preparing quantum states associated with the RKHA feature vectors $\varphi_{\sigma,\tau,d}(x)$ from \cref{sec:finite_dimensional_overview}. Efforts in these directions are currently underway. In a broader context, we hope that this work will stimulate further research at the interface of ergodic theory and quantum theory.

\appendix

\section*{Acknowledgments}

Dimitrios Giannakis acknowledges support from the U.S.\ Department of Defense, Basic Research Office under Vannevar Bush Faculty Fellowship grant N00014-21-1-2946, and the U.S.\ Office of Naval Research under MURI grant N00014-19-1-242, and the U.S.\ Department of Energy under grant DE-SC0025101. Mohammad Javad Latifi Jebelli and Michael Montgomery were supported as postdoctoral fellows from these grants. Philipp Pfeffer is supported by the project no. P2018-02-001 ``DeepTurb -- Deep Learning in and of Turbulence'' of the Carl Zeiss Foundation, Germany. J\"org Schumacher is supported by the European Union (ERC, MesoComp, 101052786). Views and opinions expressed are however those of the authors only and do not necessarily reflect those of the European Union or the European Research Council. Neither the European Union nor the granting authority can be held responsible for them.

\section*{Data availability statement}

Python code reproducing the numerical results in this paper is available in the repository \url{https://github.com/dg227/NLSA} under directory \url{Python/examples/fock_space}. Data obtained through this code will be made available on reasonable request.

\section{Markov semigroup construction}
\label{app:markov}

This appendix summarizes the construction of the 1-parameter family of Markov kernels $k_\tau \colon X \times X \to \mathbb R$ associated with the Markov semigroup $ \{ G_\tau \}_{\tau\geq 0}$ introduced in \cref{sec:spectral_approx}. Our exposition follows closely \cite{GiannakisValva24b}, where we refer the reader for further details.

The starting point of the construction is a kernel function $k \colon X \times X \to \mathbb R_{>0} $ that is continuous, symmetric, strictly positive-valued, and integrally strictly positive-definite. The latter condition means \cite{SriperumbudurEtAl11} that
\begin{equation}
    \label{eq:integrally_pd}
    \int_{X \times X} k(x, y) \, d(\nu \times \nu)(x, y) > 0
\end{equation}
for every finite Borel measure $\nu$. These conditions imply that $k$ is strictly positive definite. In addition, $K_\nu \colon L^2(\nu) \to L^2(\nu)$ defined as
\begin{displaymath}
    K_\nu f = \int_X k(\cdot, x) f(x) \, d\nu(x)
\end{displaymath}
is a self-adjoint, strictly positive integral operator of trace class that preserves positive functions. In this work, we are interested in integral operators $K_\nu $ associated with the invariant measure of the dynamical system, $\nu = \mu$, and the corresponding sampling measures on dynamical trajectories, $\nu = \mu_N$.

When $X$ can be continuously injected into $\mathbb R^d$ under a map $F \colon X \to \mathbb R^d$, it inherits a metric $d \colon X \times X \to \mathbb R$ from the Euclidean metric, $d(x, y) = \lVert F(x) - F(y)\rVert_2$. Then, a prototypical example satisfying~\eqref{eq:integrally_pd} is the Gaussian radial basis function (RBF) kernel,
\begin{equation}
    \label{eq:k_gauss}
    k_\varepsilon^\text{RBF}(x,y) = \exp\left(- \frac{d^2(x, y)}{\varepsilon^2}\right),
\end{equation}
where $\varepsilon > 0$ is a bandwidth parameter. Given a continuous, strictly positive function $\rho \colon X \to \mathbb R_{>0}$, a generalization of~\eqref{eq:k_gauss} is the so-called variable-bandwidth or self-tuning Gaussian kernel \cites{ZelnikManorPerona04,BerryHarlim16b},
\begin{equation}
    \label{eq:k_vb}
    k_\varepsilon^\text{vb}(x, y) = \exp\left(- \frac{d^2(x,y)}{\varepsilon^2 \rho(x)\rho(y)}\right).
\end{equation}
In the numerical experiments of \cref{sec:experiments}, we use~\eqref{eq:k_vb} with the bandwidth function
\begin{displaymath}
    \rho(x) = \left(\int_X k_{\tilde\varepsilon}(x, y) \, d\nu(y)\right)^{-1/m_\nu}.
\end{displaymath}
In the above, the bandwidth parameter $\tilde\varepsilon$ is independent from $\varepsilon$, and $m_\nu$ is a dimension parameter for the support of $\nu$. In the data-driven computations of \cref{sec:experiments}, we determine $\varepsilon$ and $\tilde \varepsilon$ automatically using a variant of a method proposed in \cite{CoifmanEtAl08}; see \cite{BerryEtAl15}*{Appendix~A} and \cite{Giannakis19}*{Algorithm~1} for further details and pseudocode. The same method produces a numerical estimate of $m_\nu$ that we use in our computations.

Intuitively, the role of $\rho$ in~\eqref{eq:k_vb} is to decrease (increase) the effective bandwidth of the kernel in regions of data space with high (low) concentration of the data distribution, $F_*\nu$. This is particularly important in the case of the L63 experiments where $F_*\mu$ associated with the invariant measure $\mu$ is singular relative to the ambient Lebesgue measure on $\mathbb R^3$.

\begin{rk}
    Bandwidth functions $\rho$ computed for the sampling measures $\mu_N$ are data-dependent, $\rho = \rho_N$, making the variable-bandwidth kernel $k_\varepsilon^\text{vb}$ also data-dependent. In such cases, weak-$^*$ convergence \eqref{eq:ergodic_av} of $\mu_N$ in the large-data limit ($N \to \infty$) implies uniform convergence of $\rho_N$ on the compact set $M$ to the bandwidth function $\rho$ computed for the invariant measure $\mu$ (e.g., \cite{VonLuxburgEtAl08}). This implies in turn uniform convergence of $k_\varepsilon^\text{vb}$ on $M \times M$, and the latter is sufficient to ensure convergence of our data-driven scheme in the large-data limit. For simplicity of exposition, in what follows we will not explicitly distinguish between data-dependent and data-independent kernels in our notation.
\end{rk}

Assume that $\nu$ is a probability measure with compact support. Using a variant of the bistochastic kernel formulation proposed in \cite{CoifmanHirn13}, we normalize the kernel $k$ to build a symmetric Markov kernel $p_\nu \colon X \times X \to \mathbb R_{>0}$, defined as
\begin{gather*}
    p_\nu(x,y) = \int_X \frac{k(x,z) k(z, y)}{d_\nu(x)q_\nu(z)d_\nu(y)} \, d\nu(z),
\end{gather*}
for $d_\nu = K_\nu 1_X$ and $q_\nu = K_\nu (1 / d_\nu)$. Clearly, $p_\nu(x, y) > 0$ and $p_\nu(x, y) = p_\nu(y, x)$, and the normalization $\int_X p_\nu(x,y) \, d\nu(y) = 1$ holds by construction. It can also be shown that $p_\nu$ is strictly positive definite, and thus has an associated RKHS $\mathcal H \subset C(X)$. It then follows that $G_\nu \colon L^2(\nu) \to L^2(\nu)$,
\begin{displaymath}
    G_\nu f = \int_X p_\nu(\cdot, x) f(x) \, d\nu(x),
\end{displaymath}
is a self-adjoint, strictly positive-definite, Markov operator of trace class. It can also be shown by strict positivity of $p_\nu$ that $G_\nu$ is ergodic; i.e., $G_\nu$ has a simple eigenvalue at 1 with corresponding eigenvector that is $\nu$-a.e.\ equal to 1.

Let us consider an eigendecomposition of $G_\nu$,
\begin{displaymath}
    G_\nu \phi_{j,\nu} = \lambda_{j,\nu} \phi_{j,\nu},
\end{displaymath}
where the eigenvalues $1 = \lambda_{0,\nu} > \lambda_{1,\nu} \geq \lambda_{2,\nu}$ are strictly positive and have finite multiplicities, and the corresponding eigenvectors form an orthonormal basis $ \{ \phi_{j,\nu} \}_j$ of $L^2(\nu)$ with $\phi_{0,\nu} = 1_X$. Each eigenvector $\phi_{j,\nu}$ has a continuous representative $\varphi_{j,\nu} \in C(X)$, where
\begin{equation}
    \label{eq:varphi}
    \varphi_{j,\nu} = \frac{1}{\lambda_{j,\nu}}\int_X p_\nu(\cdot, x) f(x) \, d\nu(x).
\end{equation}
Defining $\psi_{j,\nu} = \sqrt{\lambda_{j,\nu}} \phi_{j,\nu}$, we have that $\{ \psi_{j,\nu} \}$ is an orthonormal basis of $\mathcal H$. The RKHS $\mathcal H$ also induces a Dirichlet energy functional, $\mathcal E_\nu \colon D(\mathcal E_\nu) \to \mathbb R_{\geq 0}$, providing a notion of regularity of $L^2(\nu)$ elements in a dense domain $D(\mathcal E_\nu) \subseteq L^2(\nu)$ with representatives in $\mathcal H$. For every $f \in D(\mathcal E_\nu)$, we define
\begin{equation}
    \label{eq:dirichlet_energy}
    \mathcal E_\nu(f) = \sum_j \frac{\lvert \langle \phi_{j,\nu}, f \rangle_{L^2(\nu)}\rvert^2}{\lambda_j} - 1,
\end{equation}
and we have $\mathcal E_\nu(f) \geq 0$ with equality iff $f$ is constant $\nu$-a.e.

Next, define the self-adjoint operator $\Delta_\nu \colon D(\Delta_\nu) \to L^2(\nu)$, $D(\Delta_\nu) \subseteq L^2(\nu)$, via the eigendecomposition
\begin{displaymath}
    \Delta_\nu \phi_{j,\nu} = \eta_{j,\nu} \phi_{j,\nu}, \quad \eta_{j, \nu} = \frac{\lambda_{j,\nu}^{-1} - 1}{\lambda_{1,\nu}^{-1} - 1}.
\end{displaymath}
This operator is positive, it annihilates constant functions, and it is unbounded whenever $L^2(\nu)$ is infinite-dimensional. Moreover, $\Delta_\nu$ is normalized by convention so that its smallest nonzero eigenvalue $\eta_{1,\nu}$ is equal to 1. We interpret $\Delta_\nu$ as a Laplace-type operator that generates a Markov semigroup $\{G_{\tau,\nu} \}_{\tau \geq 0}$ on $L^2(\nu)$ with $G_{\tau,\nu} = e^{-\tau \Delta_\nu}$. The following result is adapted from \cite{DasEtAl21}*{Theorem~1}.

\begin{prop}
    \label{prop:markov_semigroup}
    With notation and assumptions as above, $\{G_{\tau,\nu} \}_{\tau \geq 0}$ is a strongly continuous semigroup of self-adjoint, strictly positive, Markov operators on $L^2(\nu)$. Moreover, the following hold for $\tau>0$ and any compact set $M \subseteq X$ containing the support of $\nu$:
    \begin{enumerate}[(i)]
        \item $G_{\tau,\nu}$ is a trace class integral operator induced by a continuous transition kernel $k_{\tau, \nu} \colon M \times M \to \mathbb R_{\geq 0}$,
            \begin{displaymath}
            G_{\tau,\nu} f = \int_X k_{\tau,\nu}(\cdot, x) f(x) \, d\nu(x), \quad k_{\tau,\nu}(x, y) = \sum_j \lambda_{j,\tau,\nu} \varphi_{j,\nu}(x) \varphi_{j,\nu}(y), \quad \lambda_{j,\tau,\nu} = e^{-\tau \eta_{j,\nu}},
            \end{displaymath}
            where the sum over $j$ in the expression for $k_{\tau,\nu}(x, y)$ converges uniformly for $(x, y)\in M \times M$.
        \item The kernel $k_{\tau,\nu}$ is strictly positive-definite, and has an associated RKHS $\mathcal H_{\tau,\nu} \subseteq C(M)$ whose restriction to $\supp(\nu)$ forms a dense subspace of $L^2(\nu)$.
        \item $\mathcal H_{\tau,\nu}$ is a subspace of $\mathcal H$ and a subspace of $\mathcal H_{\tau', \nu}$ for every $\tau' < \tau$.
        \item If the kernel $k\rvert_{M \times M}$ is $C^r$, then $\mathcal H_{\tau,\nu}$ is a subspace of $C^r(M)$.
    \end{enumerate}
\end{prop}

Note that under \cref{prop:markov_semigroup} $G_{\tau,\nu}$ admits the factorization $G_{\tau,\nu} = K_{\tau,\nu}^* K_{\tau,\nu}$, where $K_{\tau,\nu} \colon L^2(\nu) \to \mathcal H_{\tau,\nu}$ is the integral operator with (integral) kernel $k_{\tau,\nu}$ and dense range in $\mathcal H_{\tau,\nu}(\nu)$, and the adjoint $K_{\tau,\nu}^*$ implements the inclusion map (see \cref{sec:spectral_approx}). Moreover, $\{ \psi_{j,\tau,\nu} \}_j$ with
\begin{displaymath}
    \psi_{j,\tau,\nu} = \lambda_{j,\tau,\nu}^{-1/2} K_{\tau,\nu} \phi_{j,\nu} = \lambda_{j,\tau,\nu}^{1/2} \varphi_{j,\nu}
\end{displaymath}
forms an orthonormal basis of $\mathcal H_{\tau,\nu}(\nu)$.

Setting $\nu = \mu$ leads to the kernel $k_\tau \equiv k_{\tau,\mu}$, RKHS $\mathcal H_\tau \equiv \mathcal H_{\tau,\mu}$, and integral operators $K_\tau \equiv K_{\tau,\mu}$ and $G_\tau \equiv G_{\tau,\mu}$ used in the main text. \Crefrange{prty:k1}{prty:k4} are then satisfied as a consequence of \cref{prop:markov_semigroup}. In data-driven contexts we set $\nu = \mu_N$, leading to $k_{\tau,N} \equiv k_{\tau,\mu_N}$, $\mathcal H_{\tau,N} = \mathcal H_{\tau,\mu_N}$, $K_{\tau,N} \equiv K_{\tau,\mu_N}$, and $G_{\tau,N} \equiv G_{\tau,\mu_N}$. Spectral convergence of $G_{\tau,N}$ to $G_\tau$ as $N \to \infty$ occurs in the form of convergence of eigenvalues (including multiplicities), $\lim_{\tau\to 0^+}\lambda_{j,\tau,N} = \lambda_{j,\tau}$ where $\lambda_{j,\tau,N} \equiv \lambda_{j,\tau,\mu_N}$ and $\lambda_{j,\tau} \equiv \lambda_{j,\tau}$, and uniform convergence of the continuous representatives $\varphi_{j,\tau,N} \equiv \varphi_{j,\tau,\mu_N}$ to $\varphi_{j,\tau} \equiv \varphi_{j,\tau,\mu}$ on $M$ for appropriately chosen eigenvectors $\phi_{j,\tau} \equiv \phi_{j,\tau,\mu}$ and $\phi_{j,\tau,N} = \varphi_{j,\tau,\mu_N}$ of $G_\tau$ and $G_{\tau,N}$, respectively.

\section{Approximation of the identity by kernel integral operators}
\label{app:approx_id}

In \cref{sec:fock_consistency}, we use the following lemma to deduce that powers of continuous kernels satisfying \eqref{eq:kappa_decay} provide an $L^2$ approximation of the identity when acting on continuous functions.

\begin{lem}
    Let $X$ be a compact Hausdorff space and $\mu$ a Borel probability measure with full support on $X$. Let $\kappa\colon X \times X \to \mathbb R_{>0}$ be a continuous kernel such that $\kappa(x,x) > \kappa(x,y)$ for all $x,y \in X$ and $x \neq y$. For $m \in \mathbb N$, define $d_m(x) = \int_X \kappa(x,\cdot)^m \, d\mu$ and $p_m(x,y) = \frac{\kappa(x,y)^m}{d_m(x)}$. Then for $f \in C(X)$,
    $$\lim_{m \to \infty} \left\| f - \int_X p_m(\cdot,y) f(y) \, d\mu(y)\right\|_{L^2(\mu)} = 0.$$
\end{lem}
\begin{proof}
    Set $\kappa_x = \kappa(x,\cdot)$ and define the sets $E_{x,\epsilon} = \kappa_x^{-1}((\kappa(x,x)-\epsilon,\kappa(x,x)])$. First, $E_{x,\epsilon}$ provide a basis for the topology of $X$. That is, for every open neighborhood $U \ni x$, $\kappa(x,X\setminus U) \subset [0,\kappa(x,x))$ is compact hence there is an $\epsilon >0$ such that $x \in E_{x,\epsilon} \subset U$. We also have the estimates
\begin{align*}
    \int_{E_{x,\epsilon}} \kappa_x^m\, d\mu &\geq \mu(E_{x,\epsilon} \setminus E_{x,\epsilon/2}) (\kappa(x,x)-\epsilon)^m +\mu( E_{x,\epsilon/2}) (\kappa(x,x)-\epsilon/2)^m,\\
    \int_{X \backslash E_{x,\epsilon}} \kappa_x^m \, d\mu &\leq \mu(X\setminus E_{x,\epsilon})(\kappa(x,x)-\epsilon)^m.
\end{align*}
Therefore, for every $x \in X$ there is a constant $C_x$ depending on $x$ and an $m_{x,0} \in \mathbb N$ such that $d_m(x) \geq C_x (\kappa(x,x)-\epsilon/2)^m$ for $m \geq m_{x,0}$. As a result, for every $m > m_{x,0}$, we have
\begin{displaymath}
    p_m(x, \cdot)\rvert_{X\setminus E_{x,\epsilon}} \leq \mu(X \setminus E_{x,\epsilon}) \frac{(\kappa(x, x) -\epsilon)^m}{C_x (\kappa(x,x) - \epsilon/2)^m},
\end{displaymath}
so $p_m(x,\cdot)|_{X\setminus E_{x,\epsilon}}$ converges uniformly to zero as $m \to \infty$.

We will first show pointwise convergence. Fix $x \in X$, $\delta>0$, and pick $\epsilon >0$ such that $\sup_{y \in E_{x,\epsilon}} \lvert f(y) - f(x)\rvert < \delta$. Choose also $n \in \mathbb N$ such that $p_m(x,\cdot)|_{X\setminus E_{x,\epsilon}} < \delta$ for all $m \geq n$. Define $f_m = \int_X p_m(\cdot,y) f(y)\, d\mu(y)$ and observe that
$$\lvert f_m(x) - f(x) \rvert \leq \delta \|f\|_{C(X)} \mu(X\backslash E_{x,\epsilon}) + \left\lvert f(x) -\int_{E_{x,\epsilon}} p_m(x,y)f(y)\, d\mu(y)\right\rvert, \quad \forall m \geq n.$$
Set $p_{x,\epsilon,m} = \int_{E_{x,\epsilon}} p_m(x,\cdot) \, d\mu$. Since $\int_X p_m(x,\cdot) \, d\mu =1$ and $\int_{X\setminus E_{x,\epsilon}} p_m(x,\cdot) \, d\mu \leq \delta$, $1-p_{x,\epsilon,m} \leq \delta$. Therefore,
\begin{multline*}
    \left\lvert f(x) - \int_{E_{x,\epsilon}} p_m(x,y)f(y) \, d\mu(y)\right\rvert\\
    \begin{aligned}
        &= \left\lvert f(x) - \int_{E_{x,\epsilon}} p_m(x, y) (f(y) - f(x)) \, d\mu(y) - f(x) \int_{E_{x,\epsilon}} p_m(x, y) \, d\mu(y)\right\rvert \\
        &\leq \delta + \lvert f(x)\rvert (1 - p_{x,\epsilon,m}) \leq \delta + \lVert f\rVert_{C(X)} \delta,
    \end{aligned}
\end{multline*}
giving
\begin{displaymath}
    \lvert f_m(x) - f(x) \rvert \leq (2\|f\|_{C(X)} + 1) \delta.
\end{displaymath}
Pointwise convergence as $m \to \infty$ follows since $\delta$ was arbitrary,
\begin{displaymath}
    \lim_{m\to\infty} \lvert f_m(x) - f(x)\rvert = 0, \quad \forall x \in \supp(\mu).
\end{displaymath}
Finally, we conclude $L^2(\mu)$ convergence by the dominated convergence theorem as
$$|f_m(x)| \leq \int_X |f(y)|p_m(x,y) \, d\mu(y) \leq \|f\|_{C(X)}.$$
\end{proof}

\section{Spectral approximation of the generator}
\label{app:spectral_approx}

In this appendix, we summarize the spectral approximation scheme for the Koopman generator from the paper \cite{GiannakisValva24b}, used here in the numerical experiments in \cref{sec:experiments}.

Let $\tilde H = \{ f \in H: \int_X f \, d\mu = 0 \}$ be the subspace of $H$ consisting of zero-mean functions with respect to the invariant measure. By ergodicity, $\tilde H = \spn \{1_X \}^\perp$ has codimension 1. Define also $H_V = D(V) \cap \tilde H$, and equip this space with the ``graph'' inner product
\begin{displaymath}
    \langle f, g\rangle_{H_V} = \langle f, g\rangle_H + \langle V f, Vg\rangle_H.
\end{displaymath}
Since $V$ is a closed operator mapping $\spn\{1_X\}$ to itself, $(H_V, \langle \cdot, \cdot \rangle_{H_V})$ is a Hilbert space.

Next, for $z>0$, define the bounded, continuous, antisymmetric function $q_z \colon i \mathbb R \to i \mathbb R$, where
\begin{displaymath}
    q_z(i\omega) = \frac{i\omega}{z^2 + \omega^2}, \quad \ran q_z = i \left[ -\frac{1}{2z}, \frac{1}{2z}\right].
\end{displaymath}
Define also the bounded operator $Q_z = q_z(V\rvert_{\tilde H}) \in B(\tilde H)$ via the Borel functional calculus. Observe that $Q_z = R_z^* V R_z$, where $R_z = (z - V \rvert_{\tilde H})^{-1}$ is the resolvent of $V\rvert_{\tilde H} \colon \tilde H \to \tilde H$ at $z$.

For $\tau>0$ and $G_\tau \colon H \to H$ constructed as in \cref{sec:spectral_approx,app:markov}, define the compact, skew-adjoint operators
\begin{displaymath}
    Q_{z,\tau} = R_z^* V_\tau R_z, \quad V_\tau = G_{\tau/2} V G_{\tau/2}.
\end{displaymath}
Note that well-definition and compactness of $Q_{z,\tau}$ follows from the fact that $G_{\tau/2}$ is a Markov operator induced by a $C^1$ kernel $k_\tau$ (in particular, $\ran G_\tau \rvert_{\tilde H} \subset H_V$), making $V G_{\tau/2} \rvert_{\tilde H}$ a bounded operator. Since $G_\tau \to \Id$ as $\tau \to 0^+$, we view $Q_{z,\tau}$ as a compact approximation of $Q_z = q_z(V)$.

Define now the domain $\Omega_z = (-\infty, -z] \cup [z, \infty)$, and observe that $\tilde q_z := q_z \rvert_{\Omega_z}$ is invertible with (unbounded) inverse $\tilde{q}_z^{-1}\colon i[-(2z)^{-1}, (2z)^{-1}] \setminus \{0\} \to i \mathbb R$,
\begin{displaymath}
    \tilde{q}_z^{-1}(i\omega) = i \frac{1 + \sqrt{1 - 4z^2\omega^2}}{2\omega}.
\end{displaymath}
As $z \to 0^+$, the domain $\Omega_z$ increases to $i \mathbb R$, so we view $\tilde q_z^{-1} \circ q_z$ as an approximation of the identity. Correspondingly, we view $\tilde V_z := \tilde q_z^{-1} (q_z(V\rvert_{\tilde H})) \equiv \tilde q_z^{-1}(Q_z)$ as an unbounded, skew-adjoint approximation of $V\rvert_{\tilde H}$.

Let $\tilde V_{z,\tau} = b_z(Q_{z,\tau})$ where $b_z \colon i \mathbb R \to i \mathbb R$ is any continuous extension of $\tilde q_z^{-1}$. Using the orthogonal projection $\proj_{\tilde H} \colon H \to \tilde H$, extend $\tilde V_{z,\tau}$ to the skew-adjoint operator $V_{z,\tau} = \tilde V_{z,\tau} \proj_{\tilde H}$ with dense domain $D(V_{z,\tau}) \subset H$. In \cite{GiannakisValva24}*{Theorem 6} it is shown:

\begin{prop}
    $V_{z,\tau}$ is a skew-adjoint operator with compact resolvent that converges to $V$ in strong resolvent sense in the iterated limit of $z \to 0^+$ after $\tau \to 0^+$. In particular, \crefrange{prty:v1}{prty:v5} hold for $V_{z,\tau}$ in that limit.
\end{prop}

Let us now consider an eigendecomposition of the compact, skew-adjoint operator $Q_{z,\tau}$:
\begin{equation}
    \label{eq:q_z_tau_eigen}
    Q_{z,\tau} \xi_{j,z,\tau} = \beta_{j,z,\tau} \xi_{j,z,\tau},
\end{equation}
where $\beta_{j,z,\tau} \in i \mathbb R$ and the eigenvectors $\xi_{j,z,\tau}$ form an orthonormal basis of $\tilde H$. We index the eigenvalues by $j \in \mathbb Z \setminus \{ 0 \}$ so that $\beta_{-j,z,\tau} = \overline{\beta_{j,z,\tau}}$ and $\xi_{-j,z,\tau} = \overline{\xi_{j,z,\tau}}$. We can then build an eigendecomposition
\begin{equation}
    \label{eq:v_z_tau_eigen}
    V_{z,\tau} \xi_{j,z,\tau} = i \omega_{j,z,\tau} \xi_{j,z,\tau},
\end{equation}
with $i\omega_{j,z,\tau} = b_z(\beta_{j,z,\tau})$, and $\xi_{j,z,\tau}$ as above for $j \in \mathbb Z \setminus \{ 0 \} $, and $\omega_{0,z,\tau} = 0$, $\xi_{0,z,\tau} = 1_X$. The eigendecomposition~\eqref{eq:v_z_tau_eigen} completely characterizes $V_{z,\tau}$.

\begin{rk}
    The extension $b_z \supset \tilde q_z^{-1}$ is introduced since the spectrum of $Q_{z,\tau}$ is not guaranteed to lie within the domain of definition $i [-(2z)^{-1}, (2z)^{-1}]$ of $\tilde q_z^{-1}$ for any $\tau >0$. In practical applications, we have not observed any instances of (numerical approximations of) eigenvalues of $Q_{z,\tau}$ that do not lie in $i [-(2z)^{-1}, (2z)^{-1}]$, so we work with $\tilde q_z^{-1}$ without making an explicit choice of extension $b_z$.
\end{rk}

To compute the eigenpairs $(\beta_{j,z,\tau}, \xi_{j,z,\tau})$, \cite{GiannakisValva24b} formulates an associated variational eigenvalue problem to~\eqref{eq:q_z_tau_eigen}. Define the sesquilinear forms $A_\tau \colon H_V \times H_V \to \mathbb C$ and $B_z \colon H_V \times H_V \to \mathbb C$ such that
\begin{displaymath}
    A_\tau(u, v) = \langle G_{\tau/2} u, V G_{\tau/2} v \rangle_H, \quad B_z(u,v) = \langle (z- V) u, (z- V) v\rangle_H.
\end{displaymath}
It is shown that $(\beta_{j,z,\tau}, \xi_{j,z,\tau})$ solves~\eqref{eq:q_z_tau_eigen} for $\beta_{j,z,\tau} \neq 0$ iff $\xi_{j,z,\tau} = (z - V) v_{j,z,\tau}$ and $(\beta_{j,z,\tau}, v_{j,z,\tau})$ solves the following variational problem.

\begin{defn}[variational eigenvalue problem for $Q_{z,\tau}$]
    \label{def:variational_eigen}
    Find $\beta_{j,z,\tau} \in i\mathbb R$ and $v_{j,z,\tau} \in \tilde H \setminus \{ 0 \}$ such that
    \begin{displaymath}
        A_\tau(u, v_{j,z,\tau}) = \beta_{j,z,\tau} B_z(u, v_{j,z,\tau}), \quad \forall u \in H_V.
    \end{displaymath}
\end{defn}

The variational problem in \cref{def:variational_eigen} is ``physics-informed'', in the sense that the sesquilinear forms $A_\tau$ and $B_z$ can both be evaluated on elements $u = \iota f \in H_V$ and $v = \iota g \in H_V$ with continuous representatives $f, g \in C^1(M)$ using the generating vector field $\vec V \colon M \to T M$ of the dynamics. In particular, using~\eqref{eq:generator_c1}, we get
\begin{align}
    \label{eq:sesqui_a}
    A_\tau(u,v) &= \int_{X \times X \times X} k_\tau'(w, x) k_\tau'(w, y) u(x) v(y) \, d (\mu \times \mu \times \mu)(w, x, y), \\
    \nonumber
    B_z(u, v) &= \int_X \left((z- \vec V \cdot \nabla) f\right) \left((z- \vec V \cdot \nabla) g\right) \, d\mu,
\end{align}
where $k_\tau'(\cdot, y) = \vec V \cdot \nabla k(\cdot, x)$.
In addition, if $f,g \in \ran K_\tau$, i.e., $f = K_\tau r$ and $ g = K_\tau s$ for $r, s\in H$, we have
\begin{equation}
    \label{eq:sesqui_b}
    B_z(u,v) = \int_{X \times X \times X} (z k_\tau(w, x) - k_\tau'(w, x))(z k_\tau(w, y) - k_\tau'(w, y)) r(x) s(y) \, d(\mu \times \mu \times \mu)(w, x,y).
\end{equation}
Thus, in such cases, $A_\tau(u, v)$ and $B_z(u, v)$ can be computed by ``pushing'' the action of the generator $V$ to the directional derivatives $k'_\tau(\cdot, x)$ of the kernel sections $k_\tau(\cdot, x)$ with respect to the dynamical vector field $\vec V$.

The approach of \cite{GiannakisValva24b} utilizes automatic differentiation to compute $k'_\tau$ without discretization errors that affect finite-difference approximations of the generator (e.g., \cites{Giannakis19,DasEtAl21}). Using $E_l = \spn \{\phi_{1,\tau}, \ldots, \phi_{l,\tau} \} \subset H_V$ as Galerkin approximation spaces, approximate solutions $(\beta_{j,\tau,lj, v_{j,z,\tau,l}})$ are computed by restricting the trial and test functions in \cref{def:variational_eigen} to lie in $E_l$. Numerically, this is equivalent to solving a matrix generalized eigenvalue problem
\begin{displaymath}
    \bm A_\tau \bm c_{j,z,\tau,l} = \beta_{j,z,\tau,l} \bm B_z \bm c_{j, z, \tau, l}, \quad \beta_{j,z,\tau,l} \in i \mathbb R, \quad \bm c_{j,z,\tau,l} \in \mathbb C^l \setminus \{ 0 \}.
\end{displaymath}
Here $\bm A_\tau = [A_\tau(\phi_i, \phi_j)]_{i,j=1}^l$ and $\bm B_z = [B_z(\phi_i, \phi_j)]_{i,j=1}^l$ are $l\times l$ matrices whose elements are computed using~\eqref{eq:sesqui_a} and~\eqref{eq:sesqui_b}, respectively, in conjunction with the kernel integral representation~\eqref{eq:varphi} of the $C^1(M)$ representatives $\varphi_{j,\tau}$ of the basis functions. Moreover, the generalized eigenvector $\bm c_{j,z,\tau,l} = (c_{1j,z,\tau,l}, \ldots, c_{lj,z,\tau,l})$ contains the expansion coefficients of $v_{j,z,\tau,l}$ in the kernel eigenbasis,
\begin{displaymath}
    v_{j,z,\tau,l} = \sum_{i=1}^l c_{ij,z,\tau,l} \phi_i.
\end{displaymath}
Specific formulas for the matrix elements of $\bm A_\tau$ and $\bm B_z$ when using the Markov kernel construction from \cref{app:markov} can be found in \cite{GiannakisValva24b}*{section~4.3}. By convention, we order solutions in increasing order of Dirichlet energy~\eqref{eq:dirichlet_energy}, viz.
\begin{displaymath}
    \mathcal E(v_{j,z,\tau,l}) = \sum_{i=1}^l \frac{\lvert c_{ij,z,\tau,l}\rvert^2}{\lambda_j}.
\end{displaymath}

By results on Galerkin approximation of variational eigenvalue problems \cite{BabuskaOsborn91}, as $l \to \infty$ $\beta_{j,z,\tau,l}$ converges to $\beta_{j,z,\tau}$ and, for an appropriate choice of eigenvectors, $v_{j,z,\tau,l}$ converges to $v_{j,z,\tau}$ in the norm of $H_V$ for an appropriate choice of eigenvectors. The latter, implies convergence of $\xi_{j,z,\tau,l} := (z - V) v_{j,z,\tau,l}$ in the norm of $H$. Note that if $\beta_{j,z,\tau}$ is nonzero, $v_{j,z,\tau}$ lies in $\ran G_{\tau/2}$ and thus has a representative in $\mathcal H_\tau \subseteq \mathcal H$. Thus, the Dirichlet energy $\mathcal E(v_{j,z,\tau})$ is finite so long as $\beta_{j,z,\tau} \neq 0$.

The scheme also has a data-driven formulation wherein all kernel integral operators are replaced by their counterparts induced from the sampling measures $\mu_N$, as described in \cref{sec:data_driven_overview,app:markov}. These methods converge in the limit of large data, $N \to \infty$; see \cite{GiannakisValva24b}*{section~4.4} for further details.

% \bibliography{bibliography}

\end{document}